\documentclass[12pt]{article}
\title{Edge limits of  truncated 
circular
beta ensembles}
\date{}
\author{Yun Li and Benedek Valk\'o}

\usepackage{units}
\usepackage{amsmath, amsthm, amssymb,stackengine,bbm}
\usepackage{subcaption}
\usepackage{graphicx}
\usepackage{amsmath, enumerate, bm}
\usepackage{color}
\usepackage{multirow}
\newcommand{\xdownarrow}[1]{%
  {\left\downarrow\vbox to #1{}\right.\kern-\nulldelimiterspace}
}

\usepackage{hyperref}
    \newtheorem{theorem}{Theorem}
    \newtheorem{lemma}[theorem]{Lemma}
    \newtheorem{proposition}[theorem]{Proposition}
    \newtheorem{corollary}[theorem]{Corollary}
    \newtheorem{claim}[theorem]{Claim}
    \newtheorem{fact}[theorem]{Fact}
    \newtheorem{assumption}[theorem]{Assumption}

\theoremstyle{definition} 
    \newtheorem{definition}[theorem]{Definition}
    \newtheorem{remark}[theorem]{Remark}

\newcommand{\eps}{\varepsilon}
\newcommand{\Z}{{\mathbb Z}}
\newcommand{\R}{{\mathbb R}}
\newcommand{\D}{{\mathbb D}}

\newcommand{\C}{{\mathbb C}}
\newcommand{\HH}{{\mathbb H}}

\newcommand{\diag}{\operatorname{diag}}

\newcommand{\cA}{{\mathcal A}}

\newcommand{\cL}{{\mathcal L}}
\newcommand{\cU}{{\mathcal U}}
\newcommand{\cE}{{\mathcal E}}
\newcommand{\cM}{{\mathcal M}}
\newcommand{\cH}{{\mathcal H}}
\newcommand{\cC}{{\mathcal C}}
\newcommand{\cT}{{\mathcal T}}
\newcommand{\cP}{{\mathcal P}}

\newcommand{\cI}{{\mathcal I}}

\newcommand{\Sineop}{\mathtt{Sine}_{\beta}}
\newcommand{\Sineb}{\operatorname{Sine}_{\beta}}

\newcommand{\HPop}{\mathtt{HP}_{\beta,\delta}}

\newcommand{\Bessop}{\mathtt{Bess}_{\beta,a}}

\newcommand{\CJb}{\operatorname{CJ}}
\newcommand{\ROb}{\operatorname{RO}}

\newcommand{\Circ}{\mathsf{Circ}}
\newcommand{\CJ}{\mathsf{CJ}}
\newcommand{\RO}{\mathsf{RO}}

\newcommand{\rCJ}{\mathsf{CJ}^{[r]}}
\newcommand{\rRO}{\mathsf{RO}^{[r]}}

\newcommand{\Circop}{\mathtt{Circ}}
\newcommand{\CJop}{\mathtt{CJ}}
\newcommand{\ROop}{\mathtt{RO}}

\newcommand{\Dirop}{\mathtt{Dir}}
\newcommand{\ttr}{\mathfrak t}
\newcommand{\tr}{\operatorname{tr}}

\newcommand{\res}{\mathtt{r\,}}

\newcommand{\spec}{{\text{spec}}}
\newcommand{\tl}{\tilde}
\newcommand{\wtl}{\widetilde}

\newcommand{\uu}{\mathfrak{u}}
\newcommand{\mat}[4]{\left( \begin{array}{cc}
		#1 & #2  \\
		#3 & #4  \\
	\end{array} \right)}
\newcommand{\bin}[2]{\left (
	\begin{array} {c}
		#1 \\
		#2
	\end{array}
	\right )}

\newcommand{\bzeta}{\boldsymbol{\zeta}}

\newcommand{\ed}{\stackrel{d}{=}}

\newcommand{\trunc}[1]{#1^{\textstyle{\ulcorner}}}

\newcommand{\tCircc}{\mathsf{Circ}^{\textstyle{\ulcorner}}}

\newcommand{\rev}[1]{\wtl{#1}}

\newcommand{\aff}[1]{\overset{\text{\tiny${\hookleftarrow}$}}{#1}}

\begin{document}
\maketitle

\begin{abstract}We study the scaling limit
of the rank-one truncation of various  beta ensemble generalizations of classical unitary/orthogonal random matrices: the circular beta ensemble, the real orthogonal beta ensemble, and the circular Jacobi beta ensemble.  We derive the scaling limit of the normalized characteristic polynomials and the point process limit of the eigenvalues near the point 1. We also treat multiplicative rank one perturbations of our models. Our approach relies on a representation of  truncated beta ensembles given by Killip-Kozhan \cite{KK}, together with the random operator framework developed in \cite{BVBV_op, BVBV_19,BVBV_szeta} to study scaling limits of beta ensembles.
\end{abstract}

\section{Introduction}

For the classical  unitary and orthogonal random matrix ensembles the point process scaling limit of the eigenvalues is well understood. The eigenvalues are on the unit circle, and 
if one scales the eigenangles appropriately, one obtains a point process limit on the real line. More recently the scaling limit of the (normalized) characteristic polynomials of these classical ensembles has been derived and characterized as well, these limits lead to random entire functions where the zero set is given by the point process limit of the eigenvalues.  

If we remove the first row and first column of a unitary (or orthogonal) matrix then the resulting matrix  has eigenvalues within the unit disk.  It is natural to ask what one can say about the limits of the eigenvalues and the characteristic polynomial if one studies the truncated random matrices, and what connections can be shown between the limit objects of the original and the truncated models. Our main goal is to study these questions for beta-generalizations of classical random orthogonal and unitary ensembles.  We will also consider similar questions for multiplicative rank one perturbations of these models. Non-normal perturbations of classical ensembles have a rich history, see e.g.~the surveys \cite{FyodorovSommers2003} and \cite{Forrester_2023} and the references within.

\subsection{Haar unitary matrices and their truncations}

To start with a concrete example, we first consider the case of Haar unitary matrices. 
Let $M_n$ be an $n\times n$ uniformly chosen  unitary matrix. With probability one $M_n$ has $n$ distinct eigenvalues $e^{i \theta_k}, 1\le k\le n$, all on the unit circle. The joint eigenvalue density is given by 
\begin{equation}\label{eq:cue}
\frac{1}{Z_{n}}\prod_{1\le j<k\le n}|e^{i\theta_j}-e^{i\theta_k}|^2, \qquad \theta_j\in [-\pi,\pi),
\end{equation}
where $Z_n$ is an explicit normalizing constant (see e.g.~\cite{ForBook}). The distribution given by \eqref{eq:cue} is called the size $n$ \emph{circular unitary ensemble}. Because of the appearance of the squared Vandermonde determinant in the probability density, this ensemble is \emph{determinantal} (\cite{AGZ,HKPV}), all finite dimensional marginal densities can be expressed via determinants built from a fixed kernel function. (We will provide more detail on the results discussed within this section in the Appendix.) The point process scaling limits of finite determinantal ensembles can be derived by studying the corresponding scaling limits of the determinantal kernel. It is a classical result due to Gaudin, Mehta, Dyson \cite{AGZ, mehta} that if we scale the eigenangles of $M_n$ by $n$ then we get a translation invariant determinantal point process in the limit. We call this point process  the $\operatorname{Sine}_2$ process.

In a more recent result, Chhaibi, Najnudel and Nikeghbali \cite{CNN} studied the scaling limit of the (normalized) characteristic polynomial 
\[
p_n(z):=\frac{\det(I_n-z M_n^{-1})}{\det(I_n- M_n^{-1})}=\prod_{j=1}^n \frac{1-z e^{-i \theta_j}}{1- e^{-i \theta_j}}
\]
of the circular unitary ensemble.  They showed that under the  scaling of the Gaudin-Mehta-Dyson theorem  one obtains a random entire function $\boldsymbol{\zeta}$ (named the \emph{stochastic zeta function}) with zero set given by the $\operatorname{Sine}_2$ process.

For a square matrix $M$ we denote by $\trunc{M}$ the matrix obtained by removing the first row and column from $M$. Note that we can write $\trunc{M}$ as $\Pi^{\dag} M \Pi$ where $\Pi$ is the appropriate projection matrix, and $^{\dag}$ denotes the transpose.

Now consider the truncated version of a uniformly chosen $(n+1)\times(n+1)$ unitary matrix, i.e.~$\trunc{M_{n+1}}$.  With probability one this matrix  has eigenvalues in the open unit disk $\D=\{z\in \C: |z|<1\}$. The obtained random matrix has been studied in the physics literature because of its connection to chaotic scattering problems (see \cite{FyodorovSavin, KhoruzhenkoSommers} for further discussion and references). 
In \cite{ZS} {\. Z}yczkowski and Sommers proved   that the joint eigenvalue density   of $\trunc{M_{n+1}}$ (with respect to the Lebesgue measure in the unit disk) is given by
\begin{equation}\label{eq:tcue}
\frac{1}{\pi^n}\prod_{1\le j<k\le n}|z_j-z_k|^2 , \qquad z_j\in \D.
\end{equation}
We call this distribution the \emph{truncated circular unitary ensemble}. (Note that the \cite{ZS} provides a description for the eigenvalue distribution for general rank-$k$ truncation as well.) The squared Vandermonde term in \eqref{eq:tcue} indicates that this is also a determinantal point process.  By studying the determinantal kernel one can show that 
 the point process limit of the eigenvalues of $\trunc{M_{n+1}}$ \emph{without any additional scaling} leads to a determinantal point process limit in $\D$. 
We may call the resulting process the \emph{bulk scaling limit} of the truncated circular unitary ensemble.

 It is natural to ask if this point process can be connected to the zeroes of a `nice' random analytic function, since it is the scaling limit of the zeros of the characteristic polynomial of $\trunc{M_{n+1}}$. In Peres-Vir\'ag \cite{PV} it was shown that this is indeed the case, the bulk scaling limit of the eigenvalues of $\trunc{M_{n+1}}$  has the same distribution as the zero set of the so-called Gaussian analytic function.  

One can treat the eigenvalues of $\trunc{M_{n+1}}$  as a perturbation of the original eigenvalues of $M_{n+1}$. Because of this, it is natural to study the behavior of the eigenvalues of the truncated matrix under the scaling 
\begin{align}\label{eq:edgescaling}
    z\mapsto - n i\log z,
\end{align} 
since this corresponds to the scaling $e^{i \theta}\mapsto n \theta$ that takes the original (unit length) eigenvalues to the $\operatorname{Sine}_2$ process.\footnote{Throughout the paper we are considering the  branch of logarithm that is defined on $\C\setminus (-\infty,0]$ and satisfies $\log(1)=0$.} See Figure \ref{fig:tcue} for an illustration.
It was shown in \cite{ABKN} that under this scaling the kernel of the truncated circular unitary ensemble (and hence the ensemble itself) indeed has a limit. The limiting point process is determinantal, and it  is supported in the open upper half plane $\HH=\{z\in \C: \Im z>0\}$.
We call this the \emph{(hard) edge scaling limit} 
of the truncated model, since we zoom in near $z=1$. 

The point process obtained as the edge limit of the truncated circular ensembles limit process in \cite{ABKN}  appeared before  in \cite{FyodorovKhoruzhenko} and \cite{FyodorovSommers} as the point process limit of the rank-one additive anti-Hermitian perturbation for the Gaussian unitary ensemble under the appropriate scaling.

\begin{figure}[h]
	\begin{subfigure}[t]{0.38\linewidth}
		\includegraphics[width=\linewidth]{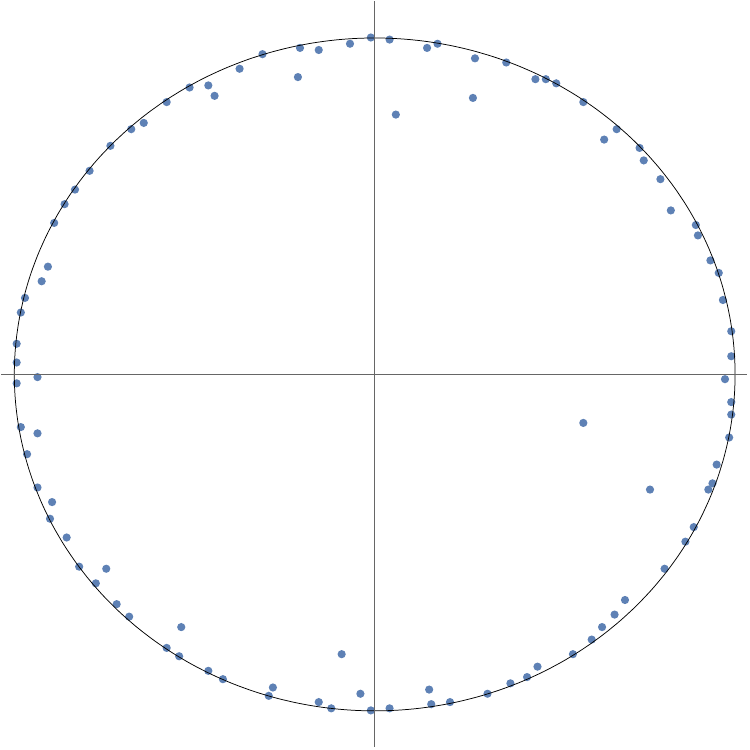}
	\end{subfigure}
	\hfill
	\begin{subfigure}[t]{0.55\linewidth}
		\includegraphics[width=\linewidth]{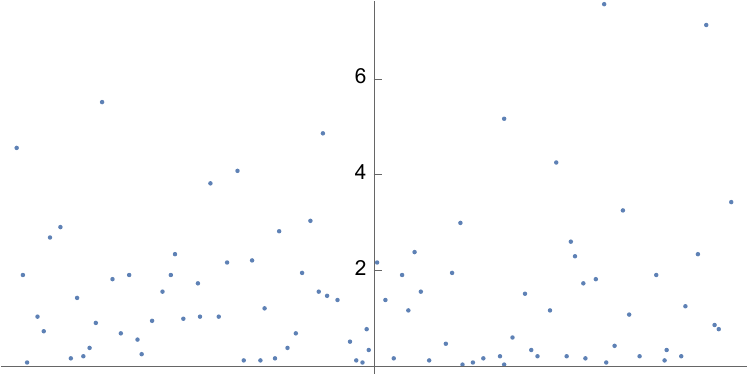}
	\end{subfigure}
	\caption{The picture on the left shows the eigenvalues of a truncated uniformly chosen $100\times 100$ unitary matrix.  The picture on the right shows the same eigenvalues under the edge scaling \eqref{eq:edgescaling}.}\label{fig:tcue}
\end{figure}

It is  natural to ask if one can connect the edge limit of the truncated circular ensemble to the zero set of a random analytic function, and whether one can characterize this random function in a natural way. We answer this question in the affirmative in our main result, Theorem \ref{thm:main} below. We provide a scaling limit for the normalized characteristic polynomial of the truncated model under the edge scaling, and describe the limiting random entire function. In fact, our goal is to study this and related questions in a more general setting: for the beta-generalizations of the circular unitary and other random unitary and orthogonal ensembles.

\subsection{CMV matrices, beta ensembles, and their truncations}

 The size $n$ circular beta ensemble with $\beta>0$ is the distribution of $n$ points $\{e^{i\theta_j}, 1\le j\le n\}$ on the unit circle with  joint probability density given by 
\begin{equation}\label{eq:cbe}
\frac{1}{Z_{n,\beta}}\prod_{j<k\le n}|e^{i\theta_j}-e^{i\theta_k}|^\beta, \qquad \theta_j\in [-\pi,\pi).
\end{equation}
Here $Z_{n,\beta}$ is  an explicit normalizing constant, see \cite{ForBook}. 
When $\beta=2$ we get the circular unitary ensemble. The cases when $\beta = 1$ and $4$ correspond to symmetric/self-dual random unitary matrices, but for general $\beta>0$  
there is no known invariant random matrix ensemble with the appropriate joint eigenvalue distribution. Note however that \eqref{eq:cbe} has a natural interpretation as the Gibbs measure corresponding to a log-gas of $n$-points restricted to the unit circle and interacting via a logarithmic potential. 

In \cite{KillipNenciu}  Killip and Nenciu (motivated by the results of \cite{DE}) constructed a family of sparse random unitary matrix models $\{\Circ_{n,\beta},n\ge 1\}$ with joint eigenvalue distribution given by \eqref{eq:cbe}.
Their construction is based on the theory of  orthogonal polynomials on the unit circle. 
We provide here a quick overview of their approach, the precise statements will be reviewed in Section \ref{sec:cmv}.

Suppose that  $\mu$ is a discrete probability  measure on the unit circle $\partial \D$ with a finite support of $n$ points. 
The probability measure $\mu$ can be encoded with its system of monic orthogonal polynomials. These polynomials satisfy the so-called Szeg\H{o} recursion, which can be parameterized with a finite collection of complex numbers $\alpha_0, \dots, \alpha_{n-1}$, called the Verblunsky coefficients. In \cite{CMV} Cantero, Moral, and Velasquez provided  a construction for a `canonical' sparse (five-diagonal) $n\times n$ unitary matrix (called the \emph{CMV matrix})
\[
\cC=\cC(\alpha_0,\dots, \alpha_{n-1})
\]
in terms of the Verblunski coefficients, so that the spectral measure  of $\cC$ with respect to the unit vector $\mathbf{e}_1=(1,0,\dots,0)^{\dag}$ is exactly $\mu$. 
Moreover, if the probability measure $\mu$ is the spectral measure of an $n\times n$ unitary matrix $U$ with respect to $\mathbf{e}_1 $ then the CMV matrix $\cC$ corresponding to $\mu$ is unitary equivalent to $U$.  
Note that the CMV matrix  is the analogue of the tridiagonal (Jacobi) matrix constructed from the coefficients of the three-term recursion of the orthogonal polynomials of a finitely supported probability measure on $\R$.

Let $M_n$ be an $n\times n$ Haar unitary matrix, 
and consider its spectral measure $\mu_n$ with respect to $\mathbf{e}_1$. This is a (random) probability measure with  support given by the circular unitary ensemble \eqref{eq:cue}. Using  unitary invariance one can show that the joint distribution of the weights of $\mu_n$ is given by a particular Dirichlet distribution, and that the weights are independent of the support of $\mu_n$. Moreover, the Verblunsky coefficients of $\mu_n$ are independent random variables, and their distributions can be computed explicitly. This motivated Killip and Nenciu in \cite{KillipNenciu} to study the random probability measure $\mu_{n,\beta}^{\textup{KN}}$
with support given by the circular beta ensemble \eqref{eq:cbe} and weights  
chosen independently from a particular ($\beta$-dependent) Dirichlet distribution. \cite{KillipNenciu} showed that the Verblunsky coefficients of $\mu_{n,\beta}^{\textup{KN}}$ are still independent, with explicitly given distributions. The corresponding CMV matrix $\Circ_{n,\beta}:=\cC$  provides a natural sparse random unitary matrix with spectrum given by the circular beta ensemble \eqref{eq:cbe}. For $\beta=2$ this matrix is unitary equivalent to the Haar unitary matrix  $M_n$, and their spectral measures with respect to $\mathbf{e}_1$ have the same distribution. 

In \cite{KK} Killip and Kozhan  studied how removing the first row and column changes the spectrum of  classical random unitary and orthogonal matrices.
An important observation of \cite{KK} (which is crucial for our paper as well) is the following: if $U$ is an $n\times n$ unitary matrix then the truncated matrix $\trunc{U}$ is unitary equivalent to the truncated version of the CMV matrix $\cC$ corresponding to  $U$, which in turn is unitary equivalent to an $(n-1)\times (n-1)$ CMV matrix built from a simple transformation of the Verblunsky coefficients of $U$. This means that if we know the Verblunsky coefficients of $U$ then we can construct a sparse matrix whose spectrum is the same as that of $\trunc{U}$. 

This observation allowed \cite{KK} to provide a sparse matrix model with spectrum distributed as \eqref{eq:tcue}. Their approach also allowed them to study the  matrix models $\Circ_{n,\beta}$ of \cite{KillipNenciu} with the first row and column removed.  They proved that the  joint eigenvalue density of the truncated matrix  $\Circ_{n+1,\beta}^{\textstyle{\ulcorner}} $  
is given by
\begin{equation}\label{eq:tcbe}
\frac{\beta^{n}}{(2\pi)^{n}}\prod_{1\le j,k\le n}(1-z_j\bar z_k)^{\frac{\beta}{2}-1}\prod_{j<k\le n}|z_j-z_k|^2, \qquad z_j\in \D.
\end{equation}
We call the resulting distribution the size $n$ \emph{truncated circular beta ensemble}. Note that for $\beta=2$ we recover \eqref{eq:tcue}.
(We remark that \cite{KK} also provided a log-gas interpretation for \eqref{eq:tcbe}.)
Our goal is to study this ensemble (together with some other related models) under the edge scaling \eqref{eq:edgescaling}.

The approach of Killip and Nenciu \cite{KillipNenciu} can be extended to provide random matrix representations of beta-generalizations of other random unitary and orthogonal ensembles where the joint distribution of the Verblunsky coefficients can be described explicitly.  The results of Killip and Kozhan \cite{KK} then provide a natural random matrix representation of the \emph{truncated} version of these beta ensembles. Our main results provide descriptions of the edge scaling limits of these truncated ensembles.

\subsection{Scaling limits of circular beta ensembles and their truncations}

Using the Killip-Nenciu representation Killip and Stoiciu in \cite{KS} showed that under the scaling \eqref{eq:edgescaling} the circular beta ensemble  has a point process limit. They characterized the limiting point process  via  its counting function using a system of stochastic differential equations. This limit process was later  shown to be the same as the $\Sineb$ process, the bulk scaling limit of the Gaussian beta ensemble (\cite{Nakano, BVBV_op}). Note that $\Sineb$ is not determinantal for general $\beta$, in fact there is no known description for its joint intensity functions in the general case.

In a series of papers \cite{BVBV_op, BVBV_19, BVBV_szeta}  Valk\'o and Vir\'ag developed a
framework to study the scaling limits of   beta ensembles using Dirac-type differential operators (see Section \ref{sec:Dirac} for a more detailed discussion).  
 \cite{BVBV_op} showed that the spectra of unitary CMV matrices and some of their point process limits (including the $\Sineb$ process) can be represented as the eigenvalues of random Dirac-type differential operators. A Dirac-type differential operator can be parametrized by a path in the upper half plane $\HH:=\{z: \Im z>0\}$ together with two boundary points in $\partial \HH=\R \cup \{\infty\}$. In the case of a unitary CMV matrix these parameters can be built from the Verblunsky coefficients. \cite{BVBV_19} showed how this representation can be used to prove operator level convergence of the circular beta ensemble to the $\Sineb$ process. The path parameter of the random differential operator corresponding to $\Circ_{n,\beta}$ is a random walk,  which under the appropriate scaling converges to a time-changed hyperbolic Brownian motion. This process is the path parameter of the random Dirac operator corresponding to the $\Sineb$ process.

\cite{BVBV_szeta} developed a framework to study scaling limits of normalized characteristic polynomials of beta ensembles.   In particular, \cite{BVBV_szeta} proved that the normalized and scaled characteristic polynomials of the $\Circ_{n,\beta}$ converge to a random entire function $\boldsymbol{\zeta}_\beta$ with zero set given by $\Sineb$. (For $\beta=2$ this random entire function is the stochastic zeta function introduced in \cite{CNN}.) The random function $\bzeta_\beta$ is characterized via various equivalent ways, in particular as the solution of the following random shooting problem.

\begin{theorem}[\cite{BVBV_szeta}]\label{thm:szeta}
    Let $b_1, b_2$ be independent two-sided standard Brownian motion, and $q$ an independent standard Cauchy random variable. 
Consider the unique strong solution  $\cH_\beta:(-\infty,0]\times\C\to\C^2$  of the stochastic differential equation 
\begin{equation}\label{eq:cH}
d\cH_\beta=\begin{pmatrix}
0 &-db_1 \\ 0 & db_2
\end{pmatrix}\cH_\beta-z\frac{\beta}{8}e^{\frac{\beta}{4}u}\begin{pmatrix}
0 & -1\\1&0
\end{pmatrix}\cH_\beta du,\quad u\le 0,
\end{equation}
subject to the initial condition $\lim_{u\to-\infty}\sup_{|z|<1}|\cH_\beta(u,z)-\binom{1}{0}| =0$. Then $\bzeta_\beta$ has the same distribution as the random function $\cH_\beta(0,z)^{\dag}\binom{1}{-q}$. 
\end{theorem}

Our main result gives the edge scaling limit of the  truncated circular beta ensemble \eqref{eq:tcbe} together with the scaling limit of its normalized characteristic polynomial. This result also provides a connection to the limit objects of the original  circular beta ensemble. 
\begin{theorem}\label{thm:main}
    Under the edge-scaling \eqref{eq:edgescaling} the truncated circular beta ensemble converges to a point process $\mathcal{X}_\beta$, which has the same distribution as the zero set of the random entire function $\cE_\beta=\cH_\beta(0,\cdot)^{\dag}\binom{1}{-i}$ defined via \eqref{eq:cH}. Moreover, $\cE_\beta$ is the scaling limit of the normalized characteristic polynomials of the truncated circular beta ensemble. 
\end{theorem}
Theorem \ref{thm:main} is proved in Section \ref{sub:truncated_circ}.  In fact, we will show that there is a coupling of the finite ensembles and the limiting object so that the stated limits  hold with probability one with effective (random) error bounds (see Proposition \ref{prop:tcirc_func}). 
The proof of the theorem uses the random operator framework to analyze the scaling limit of the normalized characteristic polynomial of truncated CMV matrices.

Theorem \ref{thm:main} provides a connection between the scaling limits of the full and the truncated circular beta ensemble that is new even in the classical $\beta=2$ case. 
It shows that the scaling limit of the characteristic polynomials of the circular beta ensemble can be obtained from the corresponding limit of the truncated model and an independent Cauchy random variable. Both $\bzeta_\beta$ and $\cE_\beta$ are random entire functions, and hence they are determined by their restriction to $\R$. By Theorems \ref{thm:szeta} and \ref{thm:main} we have the following equality in distribution:
\[
\{\bzeta_\beta(s): s\in \R\}\ed \{\Re \cE_\beta(s)+q \, \Im \cE_\beta(s) : s\in \R  \},
\]
where $q$ is a Cauchy random variable independent of $\cE_\beta$. \medskip

If $M$ is a square matrix then the spectrum of $\trunc{M}$ can also be studied by considering the spectrum of the rank one multiplicative perturbation
\[
M\cdot \diag(0,1,1,\dots,1)
\] 
instead. (Of course, this also adds an additional zero eigenvalue.)
This motivates the study of general rank one  multiplicative perturbations. For $r\in \R$  define  
\begin{align}\label{def:mult_pert}
 M^{[r]}:=M\cdot \diag(r,1,1,\dots,1).
\end{align}
If $M$ is a Haar unitary matrix  then the distribution of $M^{[r]}$ has been studied in 
\cite{Fyodorov2001}. (See also \cite{FyodorovKhoruzhenko} and \cite{FyodorovSommers} for related results on  rank-one additive anti-Hermitian perturbations for the Gaussian unitary ensemble.) Note that because of the various symmetries of the model, we may assume $r\in [0,1]$. 

Following the definition of the truncated circular beta ensemble, it is natural to define  the appropriate rank-one multiplicative perturbation of the circular beta ensemble as the spectrum of $\Circ_{n,\beta}^{[r]}$. In \cite{KK} Killip and Kozhan derived the joint eigenvalue distribution of $\Circ_{n,\beta}^{[r]}$. Moreover, they showed that if $\cC$ is a unitary CMV matrix then  the spectrum of $\cC^{[r]}$ is the same as a certain explicitly determined CMV matrix.  Using their results we are able to extend the results of Theorem \ref{thm:main}.

\begin{theorem}\label{thm:trunc_circ}
Fix $r\in [0,1]$. Consider the random function $\cH_\beta$ defined via \eqref{eq:cH}, and let $q$ be a standard Cauchy random variable independent of $b_1, b_2$ appearing in \eqref{eq:cH}. Under the edge-scaling \eqref{eq:edgescaling} the eigenvalues of $\Circ_{n,\beta}^{[r]}$ converge to a point process $\mathcal{X}_{r,\beta}$, which has the same distribution as the zero set of the random entire function $\cE_{r,\beta}=\cH_\beta(0,\cdot)^{\dag}\binom{1}{-c_r}$ 
with
\begin{align}\label{eq:cr}
c_r=\frac{q+i \frac{1-r}{1+r}}{1-i q \frac{1-r}{1+r}}. 
\end{align}
Moreover, $\cE_{r,\beta}$ is the limit of the normalized characteristic polynomials of $\Circ_{n,\beta}^{[r]}$ under the same scaling.
\end{theorem}
Note that $c_r=q$ for $r=1$ and $c_r=i$ for $r=0$, so this result gives an interpolation between the scaling limits of the unperturbed and the truncated circular beta ensemble.

Our approach extends to other matrix models as well. We provide versions of Theorems \ref{thm:main} and \ref{thm:trunc_circ} for the \emph{real orthogonal beta ensemble} and the \emph{circular Jacobi beta ensemble}. 

The real orthogonal beta ensemble was introduced in \cite{KillipNenciu} (see also \cite{KK}) as a generalization of the joint eigenvalue distributions of a certain class of the classical compact random matrix models. (See Section \ref{sub:finite} for more detail.) The operator level limit of the real orthogonal beta ensemble in the  hard-edge limit was derived in \cite{LV}.   The real orthogonal beta ensemble can be naturally transformed into another classical model, the (real) Jacobi beta ensemble, whose edge scaling limits were studied in \cite{HM2012}. The truncated version of the real orthogonal beta ensemble was introduced in \cite{KK}, where the authors constructed a sparse matrix model and derived the joint eigenvalue distribution. In Theorem \ref{thm:RO_main} and Corollary \ref{cor:RO_perturb} of Section \ref{sub:truncated_RO}, we will  establish the scaling limit of the truncated (and perturbed) real orthogonal beta ensemble, together with the scaling limit of its normalized characteristic polynomial.

The circular Jacobi beta ensemble is a one-parameter extension  of the circular beta ensemble. For a complex parameter $\delta$ with $\Re \delta>-1/2$ it is given by the joint density function 
\begin{align}
\frac{1}{Z_{n,\beta,\delta}^{\CJb}}\prod_{j<k\le n} \left|e^{i \theta_j}-e^{i \theta_k}\right|^\beta  \prod_{k=1}^n (1-e^{- i \theta_k})^{ \delta}(1-e^{ i \theta_k})^{ \bar\delta} \label{eq:CJ_pdf}
\end{align}
with respect to the uniform measure on the unit circle. For $\delta=0$ this is just the circular beta ensemble. When $\delta=\beta \tfrac{k}{2}$ with a positive integer $k$ then this model can  be viewed as the circular beta ensemble conditioned to have $k$ particles at $e^{i \theta}=1$. (See Section \ref{sub:finite} for additional details.)
In \cite{BNR2009} the authors constructed a family of unitary matrix models whose eigenvalues are distributed according to \eqref{eq:CJ_pdf}. Following the Killip-Nenciu approach they studied a random probability measure where the support is given by  the circular Jacobi beta ensemble, and the weights are given by an independently chosen beta-dependent Dirichlet distribution. \cite{BNR2009} showed that although the Verblunsky coefficients for this measure are usually not independent, a modified version of these coefficients are in fact independent, and their distributions can be described explicitly. \cite{LV} used this representation to derive the point process and operator level scaling limit of the circular Jacobi beta ensemble. We build on these results together with the ideas of \cite{KK} to define the truncated (and the perturbed) version of the circular Jacobi beta ensemble (see Section \ref{sub:CJ_finite}). We show that the joint density of the truncated circular beta ensemble is given by a constant multiple of 
\[
\prod_{j,k=1}^n(1-z_j\bar z_k)^{\frac{\beta}{2}-1}\prod_{j<k}|z_k-z_j|^2\prod_{j=1}^n\left((1-z_j)^{\bar\delta}(1-\bar z_j)^\delta\right),
\]
generalizing  \eqref{eq:tcbe}. We then derive the point process scaling limit of these new models, together with the scaling limit of their normalized characteristic polynomial. See Theorem \ref{thm:CJ_main} and Corollary \ref{cor:CJ_perturb} of Section \ref{sub:CJ_limit} for the precise statements.\medskip

\subsection{Outline of the paper}

\noindent In Section \ref{sec:Dirac} we review the required background for the considered Dirac-type differential operators. In Section \ref{sec:cmv} we give an overview of  how finitely supported probability measures on the unit circle can be represented with CMV matrices and Dirac-type differential operators. Section \ref{sec:finite} describes the considered finite beta ensembles and their operator level limits, and summarizes the known results on scaling limits of characteristic polynomials of these models.  
In Section \ref{sec:general} we provide general results regarding the convergence of the eigenvalues truncated and perturbed CMV matrices. Section \ref{sec:trCbE_trRbE} provides the proofs for Theorems \ref{thm:main} and \ref{thm:trunc_circ} on the scaling limits of the truncated and the perturbed circular beta ensemble, and proves the corresponding results for the real orthogonal beta ensemble. Section \ref{sec:trCJbE} proves our results on the truncated and the perturbed circular Jacobi beta ensemble. Section \ref{sec:app} is an appendix that contains a more detailed discussion of the determinantal point processes mentioned in the Introduction, together with a few open problems. 
\bigskip

\noindent {\bf Acknowledgments.}  We thank B\'alint Vir\'ag  for helpful comments and references. 
B.V~was partially supported by  the University of Wisconsin – Madison Office of the Vice Chancellor for Research and Graduate Education with funding from the Wisconsin Alumni Research Foundation and by the National Science Foundation award DMS-2246435. Y.L.~is partially supported by the Shuimu Tsinghua Scholar Program and China International Postdoctoral Exchange Fellowship Program No.YJ20220279.

\section{Dirac-type differential operators}

\label{sec:Dirac}

This section provides a brief overview of Dirac-type operators. A more detailed discussion can be found in \cite{BVBV_op}, \cite{BVBV_szeta}, and \cite{Weidmann}.

\subsection{Basics of Dirac operators}\label{sub:Dirac_basics}

Let $\cI$ be  $[0,1)$ or $(0,1]$. Suppose $x+iy:\cI\mapsto\HH=\{z\in\C:\Im z>0\}$ is a locally bounded measurable function, we define 
\begin{equation}\label{def:R}
R=\frac{X^{\dag} X}{2\det X},\qquad X=\begin{pmatrix}
1 & -x\\0&y
\end{pmatrix}
\end{equation}
to be the positive definite matrix-valued function that encodes $x+iy$. We consider  differential operators of the form 
\begin{equation}\label{def:Dirop}
\tau: f \to R^{-1}(t)Jf', \qquad f:\cI\mapsto \R^2, \qquad J=\begin{pmatrix}0&-1\\1&0\end{pmatrix}.
\end{equation}
We call  $\tau$ a \emph{Dirac-type operator}, $x+iy$ the \emph{generating path} of $\tau$, and $R(t)$ the \emph{weight function} of $\tau$.

The boundary conditions for $\tau$ at $t=0,1$ are given by two non-zero, non-parallel vectors $\uu_0,\uu_1\in\R^2$. We  assume that these vectors are normalized and satisfy certain integrability conditions with respect to the  weight function.
\begin{assumption}\label{assumption:u} We assume that
	\begin{equation}\label{eq:u_normalization}
	\uu_0^{\dag} J \uu_1=1,
	\end{equation}
	\begin{equation}\label{HS_finite}
	\int_0^{1} \int_0^t \mathfrak  u_0^{\dag} R(s) \mathfrak  u_0   \, \mathfrak  u_1^{\dag} R(t) \mathfrak  u_1 ds dt<\infty,
	\end{equation}
and
	\begin{align*}
	\int_0^1  \|R(s) \uu_1\| ds<\infty  \text { if $\cI=[0,1)$, \quad and\quad } \int_0^1  \|\uu_0^{\dag}R(s)\| ds<\infty \text{ if $\cI=(0,1]$.}
	\end{align*}
\end{assumption}
Let $L^2_R$ denote the $L^2$ space of functions $f:\cI\mapsto \R^2$ with the $L^2$-norm 
\[
\|f\|_R^2:=\int_\cI f(s)^{\dag} R(s)f(s) ds.
\]
Under Assumption \ref{assumption:u}, the operator $\tau$ defined according to \eqref{def:R} and \eqref{def:Dirop}   is self-adjoint on the following domain: 
\begin{equation}\label{Dir:domain}
\mathrm{dom}(\tau)=\{v\in L^2_R\cap\text{AC}:  \tau v\in L^2_R, \lim_{s\to 0}v(s)^{\dag} J\uu_0=0,\lim_{s\to 1} v(s)^{\dag}J\uu_1=0\}.
\end{equation}
Here $\text{AC}(\cI)$ is the set of absolutely continuous real functions on $\cI$. 

Throughout the paper, we use the notations $\Dirop(R,\uu_0,\uu_1)$ or $\Dirop(x+i y, \uu_0, \uu_1)$ interchangeably for the operator $\tau$ defined via \eqref{def:Dirop} and \eqref{def:R} on the domain \eqref{Dir:domain}. 
We can identify a nonzero vector  in $\R^2$ with a boundary point in $\partial \HH=\R\cup\{\infty\}$ using the projection operator
\begin{equation}\label{eq:P}
\mathcal{P}\binom{x_1}{x_2}=\begin{cases}
x_1/x_2 &\mbox{if $x_2\ne 0$,}\\
\infty &\mbox{if $x_2=0$.}
\end{cases}
\end{equation}
Using this identification the boundary conditions can be parametrized by two points in $\partial \HH$.

Under Assumption \ref{assumption:u}, the operator $\tau$ is invertible, and $\tau^{-1}$ is a Hilbert-Schmidt integral operator.
\begin{proposition}\label{prop:inv_HS}
	Suppose that $\tau={\tt Dir}(R,\uu_0,\uu_1)$ satisfies Assumption \ref{assumption:u}. Then 
	\begin{equation*}
		\|\tau^{-1}\|_{\textup{HS}}^2= 2	\int_0^{1} \int_0^t \mathfrak  u_0  ^{\dag} R(s) \mathfrak  u_0   \, \mathfrak  u_1^{\dag} R(t) \mathfrak  u_1 ds dt<\infty, 
	\end{equation*}
	and $\tau^{-1}$ is a Hilbert-Schmidt integral operator on $L^2_R$  given by 
	\begin{equation}\label{eq:HS_kernel}
	\tau^{-1}f(s) = \int_\cI K_{\tau^{-1}}(s,t)f(t)dt,\quad K_{\tau^{-1}}(s,t) = \big( \uu_0\uu_1^{\dag}1_{s<t}+\uu_1\uu_0^{\dag}1_{s\geq t}\big) R(t).
	\end{equation} 
	
	Consider the conjugated operator $X\tau X^{-1}$, where $X$ is defined as in \eqref{eq:P}. We denote the inverse of this operator by 
 $\res\tau:=\left(X\tau X^{-1}\right)^{-1}$, this is an integral operator on $L^2(\cI)$ 
	with kernel
	\begin{equation}\label{Dir:inverse}
	K_{\res \tau}(s,t) = \tfrac12\big( \mathfrak{a}_0(s)\mathfrak{a}_1(t)^{\dag}\mathbf{1}_{s<t}+\mathfrak{a}_1(s)\mathfrak{a}_0(t)^{\dag}\mathbf{1}_{s\geq t}\big), \qquad \mathfrak{a}_j(s)=\frac{X(s)\uu_j}{\sqrt{\det X}}, \quad j=0,1.
	\end{equation}
	The conjugated operator $X\tau X^{-1}$ has the same spectrum as the operator $\tau$, and $\|\res\tau\|_{\textup{HS}} = \|\tau^{-1}\|_{\textup{HS}}$.
\end{proposition}

Proposition \ref{prop:inv_HS} and the fact that $\tau$ is self-adjoint implies that $\tau$ has a discrete spectrum, with countably many eigenvalues that are all nonzero real numbers, and can only accumulate near $\pm \infty$. We label the eigenvalues of $\tau$ in increasing order by $\lambda_k, k\in \Z$ so that $\lambda_{-1}<0<\lambda_0$.
 
Let us remark that if $\tau={\tt Dir}(R,\uu_0,\uu_1)$ satisfies Assumptions \ref{assumption:u}, then for any $\sigma\in(0,1)$, the operator $\tau$ restricted in $\cI\cap [0,\sigma]$ is well-defined with boundary conditions $\uu_0,\uu_1$ at endpoints $t=0,\sigma$ and the restricted weight function $R|_{t\in\cI\cap [0,\sigma]}$. We denote the restricted operator and its resolvent as $\tau_{\sigma},\res \tau_{\sigma}$ respectively.

\subsection{Canonical systems and the secular functions}\label{sub:Dirac_secular}

In order to study scaling limits of characteristic polynomials, \cite{BVBV_szeta} introduced the secular function and structure function of a Dirac-type operator. We briefly review the required definitions and results.

Suppose that $\cI=(0,1]$ and consider the eigenvalue equation of a Dirac operator $\tau H=zH$ as a canonical system, i.e.,
	\begin{equation*}
J\frac{d}{dt}H(t,z)= zR(t)H(t,z).
\end{equation*}
\cite{BVBV_szeta} showed that this  system has a unique solution under our assumptions if we set the initial condition to be $\uu_0$. 
\begin{proposition} \label{prop:canonical_H}
	Suppose that $\cI=(0,1]$ and $\tau={\tt Dir}(R,\uu_0,\uu_1)$ satisfies Assumptions \ref{assumption:u}. Then there exists a unique vector valued function $H:\cI\times \C\mapsto\C^2$  so that for every $z\in\C$, the function $H(\cdot,z)$ solves the following ordinary differential equation
	\begin{equation}\label{eq:canonical}
	J\frac{d}{dt}H(t,z)= zR(t)H(t,z),\qquad t\in\cI,\qquad  \lim_{t\to 0}H(t,z) = \uu_0.
	\end{equation}
	Write $H=\binom{A}{B}$. For any $t\in\cI$, the function $H(t,z)$ satisfies $\|H(t,z)\|\ge 0$, and its components $A(t,\cdot),B(t,\cdot)$ are analytic functions such that $A(t,x),B(t,x)\in\R$ for $x\in\R$. 
\end{proposition}
\begin{definition}\label{def:RAF}
	Under the assumptions of Proposition \ref{prop:canonical_H}, we define the secular function of $\tau$ as 
 	\begin{equation}\label{def:secular}
 	\zeta_{\tau}(\cdot) =  H(1,\cdot)^{\dag} J\uu_1,
 	\end{equation}
 	and the structure function of $\tau$ as 
 	\begin{equation}\label{def:structure}
 	\cE_\tau(\cdot) = H(1,\cdot)^{\dag} \binom{1}{-i} = A(1,\cdot)-iB(1,\cdot). 
 	\end{equation}
\end{definition}
We define the integral trace of $\res \tau$ as the  integral of the trace of the kernel $K_{\res \tau}$, and denote it by $\mathfrak{t}_\tau$:
\begin{align}
\mathfrak{t}_\tau=\int_0^{1} \tr K_{\ttr_\tau}(s,s) ds=\tfrac12 \int_0^{1} \mathfrak{a}_0(s)^{\dag} \mathfrak{a}_1(s) ds=\int_0^1 \uu_0^{\dag} R(s) \uu_1 ds.
\label{eq:int_tr}
\end{align} 
It was proved in \cite{BVBV_szeta} that the secular function $\zeta_\tau$ can be represented as 
\begin{align}
\zeta_\tau(z)=e^{-z{\mathfrak t}_\tau}{\det}_2(I-z \,\res \tau) =e^{-\tfrac{z}{2} \int_0^{1} \mathfrak{a}_0(s)^{\dag} \mathfrak{a}_1(s) ds} \prod_k (1-z/ \,\lambda_k)e^{z/ \lambda_k},\label{eq:zeta}
\end{align}
where ${\det}_2$ is the second regularized determinant, see \cite{Simon_det}. Note that the integral trace is finite under Assumption \ref{assumption:u}, and the secular function $\zeta_\tau$ is an entire function with zero set given by $\spec(\tau)$, the spectrum of $\tau$. We refer to \cite{BVBV_szeta} for additional details. The secular function of $\tau$ can be viewed as a generalization of the normalized characteristic polynomial of a  matrix.

The next proposition provides comparisons for the solutions of two canonical systems of the form \eqref{eq:canonical}. It also shows that $H(t,\cdot)$ is continuous on compacts in $z\in \C$ at $t=0$.

\begin{proposition}[Proposition 12 of \cite{BVBV_palm}]\label{prop:H_comparison}
	Suppose that $\cI=(0,1]$ and $\tau=\Dirop(R,\uu_0,\uu_1)$,
 $\tilde \tau=\Dirop(\tl R,\uu_0,\tilde\uu_1)$ satisfy Assumption \ref{assumption:u}. Let $H,\tl H$ be the solutions of the corresponding canonical systems \eqref{eq:canonical}, and define $\mathfrak{a}_0, \tilde{\mathfrak{a}}_0$ according to \eqref{Dir:inverse}. Recall that $\tau_t$ is the operator $\tau$ restricted to $(0,t]$. Then there is an absolute constant $c>1$ depending only on  $\uu_0$ so that for all $t\in \cI$, $z\in \C$ we have
	\begin{align}\notag
	&    |H(t,z)-\tilde H(t,z)|\le \left(c^{|z|\left(|\ttr_{\tau_t}-\tilde \ttr_{\tau_t}|+\|\res \tau_t-\res \tilde\tau_t\|+\sqrt{\int_{0}^t|\mathfrak{a}_0(s)-\tilde{\mathfrak{a}}_0(s)|^2 ds \int_{0}^t(|\mathfrak{a}_0(s)|^2+|\tilde{\mathfrak{a}}_0(s)|^2)ds}\right)}-1\right)\\&\hskip150pt \times
	c^{\left(|z|(|\ttr_{\tau_t}|+|\ttr_{\tilde\tau_t}|+\|\res \tau_t\|+\|\res \tilde\tau_t\|+\int_0^t \left(|\mathfrak{a}_0(s)|^2+|\tilde{\mathfrak{a}}_0(s)|^2 \right)ds)+1\right)^2}.\label{eq:H1_cont}
	\end{align}
	Moreover, we have for all $t\in \cI$, $z\in \C$ 
	\begin{align} \label{eq:H_cont}
	|H(t,z)-\uu_0|\le \left(c^{|z|(|\ttr_{\tau_t}|+\|\res \tau_t\|+\int_0^t |\mathfrak{a}_0(s)|^2 ds)}-1\right) c^{\left(|z|(|\ttr_{\tau_t}|+\|\res \tau_t\|+\int_0^t |\mathfrak{a}_0(s)|^2 ds)+1\right)^2}.
	\end{align}
\end{proposition}
 Proposition \ref{prop:H_comparison} together with the Hoffman-Wielandt inequality in infinite dimensions (see e.g.~\cite{BhatiaElsner}) provide similar comparisons  for the secular functions and the spectra of two Dirac operators.

\begin{remark}\label{rmk:Dir_comp}
 Let $\tau_1 $, $\tau_2$ be two Dirac operators on $\cI$ satisfying Assumption \ref{assumption:u}. Denote by $\lambda_{k,i}, \zeta_i, \res\tau_i, \mathfrak t_{\tau_i}$ the eigenvalues,  secular function, resolvent and integral trace of $\tau_i$.  Then we have
\begin{align}\label{eq:ev_comp}
\sum_k \left|\lambda_{k,1}^{-1}-\lambda_{k,2}^{-1}\right|^2\le& \|\res{\tau_1}-\res{\tau_2}\|_{\textup{HS}}^2,
\end{align}
and there is an absolute constant $c>1$ so that for all $z\in \C$
\begin{align}\label{eq:zeta_comp}
|\zeta_1(z)-\zeta_2(z)|\le& \Big(c^{|z| |\mathfrak t_{\tau_1}-\mathfrak t_{\tau_2}|}-1+|z|\big\|\res \tau_1-\res\tau_2\big\|\Big)
c^{|z|^2(\|\res{\tau_1}\|^2+\|\res{\tau_2}\|^2)+|z|(|\mathfrak{t}_{\tau_1}|+|\mathfrak{t}_{\tau_2}|)+1}.
\end{align}
\end{remark}

\subsubsection*{Transformations of Dirac-type operators}

We finish this section with a short discussion on simple transformations of Dirac-type  operators. First note that the two cases $\cI=(0,1]$ and $\cI=[0,1)$ can be connected by a time reversal transformation $\rho$ on functions defined on $(0,1]$ or $[0,1)$ by $\rho f(t)=f(1-t)$. Let $\mathfrak{r}: \HH\to \HH$ be the reflection with respect to the imaginary axis defined by $\mathfrak{r}(x+i y)=-x +i y$.
The next statement provides a description of the effect of the composition of the time reversal $\rho$ and the reflection $\mathfrak{r}$, see \cite{BVBV_szeta} or \cite{BVBV_palm}.

\begin{lemma}[\cite{BVBV_szeta}]\label{lem:timerev}
Suppose that the Dirac operator $\tau=\Dirop(R,\uu_0, \uu_1)$ satisfies Assumption \ref{assumption:u}. Set $S=\diag(1,-1)$, then the operator $\rho^{-1}S \tau S \rho$ also satisfies Assumption  \ref{assumption:u} with boundary conditions $S\uu_1, S\uu_0$,  weight function $\rho SRS$, and generating path $\mathfrak{r} \rho z$. The operators $\tau$ and $\rho^{-1}S \tau S \rho$ are orthogonally equivalent  in the respective $L^2$ spaces,  and they have the same integral traces and secular functions. 
\end{lemma}

Consider the projection operator $\mathcal{P}$ defined in \eqref{eq:P}. It can be naturally generalized to the projection operator on  non-zero two-dimensional complex vectors with $\mathcal{P}\binom{z_1}{z_2}=z_1/z_2$ given $z_2\neq 0$. In this way, a $2\times 2$ non-singular matrix $A$ can be identified with a linear fractional  transformation  via $z\to \mathcal P A \binom{z}{1}$. When $A$ is real, the corresponding linear fractional transformation is an isometry of $\HH$. The next lemma describes the effect of a hyperbolic isometry on a Dirac operator.

\begin{lemma}[\cite{BVBV_szeta}]\label{lem:rotation}
	Let $Q$ be a $2\times 2$ orthogonal matrix with determinant 1. Let $\mathcal Q: \overline \HH\to \overline \HH$ be the corresponding linear isometry of $\overline \HH$ mapping $z\in \overline \HH$ to the ratio of the entries of $Q  \binom{z}{1}$. 	Suppose that the Dirac operator $\tau=\Dirop(R,\uu_0, \uu_1)$ satisfies Assumption \ref{assumption:u}. Then the operator $Q \tau Q^{-1}$ also satisfies Assumption \ref{assumption:u}, with boundary conditions $\mathcal Q \uu_0, \mathcal Q \uu_1$ and generating path $\mathcal Q(x+i y)$. The two operators are orthogonally equivalent, they have the same integral traces and secular functions. 
\end{lemma}

\section{Finitely supported measures on the unit circle}\label{sec:cmv}

In this section we provide an overview of the CMV construction for  finitely supported probability measures on the unit circle, and review how  Dirac-type operators can be used to study them.

\subsection{CMV matrices}\label{sub:CMV}

We briefly review some of the needed facts from the theory of orthogonal polynomials on the unit circle, together with some properties of CMV matrices. See \cite{OPUC1} for a comprehensive treatment of the subject, or \cite{SimonOPUC1foot} for a shorter summary. 

Let $\nu$ be a probability measure on $\partial \D$ with a support of $n$ distinct points $e^{i\lambda_j}, 1\le j \le n$. Let $\Phi_k, k=0, \dots, n$ be the Gram-Schmidt orthogonalization of the polynomials $1,z, \cdots, z^n$ with respect to $\nu$. Then $\Phi_k, k=0, \dots, n-1$ are the monic orthogonal polynomials with respect to $\nu$, and $\Phi_n(z):=\prod_{j=1}^n(z-e^{i \lambda_j})$. Together with the reversed polynomials  
\begin{align*}
\Phi_k^*(z):=z^k \overline{\Phi_k(1/\bar z)},
\end{align*}
these polynomials satisfy the famous Szeg\H o recursion (see e.g.~Section 1.5, vol.~1 of \cite{OPUC1}):
\begin{align}\label{eq:Szego1}
\binom{\Phi_{k+1}}{\Phi_{k+1}^*}=
\mat{1}{-\bar \alpha_k}{-\alpha_k}{1} \mat{z}{0}{0}{1} \binom{\Phi_{k}}{\Phi_{k}^*}, \qquad \binom{\Phi_0}{\Phi_0^*}
=\binom{1}{1},\qquad 0\le k\le n-1.
\end{align}
The constants $\alpha_k$, $0\le k\le n-1$ are called the Verblunsky coefficients, they satisfy $|\alpha_k|<1$ for $0\le k\le n-2$ and $|\alpha_{n-1}|=1$. The map between the probability measures supported on $n$ points on $\partial \D$ and the possible Verblunsky coefficients $(\alpha_0, \dots, \alpha_{n-1})\in \D^{n-1}\times \partial \D$ is invertible, and both the map and its inverse are continuous. (See e.g.~Theorem 1.7.11 of \cite{OPUC1}.)

The next definition introduces the CMV matrix.

 \begin{definition}\label{def:CMV}
For fixed $n\ge 1$, let $\{\alpha_k,0\le k\le n-1\}$ be a sequence of complex coefficients with  $|\alpha_k|\le 1$. Define
\[
\Xi_k = \begin{pmatrix}
\bar\alpha_k & \rho_k\\ \rho_k& -\alpha_k
\end{pmatrix},\qquad \rho_k=\sqrt{1-|\alpha_k|^2},\qquad 0\le k\le n-2
\]
and set $\Xi_{-1}=(1)$ and $\Xi_{n-1}=(\bar\alpha_{n-1})$ to be $1\times 1$ matrices.

For $n\ge 2$ we define the \emph{CMV matrix} corresponding to $\{\alpha_k,0\le k\le n-1\}$ as
\begin{align}\label{eq:CMV}
	\cC(\alpha_0,\cdots,\alpha_{n-1}) :=\cL\cM,
\end{align}
where $\cL, \cM$ are $n\times n$ block-diagonal matrices 
\begin{align*}
\cL=\operatorname{diag}\left(\Xi_0,\Xi_2\cdots,\Xi_{2\lfloor\tfrac{n-1}{2}\rfloor}\right),\qquad \cM=\operatorname{diag}\left(\Xi_{-1},\Xi_1\cdots,\Xi_{2\lfloor\tfrac{n}{2}\rfloor-1}\right).
\end{align*}
For $n=1$  we define the $1\times 1$ CMV matrix as  $\cC(\alpha_0)=(\bar \alpha_0)$.
 \end{definition}

Note that $\cC(\alpha_0,\dots, \alpha_{n-1})$ is unitary if and only if $|\alpha_{n-1}|=1$. The following proposition provides a crucial link between CMV matrices and orthogonal polynomials of a discrete probability measure on $\partial \D$.

\begin{proposition}[\cite{KK,OPUC1}]\label{prop:CMV1}
Suppose that $\nu$ is a probability measure  with a support of $n$ distinct points on $\partial \D$, and let $\alpha_k, 0\le k\le n-1$ be its sequence of Verblunsky coefficients. Then 
 for any $1\le k\le n$ we have 
    \begin{equation}\label{eq:CMV_characteristic}
    \det(zI_k-\mathcal{C}(\alpha_0,\dots,\alpha_{k-1})) = \Phi_k(z).
\end{equation}
\end{proposition}

Suppose  that $U$ is an $n\times n$ unitary matrix for which  $\mathbf{e}_1=(1,0,\dots,0)^{\dag}$ is cyclic. The spectral measure of $U$ with respect to $\mathbf{e}_1$ is  given by the following discrete probability measure:
\begin{equation}\label{eq:spectral_measure}
\mu=\sum_{k=1}^n \delta_{\lambda_k} \cdot \left|\langle \mathbf{e}_1,  \mathbf{v}_k\rangle\right|^2.
\end{equation}
Here $\lambda_k,1\le k\le n$ are the eigenvalues of $U$, and $\mathbf{v}_k, 1\le k\le n$ are the corresponding unit length (right) eigenvectors. 
The following proposition summarizes how a unitary matrix and the CMV matrix of its spectral measure are connected.

\begin{proposition}[\cite{CMV}]\label{prop:CMV2}
Suppose that $U$ is an $n\times n$ unitary matrix for which $\mathbf{e}_1$ is cyclic. Let the spectral measure of $U$ with respect to $\mathbf{e}_1$ be $\nu$, and denote the Verblunsky coefficients of $\nu$ by $\alpha_k, 0\le k\le n-1$. Then there is a unitary matrix $V$ with $V\mathbf{e}_1=\mathbf{e}_1$ and $U=V \cC(\alpha_0,\dots, \alpha_{n-1}) V^{-1}$. In particular, $U$ and 
$\cC$ are unitary equivalent, and they have the same spectral measure with respect to $\mathbf{e}_1$.
\end{proposition}

The next definition introduces a simple transformation on a collection of Verblunsky coefficients and on the corresponding CMV matrix. 

\begin{definition}\label{def:rev_Verblunsky}
    Suppose that $(\alpha_0, \dots, \alpha_{n-1})\in \D^{n-1} \times \partial \D$ is the sequence of Verblunsky coefficients of the discrete probability measure $\nu$ on $\partial \D$. Let $\cC$ denote the CMV matrix corresponding to these Verblunsky coefficients. 
    We define the `reversed' version of the Verblunsky coefficients as 
    \begin{align}\label{eq:rev_Verblunsky}
       (\rev \alpha_0,\rev  \alpha_1,\dots, \rev \alpha_{n-1}):= (-\alpha_{n-1} \bar \alpha_{n-2}, -\alpha_{n-1} \bar \alpha_{n-3}, \dots, -\alpha_{n-1}\bar \alpha_0, \alpha_{n-1}). 
    \end{align}
We denote by $\rev  \nu$ and $\rev  \cC$ the probability measure and CMV matrix corresponding to the sequence  $\rev \alpha_k,0\le k\le n-1$, and call these the reversed version of $\mu$ and $\cC$, respectively. 
\end{definition}
Note that since $|\alpha_{n-1}|=1$, the reversal operation is an involution. This operation appeared  in \cite{KillipNenciu} where it was shown that  $\rev  \cC$ is unitary equivalent to $\cC$ (if $n$ is even), and to $\cC^{\dag}$ (if $n$ is odd). This implies that $\cC$ and $\rev  \cC$ have the same eigenvalues, or equivalently, $\mu$ and $\rev  \mu$ have the same support. 
In fact, the arguments of Proposition B.2 of \cite{KillipNenciu} also imply that $\rev  \mu$ is the spectral measure of $\cC$ with respect to the unit vector $\mathbf{e}_n=(0,0,\dots,0,1)^{\dag}$.

The following results were proved in \cite{KK}. They provide  key ingredients for our study of truncated and perturbed unitary matrices. 

\begin{proposition}[\cite{KK}]\label{prop:CMV3}
	Consider the same setup as in Proposition \ref{prop:CMV2}. 
	Then if $n\ge 2$ then the truncated matrix $\trunc{U}$ is unitary equivalent to $\trunc{\cC}$, which in turn is unitary equivalent to 
	\begin{align}\label{eq:CMV_trunc}
	&\cC(\rev \alpha_0, \dots, \rev \alpha_{n-2}), \quad \text{if $n$ is even, and}\quad
	\cC(\rev \alpha_0, \dots, \rev \alpha_{n-2})^{\dag}, \quad \text{if $n$ is odd.} 
	\end{align}
	For $r\in [0,1]$ the perturbed matrix $U^{[r]}$ (as defined in \eqref{def:mult_pert}) is unitary equivalent to 
	\begin{align}\label{eq:CMV_pert}
	&\cC(\rev \alpha_0, \dots, \rev \alpha_{n-2},r\rev \alpha_{n-1}),\quad \text{if $n$ is even, and}\quad \cC(\rev \alpha_0, \dots, \rev \alpha_{n-2},r\rev \alpha_{n-1})^{\dag},\quad \text{if $n$ is odd.}
	\end{align}
\end{proposition}

\subsection{Connection to Dirac-type operators}
\label{sub:finite_Dir}

In this section, we show how the Dirac operator framework introduced in \cite{BVBV_op,BVBV_19,BVBV_szeta}  can be used  to study finitely supported probability measures on $\partial \D$.

The following definition constructs a Dirac-type differential operator corresponding to a finitely supported probability measure on $\partial \D$. 

\begin{definition}\label{def:path}
Suppose that $\mu$ is a probability measure on $\partial\D$ supported on $n$ distinct points $\{e^{i\lambda_j},1\le j\le n\}$, and let $\alpha_k,0\le k\le n-1$ be the corresponding Verblunsky coefficients. For $0\le k\le n$ we define
	\begin{align}\label{eq:b_alpha}
	b_0=0,\quad b_k=\mathcal{P}\begin{pmatrix}
	1 & \bar \alpha_0 \\ \alpha_0 & 1
	\end{pmatrix}\cdots \begin{pmatrix}
	1 & \bar \alpha_{k-1} \\ \alpha_{k-1} & 1
	\end{pmatrix}\begin{pmatrix}
	0\\1
	\end{pmatrix}\quad \text{for $1\le k\le n$}.
	\end{align}
	Set $z_k=\mathcal{U}^{-1}(b_k)$, where $\mathcal{U}$ is the Cayley transform mapping $\overline \HH$ to $\overline \D$ defined via
	\begin{equation}\label{eq:U}
	\mathcal{U} (z) =\mathcal{P}U\begin{pmatrix}
	z\\1
	\end{pmatrix}= \frac{z-i}{z+i},\quad \quad U=\begin{pmatrix}
	1 &-i\\ 1&i
	\end{pmatrix}.
	\end{equation}
  We call  $b_k,0\le k\le n$ and $z_k,0\le k\le n$ the (discrete) \emph{path parameters} of $\mu$ in $\D$ and in $\HH$, respectively. 
We say that $z:[0,1)\to \HH$ defined via 
 $z(t)=(x+iy)(t):=z_{\lfloor nt\rfloor}$
 is the \emph{generating path associated to $\mu$}. We set  $\uu_0=\binom{1}{0}$, $\uu_1=\binom{-z_n}{-1}$, and  we call $\tau=\Dirop(z(\cdot),\uu_0,\uu_1)$ the Dirac-type operator corresponding to $\mu$.
 \end{definition}

We may call $\tau$ the Dirac-type operator corresponding to the Verblunsky coefficients $\alpha_k,0\le k\le n-1$, or the path  $b_0,\dots,b_n$. We may use the notation $\tau=\Dirop(b_{-1},b_0,\dots,b_n)$ with $b_{-1}=\mathcal{P}U^{-1}\uu_0=1$ to emphasize the path dependence. We call $b_{-1}$ and $b_n$ the \emph{left}  and \emph{right} \emph{boundary point} of $\tau$, respectively.  

As the next proposition shows, the Dirac-type operator corresponding to $\mu$ encodes the support of $\mu$ in the spectrum. 
 \begin{proposition}
     [\cite{BVBV_szeta,BVBV_palm}]
	\label{prop:unitary_repr_1}
	Suppose that $\mu$ is a probability measure on $\partial\D$ supported on $n$ distinct points $e^{i\lambda_j},1\le j\le n$, with $\mu(\{1\})= 0$. 
Then the Dirac-type operator $\tau$ corresponding to $\mu$ satisfies Assumption \ref{assumption:u} with $\cI=[0,1)$, and its  spectrum is given by
 \[\spec(\tau)=n\Lambda_n+2\pi n\Z,\qquad \Lambda_n=\{\lambda_1,\dots,\lambda_n\}.
 \]
 \end{proposition}

In \cite{BVBV_op, BVBV_palm} it was observed that the Dirac operator representation for $\mu$ can be simplified using the modified Verblunsky coefficients. These are defined in terms of the Verblunsky coefficients via the recursion 
\begin{align}\label{eq:modifiedV}
\gamma_k=\bar \alpha_k \prod_{j=0}^{k-1} \frac{1-\bar \gamma_j}{1-\gamma_j}, \qquad 0\le k\le n-1.
\end{align}
Note that the modified Verblunsky coefficients satisfy $|\gamma_k|=|\alpha_k|$. Denote by $\cT$ the mapping from the Verblunsky coefficients to the modified ones. This mapping is invertible, in fact for any $k\ge 1$ it provides a one-to-one correspondence between the first  $k$ Verblunsky coefficients and the first $k$ modified Verblunsky coefficients. 
We will use the notation $\cT_k$ for this restricted map.

The modified Verblunsky coefficients of a discrete probability measure $\mu$ are connected to its 
 normalized orthogonal polynomials (with normalization at $1$). Let $\Phi_k$ be the monic orthogonal polynomials of $\mu$, with $\Phi^*_k$ being the reversed polynomials. We set
\begin{equation}\label{eq:normalized_orthogonal}
    \Psi_k(z)=\frac{\Phi_k(z)}{\Phi_k(1)},\qquad \Psi_k^*(z)=\frac{\Phi_k^*(z)}{\Phi_k^*(1)}. 
\end{equation}
These are always well defined for $0\le k\le n-1$, and $\Psi_n$ is defined as long as $\mu(\{1\})=0$. 
The polynomials $\Psi, \Psi^*$ satisfy the following  modified version of the Szeg\H{o} recursion \eqref{eq:Szego1}:
\begin{align}\label{eq:Szego2}
\binom{\Psi_{k+1}}{\Psi_{k+1}^*}=
\mat{\frac{1}{1-\gamma_k}}{-\frac{\gamma_k}{1-\gamma_k}}{-\frac{\bar \gamma_k}{1-\bar \gamma_k}}{\frac{1}{1-\bar \gamma_k}}\mat{z}{0}{0}{1} \binom{\Psi_{k}}{\Psi_{k}^*}, \quad \binom{\Psi_0}{\Psi_0^*}
=\binom{1}{1},\quad 0\le k\le n-1.
\end{align}
The matrices appearing in this recursion correspond to affine transformations. 

Affine transformations of $\HH$ can be parametrized by the elements of $\HH$ as follows. For  $z=x+iy\in\HH$ we define the matrix $A_z$ and the corresponding linear fractional transformation $\mathcal{A}_{z,\HH}:\overline \HH\mapsto\overline \HH$ as follows: 
\begin{equation}\label{e:A1def}
A_{x+i y, \HH}= 
\mat{1}{-x}{0}{y}, \qquad \mathcal A_{x+i y, \HH}(w)= \mathcal{P} A_{x+i y, \HH} \binom{w}{1}.
\end{equation}
Note that $x+i y$ is the pre-image of $i$ under $\mathcal A_{x+i y, \HH}$. 

The transformations $\cA_{z,\HH}$ are isometries of the half-plane model of the hyperbolic plane. The corresponding transformations on the unit-disk  model of the hyperbolic plane can be obtained by  conjugating with the Cayley transform. For $\gamma \in \D$ we set
\begin{align}\label{eq:A_equiv}
A_{\gamma, \D}=U A_{\cU^{-1}(\gamma),\HH} U^{-1}, \qquad \mathcal A_{\gamma, \D}=\mathcal U\circ\mathcal A_{\cU^{-1}(\gamma), \HH}\circ\mathcal U^{-1},
\end{align}
which leads to 
\begin{equation}\label{eq:Adef}
A_{\gamma, \D}= 
\mat{\frac{1}{1-\gamma}}
{\frac{\gamma}{\gamma-1}}
{\frac{\bar \gamma}{\bar \gamma-1}}
{\frac{1}{1-\bar \gamma}},\qquad
\mathcal A_{\gamma, \D}(z)= \mathcal{P} A_{\gamma,\D} \binom{z}{1}. 
\end{equation}
Note that the matrix coefficient in \eqref{eq:Szego2} is exactly $A_{\gamma_k, \D}$, and $\cA_{\gamma,\D}$ maps $\gamma$ to $0$.

The following proposition shows that the generating path of a discrete probability measure $\mu$ can be expressed via a simple (affine) recursion using the modified Verblunsky coefficients.
Define $w_k,v_k\in \mathbb R$ from the modified Verblunsky coefficients as
\begin{align}\label{def:wv}
v_k+iw_k:=\frac{2i\gamma_k}{1-\gamma_k}=\cU^{-1}(\gamma_k)-i.
\end{align}

\begin{proposition}[\cite{BVBV_szeta,BVBV_palm}]
	\label{prop:unitary_repr_2}
	Suppose that $\mu$ is a probability measure on $\partial\D$ supported on $n$ distinct points $e^{i\lambda_j},1\le j\le n$, and $\mu(\{1\})=0$. Let $\gamma_k, 0\le k\le n-1$ be the modified Verblunsky coefficients of $\mu$. Let 
	$b_k, z_k, 0\le k\le n$ be the path parameters of $\mu$ defined as in Definition \ref{def:path}. Then the following identities hold for $0\le k\le n-1$:
	\begin{align}\label{eq:z_rec}
	z_{k+1}&=z_k+(v_k+i w_k) \Im z_k, 
	\end{align}
	with $v_k,w_k$ defined in \eqref{def:wv}, and 
	\begin{equation}
	b_{k+1}=\frac{b_k+\gamma_k\frac{1-b_k}{1-\bar b_k}}{1+\bar b_k \gamma_k \frac{
			1-b_k}{1-\bar b_k}}\label{eq:b_rec}.
	\end{equation}
	We also have
	\begin{equation}
	\label{eq:bk_cA}
	b_k=\cA_{\gamma_0,\D}^{-1}\circ \cdots\circ  \cA_{\gamma_{k-1},\D}^{-1}(0).
	\end{equation}
\end{proposition}

The normalized orthogonal polynomials of $\mu$ can be expressed using the canonical system \eqref{eq:canonical} of the Dirac operator $\tau$ corresponding to $\mu$.

\begin{proposition}[\cite{BVBV_szeta,BVBV_palm}]	\label{prop:unitary_ODE}	Under the same setup as in Proposition  \ref{prop:unitary_repr_1}, consider the solution to the canonical system \eqref{eq:canonical}
	\[
	\tau H = \lambda H,\qquad H(0,\lambda)= \binom{1}{0},\qquad (t,\lambda)\in[0,1]\times \C.
	\]
	Then the normalized orthogonal polynomials $\Psi_k, \Psi_k^*$ of $\mu$ (defined via \eqref{eq:normalized_orthogonal}) satisfy
	\begin{equation}\label{eq:Phi_H}
	\binom{\Psi_k(e^{i\lambda /n})}{\Psi_k^*(e^{i\lambda /n})} = e^{i\lambda k/(2n)}\mat{1}{-z_k}{1}{-\bar z_k}H(k/n,\lambda),\qquad 0\le k\le n.
	\end{equation}
\end{proposition}
Note that Proposition \ref{prop:unitary_ODE} applied with $k=n$ gives
\begin{align}\label{eq:char_pol}
e^{-i \lambda/2} \Psi_n(e^{i \lambda/n})=\binom{1}{-z_n}^\dag H(1,\lambda)=H(1,\lambda)^{\dag} J\binom{1}{z_n}.  
\end{align}
Recall that $\Psi_n(\cdot)$ is just the normalized characteristic polynomial corresponding to the support of $\mu$:
\[
\Psi_n(z)=\frac{\Phi_n(z)}{\Phi_n(1)}=\prod_{j=1}^n \frac{z-e^{i \lambda_j}}{1-e^{i\lambda_j}}
\]
Hence \eqref{def:secular} and \eqref{eq:char_pol} imply that the scaled and normalized characteristic polynomial of $\mu$ is the same as the secular function of the Dirac-type operator corresponding to $\mu$.

\section{beta ensembles from  classical unitary and orthogonal random matrices}
\label{sec:finite}

This section collects the matrix models for various beta-generalizations of unitary and orthogonal random matrices, along with their Dirac operator representations and operator limits. 

\subsection{Finite ensembles and their Dirac operator representations}\label{sub:finite}
\subsubsection*{Circular beta ensemble}

Recall the definition of the size $n$ circular beta ensemble with density function \eqref{eq:cbe} and the spectral measure with respect to $\mathbf{e}_1$ \eqref{eq:spectral_measure}.  
Recall also that for $a_1,\dots, a_n>0$ the Dirichlet$(a_1,\dots, a_n)$ distribution is a probability measure on the set 
\[\{(x_1,x_2,\dots,x_n)\in[0,1]^n:\sum_{i=1}^n x_i=1\}\]
with joint probability density function \begin{align*}
    \frac{\Gamma(\sum_{i=1}^n a_i)}{\prod_{i=1}^n\Gamma(a_i)}\prod_{i=1}^n x_i^{a_i-1}. 
\end{align*}
We  define the size $n$ Killip-Nenciu measure as a random probability measure on $\partial \D$ as
\begin{equation}\label{eq:KillipNenciu}
\mu^{\textup{KN}}_{n,\beta} =\sum_{j=1}^n \pi_j\delta_{e^{i\lambda_j}},
\end{equation}
where the support is distributed as the size $n$ circular beta ensemble, and the weights $\pi_j, 1\le j\le n$ are
chosen according to Dirichlet$(\beta/2,\dots,\beta/2)$ distribution, independently of the support.
The joint distribution of the Verblunsky coefficients with respect to $\mu_{n,\beta}^{\textup{KN}}$ was computed in \cite{KillipNenciu}.

 \begin{definition}\label{def:Theta}
	For  $a>0$ we denote by $\Theta(a+1)$ the distribution on $\D$ that has probability density function
	\begin{align}
	\tfrac{a}{2\pi}(1-|z|^2)^{a/2-1}.
	\end{align}
	The definition is extended for  $a=0$ as follows:  $\Theta(1)$ is the uniform distribution on $\partial \D$.
\end{definition}

\begin{proposition}[\cite{KillipNenciu}]\label{prop:circ_Verblunsky}
For a fixed $n$ and $\beta>0$, the sequence of Verblunsky coefficients $\alpha_k,0\le k\le n-1$ of $\mu_{n,\beta}^{\textup{KN}}$ are independent, and  
$\alpha_k\sim \Theta(\beta(n-k-1)+1)$ for $0\le k\le n-1$.  
\end{proposition}

\begin{definition}\label{def:Circop}
    For fixed $n$ and $\beta>0$, we denote by  $\Circ_{n,\beta}$
    and
    $\Circop_{n,\beta}$ the CMV matrix and Dirac-type operator corresponding to $\mu_{n,\beta}^{\textup{KN}}$.  
\end{definition}

\subsubsection*{Real orthogonal beta ensemble}

The \emph{real orthogonal beta ensemble} is a family of distributions supported on conjugated pairs of points on the unit circle indexed by  real parameters $\beta>0,a>-1,b>-1$. 
If we parametrize the size $2n$ real orthogonal beta ensemble as $\{\pm e^{i \theta_1}, \dots, \pm e^{i \theta_n}\}$ with $\theta_j\in (0,\pi)$ then the joint density for $(\theta_1, \dots, \theta_n)$ is proportional to
\begin{align}
   \prod_{j<k\le n} |\cos(\theta_j)-\cos(\theta_k)|^\beta \times\prod_{k=1}^n |1-\cos(\theta_k)|^{\frac{\beta}{2}(a+1)-1/2} |1+\cos(\theta_k)|^{\frac{\beta}{2}(b+1)-1/2}. \label{eq:PDF_ortho}
\end{align}
The ensemble was introduced in \cite{KillipNenciu}, it can be viewed as a generalization of joint eigenvalue distributions of some classical orthogonal ensembles. For example when $\beta=2$, $a=b=\frac{1}{\beta}-1$, \eqref{eq:PDF_ortho} describes the joint eigenvalue distribution of a $2n\times 2n$ special orthogonal matrix chosen uniformly with respect to the Haar measure on $\mathbb{SO}(2n)$.

Consider the probability measure 
\begin{equation}\label{eq:RO_measure}
    \mu^{\ROb}_{2n,\beta,a,b} = \sum_{j=1}^n \frac12 \pi_j(\delta_{e^{i\theta_j}}+\delta_{e^{-i\theta_j}})
\end{equation}
on the unit circle, where the support is distributed as a $2n$ real orthogonal beta ensemble, and the weights $(\pi_1,\dots,\pi_n)$ is an independent random vector that has a Dirichlet$(\beta/2,\dots,\beta/2)$ distribution. The joint distribution of the Verblunsky coefficients of $\mu_{2n,\beta,a,b}^{\ROb}$ was derived in \cite{KK}.

\begin{definition}\label{def:Beta*}
	For $s,t>0$ let $\tl{\mathrm{B}}(s,t)$ denote the scaled (and flipped) beta distribution on $(-1,1)$ that has probability density function 
	\begin{align*}
	\tfrac{2^{1-s-t}\Gamma(s+t)}{\Gamma(s)\Gamma(t)}(1-x)^{s-1}(1+x)^{t-1}.
	\end{align*} 
\end{definition}

\begin{proposition}[Theorem 2 of \cite{KillipNenciu}, Proposition 4.5 in \cite{KK}] \label{prop:RO_Verblunsky}
	For given $\beta>0$, $a, b>-1$ 
	and fixed $n$, the Verblunsky coefficients $\alpha_k,0\le k \le 2n-1$ of the random probability measure $\mu^{\ROb}_{2n,\beta,a,b}$ are independent with $\alpha_{2n-1}=-1$, and
 for $0\le k\le 2n-2$
	\begin{align*}
	\alpha_{k}\sim \begin{cases}
	\mathrm{\tl B}\big(
	\tfrac{\beta}{4}(2n-k+2a)
	, \tfrac{\beta}{4}(2n-k+2b)
	\big),\quad &\text{if $k$ is even,}\\
	\mathrm{\tl B}\big(\tfrac{\beta}{4}(2n-k+2a+2b+1),
	\frac{\beta}{4}(2n-k-1)
	\big),\quad &\text{if $k$ is odd.}
	\end{cases}
	\end{align*}
\end{proposition}

\begin{definition}\label{def:ROop}
    For fixed $n\ge 1$, $a,b>-1$ and $\beta>0$, we define $\RO_{2n,\beta,a,b}$ and $\ROop_{2n,\beta,a,b}$ as the CMV matrix model and the Dirac-type operator corresponding to $\mu_{2n,\beta,a,b}^{\ROb}$, respectively. 
\end{definition}

\subsubsection*{Circular Jacobi beta ensemble}

The circular Jacobi beta ensemble can be viewed as a one-parameter generalization of the circular beta ensemble.
 Let $\delta$ be a complex parameter such that $\Re\delta>-1/2$. The size $n$ circular Jacobi beta ensemble is defined as the probability measure on the $n$ distinct points $\{e^{i \theta_1}, \dots, e^{i \theta_n}\}$ with $\theta_j\in [-\pi,\pi)$, where the joint density function of the angles $\theta_j$ is given by \eqref{eq:CJ_pdf}.
For $\beta=2$  the distribution was studied by Hua \cite{Hua} and Pickrell \cite{Pickrell}, this special case is sometimes called the Hua-Pickrell measure. 

Consider the random probability measure
\begin{equation}\label{eq:CJ_measure}
\mu^{\CJb}_{n,\beta,\delta} =\sum_{j=1}^n\pi_j\delta_{e^{i\theta_j}}
\end{equation}
on the unit circle, where the support is distributed as the size $n$ circular Jacobi beta ensemble,  the weights are  Dirichlet$(\beta/2,\dots,\beta/2)$ distributed and are independent of the support.  In \cite{BNR2009} the authors showed that the modified Verblunsky coefficients $\gamma_k,0\le k\le n-1$ of $\mu_{n,\beta,\delta}^{\CJb}$ are independent and described explicitly their joint distribution. 

We first introduce a generalization of the $\Theta(a+1)$ distribution defined in Definition \ref{def:Theta}. 

\begin{definition}\label{def:Theta_delta}
For  $a>0$ and $\Re \delta>-1/2$ let $\Theta(a+1,\delta)$ be the distribution on $\D$ that has probability density function
\begin{align}
	 c_{a,\delta} (1-|z|^2)^{a/2-1}(1-z)^{\bar \delta} (1-\bar z)^{\delta},\qquad c_{a,\delta}=\tfrac{\Gamma(a/2+1+\delta)\Gamma(a/2+1+\bar \delta)}{\pi \Gamma(a/2)\Gamma(a/2+1+\delta+\bar \delta)}.
	\end{align}
For the $a=0$, $\Re \delta>-1/2$ case we define $\Theta(1,\delta)$ to be the distribution on $\partial \D=\{|z|=1\}$ with probability density function
\begin{align}
\tfrac{\Gamma(1+\delta)\Gamma(1+\bar \delta)}{\Gamma(1+\delta+\bar \delta)} (1-z)^{\bar \delta} (1-\bar z)^{\delta}.
\end{align}
\end{definition}

\begin{proposition}[\cite{BNR2009}]\label{prop:CJ_mVerblunsky}
    For given $n\ge 1, \beta>0,\Re\delta>-1/2$,  the sequence of modified Verblunsky coefficients  $\gamma_k,0\le k\le n-1$ of $\mu_{n,\beta,\delta}^{\textup{CJ}}$ are independent with  $\gamma_k\sim \Theta(\beta(n-k-1)+1,\delta)$ for $0\le k\le n-1$.
\end{proposition}
Note that for  $\delta\ne 0$ the Verblunsky coefficients $\alpha_k,0\le k\le n-1$ are not independent.

\begin{definition}\label{def:CJop}
     For given $n\ge 1, \beta>0,\Re\delta>-1/2$, 
      we define $\CJ_{n,\beta,\delta}$ and $\CJop_{n,\beta,\delta}$ as the CMV matrix model and the Dirac-type operator corresponding to $\mu_{n,\beta,\delta}^{\textup{CJ}}$, respectively.    
\end{definition}

\subsection{Random operator limits from beta ensembles}\label{sub:limiting_op}
This section discusses the operator limits of the finite ensembles in Section \ref{sub:finite}. The main idea is that under the appropriate scaling, the piece-wise constant generating paths of the $\Circop_{n,\beta}, \ROop_{2n,\beta,a,b}$ and $\CJop_{n,\beta,\delta}$ operators converge to certain diffusions in the hyperbolic plane. Then one can construct random differential operators in terms of these diffusions. These limiting operators will be denoted as $\Sineop,\Bessop$ and $\HPop$, respectively. 
For convenience, we will define the limiting operators with generating paths that lie in $\HH$. 
We also set 
\begin{equation*}
v(t)=v_\beta(t)=-\frac{4}{\beta}\log(1-t)
\end{equation*}
be the logarithmic time change function.

The $\Sineop$ operator was introduced in \cite{BVBV_op} as the $n\to\infty$ limit of the $\Circop_{n,\beta}$ operator. The operator level convergence was proved in \cite{BVBV_19} with an explicit rate of convergence, see Proposition \ref{prop:Sineb_coupling} below.

\begin{definition}\label{def:Sineop}
Fix $\beta>0$. Let $B_1, B_2$ be independent standard Brownian motion, and let $\mathsf{x}_v+i \mathsf{y}_v, v\ge 0$ be the strong solution of the SDE
\begin{align}\label{eq:Sineop_sde}
d\mathsf{y}=\mathsf{y} dB_1,\quad d\mathsf{x}=\mathsf{y}dB_2, \quad \mathsf{y}(0)=1, \mathsf{x}(0)=0.
\end{align}
For $t\in[0,1)$ define $ z(t)= x(t)+i y(t) =\mathsf{x}_v+i\mathsf{y}_v$ where $v=v(t)$. Let $\uu_0=\binom{1}{0}$, $\uu_1=\binom{-q}{-1}$, where $q=\lim_{t\to\infty} \mathsf{x}(t)$. Set $\Sineop=\Dirop( z(\cdot),\uu_0,\uu_1)$.
\end{definition}

Note that $\mathsf{x}_v+i \mathsf{y}_v, v\ge 0$  is just a hyperbolic Brownian motion in $\HH$, started from $i$. 

The hard-edge Dirac operator $\Bessop$ and the Hua-Pickrell operator $\HPop$ were also introduced in \cite{BVBV_op}. It has been shown in \cite{LV} that they are indeed the operator level limits of the $\ROop_{2n,\beta,a,b}$ and $\CJop_{n,\beta,\delta}$ operators, respectively. 

\begin{definition}\label{def:Bessop}
    Fix $\beta>0,a>-1$, and let $B$ be a standard Brownian motion. Set  $\mathsf{y}(t)=e^{-\frac{\beta}{4}(2a+1)t-B(2t)}$, $ y(t)=\mathsf{y}(v(t))$, $\uu_0=\binom{1}{0}$, and $\uu_1=\binom{0}{-1}$. Define the hard-edge operator as $\Bessop:=\Dirop (i  y, \uu_0, \uu_1)$.
\end{definition}

\begin{definition}\label{def:HPop}
    Fix $\beta>0$ and $\delta\in\C$ with $\Re \delta>-1/2$. 
    Let $B_1, B_2$ be independent standard Brownian motion, and let $\mathsf{x}_v+i \mathsf{y}_v, v\ge 0$ be the strong solution of the SDE
\begin{align}\label{eq:HPsde}
d\mathsf{y}=\left( -\Re \delta dt +  dB_1 \right)\mathsf{y},\quad d\mathsf{x}=\left(\Im \delta dt +  dB_2 \right)\mathsf{y}, \quad \mathsf{y}(0)=1, \mathsf{x}(0)=0.
\end{align}
For $t\in[0,1)$ define $ z(t)= x(t)+i y(t) =\mathsf{x}_v+i\mathsf{y}_v$ where $v=v(t)$. Let $\uu_0=\binom{1}{0}$, $\uu_1=\binom{-q}{-1}$, where $q=\lim_{v\to\infty} \mathsf{x}(v)$. 
Set $\HPop=\Dirop(z(\cdot),\uu_0,\uu_1)$. 
\end{definition}
Note that the SDE \eqref{eq:HPsde} can be solved explicitly and the solution is given by
\begin{equation}\label{eq:hpsol}
	\mathsf{y}(v)=e^{B_1(v)-(\Re\delta+\frac12)v},\quad \mathsf{x}(v)=\int_0^v \mathsf{y}(s)dB_2(s)+\Im\delta\int_0^v \mathsf{y}(s) ds.
\end{equation}
In the case when $\delta=0$ the equation reduces  to \eqref{eq:Sineop_sde}. 
In particular, the $\HPop$ operator can be viewed as the $\delta$-generalization of $\Sineop$, and we have $\mathtt{HP}_{\beta,0}=\Sineop$.

\subsection{Random analytic functions from beta ensembles}

As explained in Section \ref{sub:Dirac_secular}, the Hilbert–Schmidt convergence of the resolvents of Dirac-type operators and the convergence of the integral traces imply the uniform on compacts convergence of the secular functions. In Section \ref{sub:finite_Dir} we saw that we can express the characteristic polynomials of finite ensembles on the unit circle via the secular function of an associated Dirac-type operator. The operator level limits discussed in the previous section lead then lead to convergence statements regarding the scaled and normalized characteristic polynomials of the respective finite ensembles.

In this section, we briefly review the secular functions and structure functions arising from the limits of the considered beta ensembles. The constructions rely on certain time-reversed and transformed versions of the limiting operators introduced in Section \ref{sub:limiting_op}. Recall the  transformations introduced in Section \ref{sub:Dirac_secular}. 

In \cite{BVBV_szeta} Valk\'o and Vir\'ag constructed the following Dirac-type operator on $(0,1]$ and showed that it is orthogonally equivalent to the $\Sineop$ operator. Consider the time change
\begin{equation}\label{eq:u_timechange}
    u(t)=u_\beta(t)=\frac{4}{\beta}\log t,\quad\quad t\in(0,1].
\end{equation}

\begin{definition}\label{def:Sineop_rev}
Let $b_1, b_2$ be independent two-sided standard Brownian motion, and set 
	\begin{align}\label{eq:Sine_rev_path}
	\mathsf{y}_u=e^{b_2(u)-u/2}, \qquad  \mathsf{x}_u=-\int_u^0 e^{b_2(s)-\tfrac{s}{2}} d b_1, 
	\end{align}
For $t\in(0,1]$ set $\hat z(t)=(\hat x+i\hat y)(t):=\mathsf{x}_u+i\mathsf{y}_u,\uu_0=\binom{1}{0}$, and $\uu_1=\binom{-q}{-1}$, where $q$ is a Cauchy distributed random variable independent of  $b_1, b_2$. Define $\tau_\beta =\tau_\beta^{\textup{Sine}}= \Dirop(\hat z(\cdot),\uu_0,\uu_1)$, and denote by $\zeta_\beta:=\zeta_{\tau_\beta}$ and $\cE_\beta:=\cE_{\tau_{\beta}}$ the corresponding secular and structure function according to Definition \eqref{def:RAF}.
\end{definition}

\begin{proposition}[\cite{BVBV_szeta}]\label{prop:Sineop_equiv}
Let $q$ and $\tau_\beta$ be defined as in Definition \ref{def:Sineop_rev}. Then the $\tau_\beta$ operator satisfies Assumption \ref{assumption:u}, and the orthogonally equivalent operator
	\begin{equation*}
	 \rho^{-1}SQ\tau_\beta Q^{-1}S\rho ,\qquad Q=\frac{1}{\sqrt{1+q^2}}\mat{q}{1}{-1}{q},
	\end{equation*}
 has the same distribution as $\Sineop$. Let $H_\beta(t,z)$ be the solution to the canonical system \eqref{eq:canonical} of the $\tau_\beta$ operator, and let $\mathsf{x}_u$  and  $\mathsf{y}_u$ be defined according to \eqref{eq:Sine_rev_path}. Set
 \begin{equation}\label{eq:cH_beta}
 \cH_\beta(u,z)=\mat{1}{-\mathsf{x}_u}{0}{\mathsf{y}_u}H_\beta(t(u),z),\qquad t=t(u)=e^{\beta u/4},
 \end{equation}
 then $\cH$ satisfies the stochastic differential equation \eqref{eq:cH}. Since $\cH_\beta(0,z)=H_\beta(1,z)$, the structure-function $\cE_{\beta}(z)$ can be represented as $\cE_{\beta}(z)=\cH_\beta(0,z)^{\dag}\binom{1}{-i}$.
 \end{proposition}

Using the results of Proposition \ref{prop:Sineop_equiv} \cite{BVBV_szeta} showed the convergence of the scaled and normalized characteristic polynomials of the circular beta ensemble to the stochastic zeta function. 
\begin{proposition}[\cite{BVBV_szeta}]\label{prop:Sineop_equiv_2}
 Consider the size $n$ (unperturbed) circular beta ensemble, i.e., the eigenvalues of $\Circ_{n,\beta}$, and its normalized characteristic polynomial $p_{n,\beta}(z):=\frac{\det(zI-\Circ_{n,\beta})}{\det(I-\Circ_{n,\beta})}$. There exists a coupling of $p_{n,\beta}(z),n\ge 1$, the secular function $\zeta_{\beta}(z)=\cH_{\beta}(0,z)^\dag\binom{1}{-q}$, and an a.s.~finite $C$ so that for all $z\in\C$ we have
 \begin{align*}
     |p_{n,\beta}(e^{iz/n})e^{-iz/2}-\zeta_{\beta}(z)|\le C^{|z|^2+1}\left(e^{|z|\frac{\log^3 n}{\sqrt{n}}}-1\right).
 \end{align*}
 \end{proposition}

The time-reversed and transformed version of the $\Bessop$ and $\HPop$ operators were constructed in \cite{LV} in a similar spirit.

\begin{definition}\label{def:Bess_rev}
    Let $B$ be a standard two-sided Brownian motion. Let $\mathsf{y}_u=e^{-\frac{\beta}{4}(2a+1)u+B(2u)}$ and $\hat y(t)=\mathsf{y}(u(t))$ for $t\in(0,1]$, where $u$ is defined in \eqref{eq:u_timechange}. Set $\uu_0 =\binom{1}{0}, \uu_1=\binom{0}{-1}$ and define $\tau_{\beta,a}= \tau_{\beta,a}^{\textup{Bess}} = \Dirop(i\hat y(t), \uu_0, \uu_1)$. We also define $\zeta_{\beta,a}=\zeta_{\tau_{\beta,,a}}$ and $\cE_{\beta,a}=\cE_{\tau_{\beta,a}}$ as the secular and structure function of the $\tau_{\beta,a}$ operator via Definition \ref{def:RAF}.
\end{definition}
\begin{proposition}\label{prop:Bessop_equiv}
    The operator $\tau_{\beta,a}$ satisfies Assumption \ref{assumption:u}, and we have 
    \[
    \rho^{-1}J\tau_{\beta,a}J\rho\ed \Bessop,\qquad J=\mat{0}{-1}{1}{0}.
    \]
   Let $H_{\beta,a}(t,z)$ be the solution to the canonical system \eqref{eq:canonical} of the $\tau_{\beta,a}$ operator and set $\cH_{\beta,a}(u,z)=\diag(1,\mathsf{y}_u)H_{\beta,a}(e^{\frac{\beta}{4}u},z)$ for $u\le 0,z\in\C$, where $\mathsf{y}_u$ is defined as in Definition \ref{def:Bess_rev}. Then $\cH_{\beta,a}$ is the unique strong solution to the SDE
    \begin{equation}\label{eq:Bess_cH}
d\cH= \begin{pmatrix}
0 & 0\\0& \sqrt{2}dB+(1-\frac{\beta}{4}(2a+1))du
\end{pmatrix}\cH-z\frac{\beta}{8}e^{\beta u/4}\begin{pmatrix}
0&-1\\1&0
\end{pmatrix}\cH du,
\end{equation}	
with boundary conditions $\lim\limits_{u\to -\infty} \sup_{|z|<1} \left|\cH(u,z)-\binom{1}{0}\right|=0$. Since $\cH_{\beta,a}(0,z)=H_{\beta,a}(1,z)$, we have $\cE_{\beta,a}(z)=\cH_{\beta,a}(0,z)^{\dag}\binom{1}{-i}$ and $\zeta_{\beta,a}=\cH_{\beta,a}(1,z)^{\dag}\binom{1}{0}$. 
\end{proposition}

\cite{LV} showed that  the secular function $\zeta_{\beta,a}$ of the $\Bessop$ operator is the limit in distribution of the normalized characteristic polynomial of the real orthogonal beta ensemble under the edge scaling \eqref{eq:edgescaling}.

\begin{definition}\label{def:HPop_rev}
    Let $B_1, B_2$ be independent two-sided standard Brownian motion, and for $u\le 0$ define
	\begin{align}\label{eq:xy}
	\mathsf{y}_u=e^{B_2(u)-(\Re\delta+\frac12)u}, \quad  \mathsf{x}_u=
-\int_u^0 e^{B_2(s)-(\Re\delta+\frac12)s} dB_1 - \Im\delta\int_u^0 e^{B_2(s)-(\Re\delta+\frac12)s}ds, 
	\end{align}
For $t\in (0,1]$ let $\hat z(t)=(\hat x+i\hat y)(t)=\mathsf{x}_u+i\mathsf{y}_u$ with $u=u(t)$ defined in \eqref{eq:u_timechange}. Set $\uu_0=\binom{1}{0}$, and $\uu_1=\binom{-q}{-1}$, where $q\sim\Theta(1,\delta)$ is independent of  $B_1, B_2$. Define $\tau_{\beta,\delta}=\tau_{\beta,\delta}^{\textup{HP}} = \Dirop(\hat z(\cdot),\uu_0,\uu_1)$, and denote by $\zeta_{\beta,\delta}=\zeta_{\tau_{\beta,\delta}}$ and $\cE_{\beta,\delta}=\cE_{\tau_{\beta,\delta}}$ the secular and structure function of $\tau_{\beta,\delta}$, respectively.\footnote{From the context it will always be clear if $\zeta_{\beta,\cdot}$, $\cE_{\beta,\cdot}$ refer to the objects related to the $\Bessop$ or the $\HPop$ operators.}
\end{definition}

\begin{proposition}[\cite{LV}]\label{prop:HP_equiv}    
Let $q$ and $\tau_{\beta,\delta}$ be defined as in Definition \ref{def:HPop_rev}. Then the $\tau_{\beta,\delta}$ operator satisfies Assumption \ref{assumption:u}, and the orthogonally equivalent operator
	\begin{equation*}
	 \rho^{-1}SQ\tau_{\beta,\delta} Q^{-1}S\rho ,\qquad Q=\frac{1}{\sqrt{1+q^2}}\mat{q}{1}{-1}{q},
	\end{equation*}
 has the same distribution as $\HPop$. Let $H_{\beta,\delta}$ be the solution of the canonical system \eqref{eq:canonical} of the $\tau_{\beta,\delta}$ operator, and let $\mathsf{x}_u,\mathsf{y}_u$ be defined as in \eqref{eq:xy}. For $u\le 0,z\in\C$ set 
 \[
 \cH_{\beta,\delta}(u,z) = \mat{1}{-\mathsf{x}_u}{0}{\mathsf{y}_u}H_{\beta,\delta}(e^{\beta u/4},z),
 \]
 then $\cH_{\beta,\delta}$ is the unique solution of the SDE 
 \begin{equation}\label{eq:HP_cH}
	d\mathcal{H} = \begin{pmatrix}
	0 & -dB_1\\0& dB_2
	\end{pmatrix}\mathcal{H}+ \begin{pmatrix}
	0 & -\Im\delta du\\0& -\Re\delta du
	\end{pmatrix}\mathcal{H} -z\frac{\beta}{8}e^{\beta u/4}\mat{0}{-1}{1}{0}\mathcal{H} du, 
\end{equation}
with the boundary condition $\lim\limits_{u\to -\infty} \sup_{|z|<1} \left|\cH(u,z)-\binom{1}{0}\right|=0$, and we have $\cE_{\beta,\delta}(z)=\cH_{\beta,\delta}(0,z)^{\dag}\binom{1}{-i}$. 
\end{proposition}

\cite{LV} also showed that the secular function $\zeta_{\beta,\delta}=\cH_{\beta,\delta}(0,z)^{\dag}\binom{1}{-q}$ of the $\HPop$ operator is the limit in distribution of the normalized characteristic polynomials of the circular Jacobi beta ensemble under the edge scaling \eqref{eq:edgescaling}.

\section{Convergence of the truncated models}\label{sec:general}

This section establishes general convergence results for the rank-one truncations and multiplicative perturbations of the finite ensembles when the generating paths satisfy certain path bounds. The statements in this section hold in the deterministic setting. 

Throughout the section, we assume $\mu_n$ is a probability measure on $\partial \D$ supported on $n$ distinct points, with Verblunsky coefficients $\alpha_0,\dots, \alpha_{n-1}$. We denote by $\tau_n$ the Dirac operator corresponding to $\mu_n$.

Recall the reversed (discrete) probability measures defined in Definition \ref{def:rev_Verblunsky}. We first introduce the reversed Dirac operator and its pulled-back version, which are closely connected to the truncated ensembles. 

\begin{definition}\label{def:rev_Dirac}
Let $\rev \mu_n$ be the reversed version of $\mu_n$, i.e., the probability measure corresponding to the reversed Verblunsky coefficients  $(\rev \alpha_0,\rev \alpha_1,\dots,\rev \alpha_{n-1})$.
We define $\rev \tau_n$ as the Dirac operator corresponding to $\rev \mu_n$ and call it the reversed version of $\tau_n$. We denote by $\rev  b_k,0\le k\le n$ the path parameters  of $\rev \tau_n$ in $\D$.
\end{definition}

 Recall the affine transformation $\mathcal{A}_{\gamma,\D}$ defined in \eqref{eq:Adef}.
\begin{definition}\label{def:aff_Dirac}
Set $\aff  b_k:=\cA_{\rev  b_{n-1},\D}(\rev  b_k)$ for $0\le k\le n$ with the extension that $\aff b_{-1}=\cA_{\rev b_{n-1},\D}(\rev b_{-1})=1$. We define $\aff \tau_n:=\cA_{\rev b_{n-1},\D}(\rev \tau_n):=\Dirop(1, \aff b_0,\dots,\aff b_{n-1},\aff b_n)$ as the Dirac operator with path parameters $\aff b_k,0\le k\le n$ in $\D$. We call $\aff\tau_n$ the pulled-back version of $\rev \tau_n$. 
\end{definition}

To summarize, we are considering three sets of path parameters corresponding to $\mu_n$, and each one has a Dirac-type operator associated to it. $b_k, 0\le k\le n$ are the path parameters constructed from the Verblunsky coefficients $\alpha_k, 0\le k\le n-1$ according to Definition \ref{def:path}. The reversed path $\rev b_k, 0\le k\le n$ is constructed via the same definition, but from the reversed Verblunsky coefficients \eqref{eq:rev_Verblunsky}. Finally, the path $\aff b_k, 0\le k\le n$ can be obtained from the reversed path by applying an affine transformation that maps $\rev b_{n-1}$ to 0 in $\D$. See Table \ref{table:summary} for a summary of the defined objects.

Note that by the definition of $\cA_{\gamma,\D}$ and the recursion \eqref{eq:b_rec}, we have $\aff b_{n-1}=0$ and $\aff b_n=\rev \gamma_{n-1}$, where $\rev \gamma_{n-1}$ is the last modified Verblunsky coefficient corresponding to $\rev \mu_n$.

\begin{table}
\begin{center}
    \begin{tabular}{ccccc}
         original objects&&reversed objects&&pulled-back objects  \\
         \hline\hline\\
         $\mu$&&$\rev{\mu}$&&\\
         discrete prob.~measure on $\partial \D$&&reversed measure&&\\[5pt]
         $\updownarrow$ \eqref{eq:Szego1}&\eqref{eq:rev_Verblunsky}&$\uparrow$\eqref{eq:Szego1}&&\\[5pt]
         $\alpha_k,0\le k\le n-1$&$\longleftrightarrow$&$\rev \alpha_k, 0\le k\le n-1$&&\\
         Verblunsky coefficients&&reversed Verbl.~coeffs&&\\[5pt]
         $\updownarrow$ \eqref{eq:modifiedV}&&&&\\[5pt]
         $\gamma_k,0\le k\le n-1$&&\multirow{2}{*}{$\xdownarrow{1 cm}$ Def.~\ref{def:path}}\\
         modified Verblunsky coeffs&&&&\\[5pt]
     $\updownarrow$ Def.~\ref{def:path}&&&Def.~\ref{def:aff_Dirac}&\\[5pt]
        $b_k, 0\le k\le n$&&$\rev b_k, 0\le k\le n$&$\longrightarrow$& $\aff b_k, 0\le k\le n$\\
        $z_k,0\le k\le n$&&$\rev z_k,0\le k\le n$&&$\aff z_k,0\le k\le n$\\
        path parameters&&reversed path&&pulled-back path\\[5pt]
           $\downarrow$ Def.~\ref{def:path}&&$\downarrow$Def.~\ref{def:path}&& $\downarrow$Def.~\ref{def:path}, \ref{def:aff_Dirac}\\[5pt]
           $\tau=\Dirop(b_{-1},\dots,b_n)$&& $\rev \tau=\Dirop(\rev b_{-1},\dots,\rev b_n)$&&$\aff \tau=\Dirop(\aff b_{-1},\dots,\aff b_n)$\\
           Dirac-type operator&&reversed operator&&pulled-back operator 
    \end{tabular}
\end{center}
\caption{Table of objects associated to a probability measure $\mu$ supported on $n$ points on $\partial \D$.}
\label{table:summary}
\end{table}

The next statement represents the orthogonal polynomial of degree $n-1$ of $\rev \mu_n$ using the operator $\aff \tau_n$.
\begin{proposition}\label{prop:general1}
Let $\rev \mu_n$ be the reversed version of $\mu_n$, and $\rev \tau_n,\aff \tau_n$ be defined as in Definitions \ref{def:rev_Dirac} and \ref{def:aff_Dirac}. 
 For $0\le k\le n-1$, let $\rev \Phi_k(\cdot)$ be the monic orthogonal polynomials of degree $k$ associated to $\rev \mu_n$, and set $\rev \Psi_k(\cdot)=\rev \Phi_k(\cdot)/\rev \Phi_k(1)$. Then we have
	\[
	\rev \Psi_{n-1}(e^{iz/n}) = e^{iz(n-1)/(2n)}\aff H_n\big(\tfrac{n-1}{n},z\big)^{\dag}\binom{1}{-i},
	\]
	where $\aff H_n(t,z)$ is the solution of the canonical system \eqref{eq:canonical} corresponding to the operator $\aff \tau_n$.
\end{proposition}

\begin{proof}
	Let $\rev b_k,0\le k\le n$ be the path parameters of $\rev \tau_n$ and set $\rev z_k=\mathcal{U}^{-1}(\rev b_k),0\le k\le n$ to be the corresponding path parameters in $\HH$. Here $\cU$ is the Cayley transform defined in \eqref{eq:U}. Consider the linear fractional transformation $\cA_{\rev z_{n-1},\HH}$ (the upper half plane representation of $\cA_{\rev b_{n-1},\D}$) such that  
\[
\cA_{\rev z_{n-1},\HH}(w) =\mathcal{U}^{-1}\circ\cA_{\rev b_{n-1},\D}\circ\mathcal{U}(w)= \frac{w-\Re \rev z_{n-1}}{\Im \rev z_{n-1}},\qquad w\in\HH.
\]
Define  $\aff z_k=\aff x_k+i\aff y_k=\cA_{\rev z_{n-1},\HH}(\rev z_k)$  for $0\le k\le n$. Then by \eqref{eq:A_equiv} we have $\aff z_k=\cU^{-1}(\aff b_k), 0\le k\le n$ with the extension that $\aff z_{-1}=\cU^{-1}(\rev b_{-1})=\infty$.

Introduce the temporary notation
\[
\rev X_k=A_{\rev z_k,\HH}=\mat{1}{-\rev x_k}{0}{\rev y_k}, \quad \aff X_k = A_{\aff z_k,\HH} = \mat{1}{-\aff x_k}{0}{\aff y_k}, \qquad 0\le k\le n,
\]
then we have 
\begin{equation}\label{eq:Xk_trans}
\aff X_k=\rev X_k(\rev X_{n-1})^{-1} ,\quad \quad 0\le k\le n. 
\end{equation}
For a matrix of the form 
\[
X=\mat{1}{-x}{0}{y}
\]
with $y>0$ we have the identity
\[
y^{-1}X^{\dag}J=JX^{-1}. 
\]
Using this together with \eqref{def:R} and \eqref{def:Dirop}
we obtain \begin{align}\label{eq:tau1234}
\aff \tau_{n}=\rev X_{n-1}\rev \tau_n(\rev X_{n-1})^{-1}.
\end{align}
Let $\aff H_n(t,z)$ and $\rev H_n(t,z)$ be the solutions to the canonical systems \eqref{eq:canonical} of $\aff \tau_{n}$ and $\rev \tau_{n}$ respectively, then \eqref{eq:tau1234} leads to   
\begin{align}\label{eq:HHH}
\rev H_n(t,z)=(\rev X_{n-1})^{-1}\aff H_n(t,z).
\end{align}
Let $\rev{\Psi}_k^*(z)=z^k\overline{\rev \Psi_k(1/\bar z)}$ be the reversed polynomial of $\rev \Psi_k(z)$. 
Using \eqref{eq:U}, \eqref{eq:Phi_H}, \eqref{eq:Xk_trans}, and \eqref{eq:HHH} we get 
\begin{align*}
\binom{\rev \Psi_k(e^{iz/n})}{\rev \Psi_k^{*}(e^{iz/n})} &= e^{izk/(2n)}U\rev X_k\rev H_n(k/n,z)=e^{izk/(2n)}U\aff X_k\aff H_n(k/n,z).
\end{align*}
Since $\aff z_{n-1}=\cA_{\rev z_{n-1},\HH}(\rev z_{n-1})=i$, and $\aff X_{n-1}=I$, we have
\[
\rev \Psi_{n-1}(e^{iz/n}) = e^{iz(n-1)/(2n)}\aff H_n((n-1)/n,z)^{\dag}\binom{1}{-i},
\]
finishing the proof. 
\end{proof}

	Suppose now that the probability measure $\mu_n$ is the spectral measure of an $n\times n$ unitary or orthogonal matrix $U$ with respect to $\mathbf{e}_1$. By Propositions \ref{prop:CMV1} and \ref{prop:CMV3}, the eigenvalues of the truncated matrix $\trunc{U}$ are exactly the zeros of the normalized orthogonal polynomial $\rev \Psi_{n-1}(\cdot)$, hence they can be expressed from the vector-valued function $\aff H_n$.
	
	Next we turn to the discussion of the scaling limit of the eigenvalues of the truncated matrices. If we can prove uniform on compacts convergence of the normalized orthogonal polynomials, then this leads to the convergence of the eigenvalues of the truncated models. By Proposition \ref{prop:general1} it is sufficient to show the convergence 
  the convergence of $\aff  H_n$. Using Proposition \ref{prop:H_comparison}, this can be done if we have  sufficient control of the integral traces, resolvent norms, and the proper integrals of the kernels of $\res\aff \tau_n$. We summarize this idea in the following statement. 

\begin{proposition}\label{prop:general2}
Suppose that $\mu_n, n\ge 1$ is a sequence of probability measures on $\partial \D$, with $\mu_n$ supported on $n$ distinct points. Consider the setup of 
  Proposition \ref{prop:general1} and denote by $z_n(\cdot)$  the generating path of $\aff\tau_n$ for $n\ge 1$. 
	Suppose that there exists a Dirac operator $ \tau_\infty=\Dirop(z_\infty(\cdot),\uu_0,\uu_1)$ with $\uu_0=\binom{1}{0}$ and $\cI=(0,1]$ so that 
	\begin{equation}\label{eq:discrete_assumption}
\|\res\aff \tau_n-\res\tau_\infty\|_{\textup{HS}}\to 0,\qquad \ttr_{\aff \tau_n}-\ttr_{\tau_\infty}\to 0,\qquad \int_0^1 |\mathfrak{a}_{0,n}(s)-\mathfrak{a}_{0,\infty}(s)|^2 ds \to 0,
\end{equation}
where $\mathfrak{a}_{0,n}(\cdot)=\binom{\Im z_{n}(\cdot)^{-\frac12}}{0},n\in \Z_+\cup\{\infty\}$ is defined according to \eqref{Dir:inverse}. Let $\rev \Psi_{n-1,n}$ be the normalized orthogonal polynomial of degree $n-1$ of $\rev \mu_n$, and let $\cE(z):=H(1,z)^{\dag}\binom{1}{-i}$ be the structure function of $\tau_\infty$. Then we have
\begin{equation}\label{eq:Psi_conv_general}
|e^{-iz/2}\rev \Psi_{n-1,n}(e^{iz/n})-\cE(z)|\to 0\quad \text{uniformly on compacts in $\C$ as $n\to\infty$.}
\end{equation}
\end{proposition}
Note that \eqref{eq:Psi_conv_general} implies that the zeros of $\rev \Psi_{n-1,n}$ converge to the zeros of $\cE$ under the edge scaling \eqref{eq:edgescaling}. 
\begin{proof}[Proof of Proposition \ref{prop:general2}]
Let $\aff H_n,n\ge 1$ be the solution of the canonical system \eqref{eq:canonical} of $\aff \tau_n$. By Proposition \ref{prop:general1} and the triangle inequality, we have \begin{equation}\label{ineq:circ_conv_tri}
\begin{aligned}
|\rev \Psi_{n-1,n}(e^{iz/n})e^{-iz/2}-\cE(z)|&\le |e^{-iz/(2n)}|\left|\Big(\aff H_n(\tfrac{n-1}{n},z)^{\dag}- H(\tfrac{n-1}{n},z)^{\dag}\Big)\cdot\binom{1}{-i}\right|\\
& + |e^{-iz/(2n)}|\left|\Big( H(\tfrac{n-1}{n},z)^{\dag}- H(1,z)^{\dag}\Big)\cdot\binom{1}{-i}\right|\\
& +\left|e^{-iz/(2n)}-1\right|\Big|H(1,z)^{\dag}\cdot\binom{1}{-i}\Big|.
\end{aligned}
	\end{equation}
	
	For $z$ in a compact set of $\C$ and large enough $n$, the sum of the second and the third term of \eqref{ineq:circ_conv_tri} is upper bounded by an error term of the order $O(n^{-1})$. Hence it remains to provide an upper bound for the difference $|\aff H_n(\tfrac{n-1}{n},z)- H(\tfrac{n-1}{n},z)|$.  By Proposition \ref{prop:H_comparison}, it suffices to control the terms 
\[
|\ttr_{\tau_\infty}-\ttr_{\aff \tau_{n}}|,\quad \quad \|\res  \tau_\infty-\res \aff \tau_{n}\|_{\textup{HS}},\quad \quad \int_0^1|\mathfrak{a}_{0,\infty}(s)-\mathfrak{a}_{0,n}(s)|^2 ds.
\]
Applying the assumptions \eqref{eq:discrete_assumption} finishes the proof. 
\end{proof}

Next, we turn to the discussion of the rank-one multiplicative perturbation of the matrix models. Suppose
$U$ is an $n\times n$ unitary or orthogonal matrix with CMV matrix form $\cC(\alpha_0,\dots,\alpha_{n-1})$. Let $r\in [0,1]$ and $U^{[r]}=U\cdot \diag(r,1,1,\dots,1)$ be its rank-one multiplicative perturbation. Recall from Proposition \ref{prop:CMV3} that  the perturbed matrix $U^{[r]}$ has the same eigenvalues as the CMV matrix
\begin{equation}\label{eq:CMV_perturb_general}
\cC_n^{[r]}:=\cC(\rev \alpha_0, \dots, \rev \alpha_{n-2},r\rev \alpha_{n-1}). 
\end{equation}
Note that this matrix is a function of the spectral measure of $U$. In particular, for any probability measure $\mu_n$ on $\partial \D$ (supported on $n$ points) we can define the matrix $\cC_n^{[r]}$ from its Verblunsky coefficients. 

Suppose now that we have a sequence of probability measures $\mu_n, n\ge 1$, just as in  Proposition \ref{prop:general2}.
Let $\rev\gamma_{n-1}$ be the last modified Verblunsky coefficient corresponding to the Verblunsky coefficients $\rev \alpha_0, \dots, \rev \alpha_{n-1}$ of $\rev \mu_n$. (Note the abuse of notation: we should write $\rev \alpha_{k,n}, 0\le k\le n-1$ here, but we drop the extra $n$ from the notation.) The next result shows that if $\rev \gamma_{n-1}$ converges and the  assumptions of Proposition \ref{prop:general2} hold, then the eigenvalues of $\cC_n^{[r]}$ converge as well.

\begin{proposition}\label{prop:general3}
	Consider the same setup as in Proposition \ref{prop:general2}. Fix $r\in[0,1]$, let $\Psi_{n}^{[r]}(z)=\prod_{i=1}^n \frac{z-\lambda_i}{1-\lambda_i}$ be the normalized characteristic polynomial of $\cC_n^{[r]}$. Assume that the convergence \eqref{eq:Psi_conv_general} holds, and assume further that $\lim_{n\to\infty}\rev\gamma_{n-1} =\gamma\neq 1$. Then we have
	\begin{equation}\label{eq:conv1234}
	|\Psi_{n}^{[r]}(e^{iz/n})e^{-iz/2}-\cE^{[r]}(z)| \to 0, \qquad \cE^{[r]}(z)=
	 H(1,z)^{\dag}\binom{1}{-i\frac{1-r\gamma}{1+r\gamma}}
	\end{equation}
uniformly on compacts in $\C$ as $n\to\infty$. 
\end{proposition}
Note that \eqref{eq:conv1234} implies the convergence of the eigenvalues of $\Psi_{n}^{[r]}$ to the zeros of $\cE^{[r]}$ under the edge scaling \eqref{eq:edgescaling}.
\begin{proof}[Proof of Proposition \ref{prop:general3}]
By the modified Szeg\H o recursion \eqref{eq:Szego2}, we have
\begin{equation}\label{eq:Psi_Szego}
\Psi_{n}^{[r]}(z) = \tfrac{1}{1-r\rev\gamma_{n-1}}z\Psi_{n-1}^{[r]}(z)-\tfrac{r\rev\gamma_{n-1}}{1-r\rev\gamma_{n-1}}\Psi_{n-1}^{[r],*}(z),
\end{equation}
where $\Psi_{n-1}^{[r],*}(\cdot)$ is the reversed polynomial of $\Psi_{n-1}^{[r]}(\cdot)$. Note that the polynomials $\Psi_{n-1}^{[r]}$ and $\Psi_{n-1}^{[r],*}$ do not depend on $\rev\gamma_{n-1}$.

Introduce  $\cE^*(z)=\overline{\cE(\bar z)}$, by the definition of the reversed polynomials, \eqref{eq:Phi_H} and \eqref{eq:Psi_conv_general}, we have   \begin{equation}\label{eq:Psi_conv2}
\left|\Psi_{n-1}^{[r],*}(e^{iz/n})e^{-iz/2}-\cE^*(z)\right|\to 0\quad \text{uniformly on compacts as $n\to\infty$}.
\end{equation}
Together with \eqref{eq:Psi_Szego} and the convergence of $\rev\gamma_{n-1}\to \gamma$ as $n\to\infty$, we get
\begin{align*}
\lim_{n\to\infty}\Psi_{n}^{[r]}(e^{iz/n})e^{-iz/2}= \tfrac{1}{1-r\gamma}\cE(z)-\tfrac{r\gamma}{1-r\gamma}\cE^*(z) 
= H(1,z)^{\dag}\binom{1}{-i\frac{1-r\gamma}{1+r\gamma}},
\end{align*}
uniformly on compacts in $z$.
\end{proof}

\section{Edge limits of the truncated circular and real orthogonal beta ensembles}\label{sec:trCbE_trRbE}

Our goal is to prove Theorems \ref{thm:main} and \ref{thm:trunc_circ}, the edge scaling limits of the rank-one truncation and multiplicative perturbation of the circular beta ensemble using the framework developed in Section \ref{sec:general}. We will also prove the analogue results for the real orthogonal beta ensemble. 

In both cases we will check that the random measures associated to the respective beta ensembles (introduced in Section \ref{sub:finite}) satisfy the conditions in Propositions \ref{prop:general2} and \ref{prop:general3}.

\subsection{The truncated circular beta ensemble}\label{sub:truncated_circ}

Recall the definition of the random probability measure $\mu_{n,\beta}^{\textup{KN}}$ introduced in Section \ref{sub:finite}
and the corresponding CMV matrix $\Circ_{n,\beta}$.  Proposition \ref{prop:circ_Verblunsky} gives the distribution of the  Verblunsky coefficients $\alpha_k,0\le k\le n-1$ corresponding to $\mu_{n,\beta}^{\textup{KN}}$. The truncated circular beta ensemble is the joint distribution of the eigenvalues of the truncated CMV matrix $\tCircc_{n,\beta}$. The rank-one multiplicative perturbation of the circular beta ensemble (corresponding to a parameter $r\in [0,1]$) is defined as the distribution of the eigenvalues of $\Circ_{n,\beta}^{[r]}$.

Note that by Proposition \ref{prop:circ_Verblunsky} the Verblunsky coefficients $\alpha_k, 0\le k\le n-1$   are independent and rotationally invariant. Hence from \eqref{eq:rev_Verblunsky}
we get the following distributional identity:
\begin{equation}\label{eq:rev_Verblunsky_cbe}
(\rev \alpha_{0},\rev \alpha_{1},\dots,\rev \alpha_{n-2},\rev \alpha_{n-1})\ed  (\alpha_{n-2},\alpha_{n-3},\dots,\alpha_0,\alpha_{n-1}). 
\end{equation}
Using this together with Proposition \ref{prop:CMV3}, we recover the following result of Killip and Kozhan \cite{KK}.

\begin{proposition}[\cite{KK}]\label{prop:tcirc_matrix}
For fixed $n$ and $\beta>0$, let $\alpha_k,0\le k\le n-1$ be distributed according to Proposition \ref{prop:circ_Verblunsky}. Then the joint eigenvalue distribution of the CMV matrix  $\cC(\alpha_{n-2},\alpha_{n-3},\dots,\alpha_0)$ 
is the same as that of $\tCircc_{n,\beta}$, which is given by the truncated circular beta ensemble.
For fixed $r\in[0,1]$, the  matrix $\cC(\alpha_{n-2},\alpha_{n-3},\dots,\alpha_0,r\alpha_{n-1})$ has the same distribution as $\Circ_{n,\beta}^{[r]}$, in particular its joint eigenvalue distribution is given by the rank-one multiplicative perturbation of the circular beta ensemble.
\end{proposition}
Note that \cite{KK} also proved that the joint density of the truncated circular beta ensemble is given 
by \eqref{eq:tcbe}, and described explicitly the joint eigenvalue distribution of the perturbed matrix $\Circ_{n,\beta}^{[r]}$ (see Proposition 7.2 of \cite{KK}).

Recall the Dirac operator $\Circop_{n,\beta}$ in Definition \ref{def:Circop}. We denote by $\rev \tau_{n,\beta}$ 
the reversed version of $\Circop_{n,\beta}$, and by $\aff \tau_{n,\beta}$  the pulled-back version of $\rev{\tau}_{n,\beta}$ (see Definitions \ref{def:rev_Dirac} and \ref{def:aff_Dirac}).
By studying the generating paths of the operators $\aff \tau_{n,\beta}$ and $\Circop_{n,\beta}$, we claim that these operators are orthogonally equivalent. Indeed, this observation follows from Proposition 52 of \cite{BVBV_palm} using the rotational invariance of the models and a conditioning argument.
We will provide a different proof that holds for the circular Jacobi beta ensemble ($\delta=0$ corresponding to the circular beta ensemble). To avoid repetition, we will state the result without proof here.

\begin{proposition}\label{prop:circ_equiv}
Recall the definitions of the operators $\rho$ and $S$ from Lemma \ref{lem:timerev}.
Denote by $\aff b_k,0\le k\le n$ the path parameters of $\aff \tau_{n,\beta}$, and let 
\[
Q=\frac{1}{\sqrt{1+q^2}}\begin{pmatrix} q & 1\\ -1 & q\end{pmatrix},\qquad q=\cU^{-1}(\aff b_n).
\]
Then the operator
$\rho^{-1}(SQ)\aff \tau_{n,\beta}(SQ)^{-1}\rho$ has the same distribution as $\Circop_{n,\beta}$. 
\end{proposition}

Recall the definition of the  $\Sineop$ operator from  Definition \ref{def:Sineop}, and its reversed and transformed version $\tau_\beta$ defined in Definition \ref{def:Sineop_rev}. Proposition \ref{prop:circ_equiv}  shows that one can obtain the operators $\tau_{n,\beta}$ and $\tau_\beta$ from $\Circop_{n,\beta},n\ge 1$ and $\Sineop$ under the same orthogonal transformations. The final ingredient of proving Theorems \ref{thm:main} and \ref{thm:trunc_circ} is the following strong operator level convergence of $\Circop_{n,\beta}$ to the $\Sineop$ operator. Let $d_\HH(z_1,z_2)= \operatorname{arccosh}(1+\frac{|z_1-z_2|^2}{2\Im z_1\Im z_2})$ denote the hyperbolic distance in $\HH$.

\begin{proposition}
[\cite{BVBV_19,BVBV_szeta}]\label{prop:Sineb_coupling}
    There exists an explicit coupling of the operators $\Circop_{n,\beta}=\Dirop( z_n(\cdot),\uu_0,\uu_1^{(n)}),n\ge 1$  and $\Sineop=\Dirop( z(\cdot),\uu_0,\uu_1)$ such that $\uu_1^{(n)}=\uu_1$, and for large enough $n$,	\begin{equation}\label{eq:circ_path_dist}
	\begin{aligned}
	d_{\HH} ( z_n(t), z(t))&\le \frac{\log^{3-1/8}n}{\sqrt{(1-t)n}},\quad \quad 0\le t\le t_n:=1-\frac{\log^6 n}{n},\\
	d_{\HH} ( z_n(t_n),z(t))&\le (\log\log n)^4,\quad\quad  t_n\le t< 1.
	\end{aligned}
	\end{equation}
	 Under this coupling, we have
		\begin{equation}\label{eq:circ_coupling1}
	\|\res \Circop_{n,\beta}-\res \Sineop\|_{\text{HS}}^2 \le \frac{\log^6 n}{n},\quad \quad |\ttr_{\Circop_{n,\beta}}-\ttr_{\Sineop}|\le \frac{\log^3n}{\sqrt{n}}.
    \end{equation}
\end{proposition}
Here the first inequality of \eqref{eq:circ_coupling1} was proved in \cite{BVBV_19} and the second one follows from the estimates used in the proof of Theorem 49 of \cite{BVBV_szeta} (in particular equations (107)-(109)). `Large enough $n$' means that there is a finite random variable $N_0$ so that the statements hold for $n\ge N_0$.

Proposition \ref{prop:Sineb_coupling} provides us with the necessary ingredients so that we can apply Proposition \ref{prop:general2} to prove the convergence of the normalized characteristic 
polynomials of the truncated circular beta ensemble.

\begin{proposition}\label{prop:tcirc_func} 
Let $\lambda_i,1\le i\le n-1$ be the size $n-1$ truncated circular beta ensemble and set 
$p_{n-1,\beta}(z)=\prod_{i=1}^{n-1} \frac{z-\lambda_i}{1-\lambda_i}$ be the normalized characteristic polynomial. Let $\cE_\beta$ be the structure function of $\tau_\beta$ defined via \eqref{def:structure}.
Then there is a coupling of $p_{n-1,\beta},n\ge 2$ and $\cE_\beta$ such that
	\begin{equation}\label{eq:coupled_diff}
	|p_{n-1,\beta}(e^{iz/n})e^{-iz/2}-\cE_\beta(z)|\le \left(e^{|z|\frac{\log^3 n}{\sqrt{n}}}-1\right)C^{1+|z|^2}
	\end{equation}
	for all $z\in\C$ and $n\ge 1$, where $C$ is an a.s.~finite constant. 
\end{proposition}
\begin{proof}
Let $\rev \mu_n$ be the reversed version of the Killip-Nenciu probablity measures $\mu_n:=\mu_{n,\beta}^{\textup{KN}}$, and let $\aff \tau_{n,\beta}$ be the pulled-back operator. By Propositions \ref{prop:CMV1} and \ref{prop:tcirc_matrix}, the normalized characteristic polynomial $p_{n-1,\beta}$  has the same distribution as $\rev \Psi_{n-1,\beta}$, the monic orthogonal polynomial of degree $n-1$ associated to $\rev \mu_n$. Hence it is enough to provide a coupling of $\mu_n$ and $\cE_\beta$ where \eqref{eq:coupled_diff} holds with $\rev \Psi_{n-1,\beta}$ in place of $p_{n-1,\beta}$.

By Proposition \ref{prop:general1}, we have 
\[
\rev \Psi_{n-1,\beta}(e^{iz/n})=e^{iz(n-1)/(2n)}\aff H_{n,\beta}((n-1)/n,z)^{\dag}\binom{1}{-i},
\]
where $\aff H_{n,\beta}$ solves the ODE \eqref{eq:canonical} of $\aff \tau_{n,\beta}$. Recall that $\cE_\beta(z)=H_\beta(1,z)^{\dag}\binom{1}{-i}$.

Consider the coupling of Proposition \ref{prop:Sineb_coupling}. Using the triangle inequality as in \eqref{ineq:circ_conv_tri} within the proof of Proposition \ref{prop:general2}, we have
\begin{equation}\label{eq:diff_1234}
\begin{aligned}
    |\rev \Psi_{n-1,\beta}(e^{iz/n})e^{-iz/2}-\cE_\beta(z)| & \le  |e^{-iz/(2n)}|\left|\Big(\aff H_{n,\beta}(\tfrac{n-1}{n},z)^{\dag}- H_\beta(\tfrac{n-1}{n},z)^{\dag}\Big)\cdot\binom{1}{-i}\right|\\
& + |e^{-iz/(2n)}|\left|\Big( H_\beta(\tfrac{n-1}{n},z)^{\dag}- H_\beta(1,z)^{\dag}\Big)\cdot\binom{1}{-i}\right|\\
& +\left|e^{-iz/(2n)}-1\right|\Big|H_\beta(1,z)^{\dag}\cdot\binom{1}{-i}\Big|. 
\end{aligned}
\end{equation}
Note that for any compact subset of $(0,1]$ the operator norm of the weight function $R(s)$ of $\tau_\beta$ is bounded by a finite random constant. Hence by the standard theory of ordinary differential equations, for $z$ in a compact set of $\C$, the second  term on the right hand side of \eqref{eq:diff_1234} can be upper bounded by $Cn^{-1}$ where $C$ is an a.s. finite constant. The third term on the right side of \eqref{eq:diff_1234} can also be bounded similarly, since $z\mapsto H_\beta(1,z)^{\dag}\cdot\binom{1}{-i}$ is a random entire function.

It remains to estimate the first term on the right hand side of \eqref{eq:diff_1234}. By Propositions \ref{prop:Sineop_equiv} and \ref{prop:circ_equiv}, the operators $\aff \tau_{n,\beta},n\ge 1$ and $\tau_{\beta}$ can be obtained from $\Circop_{n,\beta},n\ge 1$ and $\Sineop$ under the same orthogonal transformations. Hence in the coupling of Proposition \ref{prop:Sineb_coupling}, for large enough $n$ we have
	\begin{align}\label{eq:error_123}
	\|\res\tau_{\beta}-\res\aff \tau_{n,\beta}\|\le \frac{\log^3n}{\sqrt{n}},\quad \quad  |\ttr_{\tau_{\beta}}-\ttr_{\aff\tau_{n,\beta}}|\le \frac{\log^3n}{\sqrt{n}}, 
	\end{align}
It has also been shown in the proof of Proposition 3 of \cite{BVBV_palm}  (see equations (69)-(72) therein) that
\begin{align}\label{eq:error_234}
\int_0^1 |\mathfrak{a}_{0,n}(s)-\mathfrak{a}_0(s)|^2 ds \le \frac{\log^6n}{n}.
\end{align}
The estimates \eqref{eq:error_123} and \eqref{eq:error_234} allow us to use Proposition \ref{prop:H_comparison} 
to bound $\|\aff H_{n,\beta}(\tfrac{n-1}{n},z)-H_\beta(\tfrac{n-1}{n},z)\|$ for $n$ large enough. From this it follows that \eqref{eq:diff_1234} can be bounded from above according to \eqref{eq:coupled_diff},  which proves the proposition. 
\end{proof}

We are now ready to prove Theorems \ref{thm:main} and \ref{thm:trunc_circ}.

\begin{proof}[Proof of Theorem \ref{thm:main}]
First note that by  \eqref{eq:cH_beta} we have $\cH_\beta(0,\cdot)=H_\beta(1,\cdot)$, hence $\cE_\beta(\cdot)=\cH_\beta(0,\cdot)^{\dag}\binom{1}{-q}$. Then by Proposition \ref{prop:tcirc_func}, we obtain the convergence of the normalized characteristic polynomials of the truncated circular beta ensemble to the random analytic function $\cE_\beta$ with an explicit error bound.
Under the edge scaling \eqref{eq:edgescaling}, the size $n$ truncated circular beta ensemble has the same distribution as   the zeros of the function $p_{n,\beta}(e^{iz/n})$ defined in Proposition \ref{prop:tcirc_func}. 
The weak convergence of the truncated circular beta ensemble to $\mathcal{X}_\beta:=\{z\in\HH:\cE_\beta(z)=0\}$ follows directly from Proposition \ref{prop:tcirc_func} and Hurwitz’s theorem.
\end{proof}

\begin{proof}[Proof of Theorem \ref{thm:trunc_circ}]
    Consider the coupling of $\tau_{n,\beta}$ and $\tau_\beta$ described in Proposition \ref{prop:tcirc_func} again, the right boundary condition of these operators are coupled together. By Propositions \ref{prop:general3} and \ref{prop:tcirc_func} we get the uniform-on-compacts convergence of the normalized characteristic polynomials of $\Circ_{n,\beta}^{[r]}$ to $\cE_{r,\beta}$ with a similar error bound. 
    This in turn gives the weak convergence of the eigenvalues of $\Circ_{n,\beta}^{[r]}$ under the edge scaling and completes the proof. 
\end{proof}

\subsection{Edge limits of the truncated real orthogonal beta ensemble}\label{sub:truncated_RO}

This section discusses the edge limits of the rank-one truncation and multiplicative perturbation of the real orthogonal beta ensemble. We will follow the same approach as in Section \ref{sub:truncated_circ}. 

Recall the size $2n$ real orthogonal ensemble with joint density \eqref{eq:PDF_ortho} and the random probability measure $\mu_{2n,\beta,a,b}^{\textup{RO}}$ in \eqref{eq:RO_measure}. Proposition \ref{prop:RO_Verblunsky} describes the distribution of the corresponding Verblunsky coefficients $\alpha_k,0\le k\le 2n-1$. Since $\alpha_k$'s are all real and $\alpha_{2n-1}=-1$, from \eqref{eq:rev_Verblunsky} we get that the reversed Verblunsky coefficients $\rev \alpha_{k},0\le k\le 2n-1$ satisfy
\begin{equation}\label{eq:RO_rev_Verblunsky}
    (\rev \alpha_{0},\rev \alpha_{1},\cdots,\rev \alpha_{2n-2},\rev \alpha_{2n-1}) = (\alpha_{2n-2},\alpha_{2n-3},\cdots,\alpha_0,\alpha_{2n-1}).
\end{equation}
The rank-one truncation and multiplicative perturbation of the real orthogonal beta ensemble are defined as the joint eigenvalue distributions of the truncated CMV matrix $\trunc\RO_{2n,\beta,a,b}$ and the perturbed CMV matrix $\RO_{2n,\beta,a,b}^{[r]}$ (indexed by $r\in[0,1])$, respectively. Using \eqref{eq:RO_rev_Verblunsky} with Proposition \ref{prop:CMV3} yields the following result of Killip and Kozhan \cite{KK}.

\begin{proposition}[\cite{KK}]\label{prop:matrix_tRO}
 For given $\beta>0$, $a, b>-1$ and $n\ge 1$, let $\alpha_k,0\le k\le 2n-1$ be distributed according to Proposition \ref{prop:RO_Verblunsky}. Then the CMV matrix $\cC(\alpha_{2n-2},\alpha_{2n-3}\dots,\alpha_{0})$ has the same joint eigenvalue distribution as $\trunc{\RO}_{2n,\beta,a,b}$.
For fixed $r\in[0,1]$, the CMV matrix $\cC(\alpha_{2n-2},\dots,\alpha_{0},-r)$ has the same joint eigenvalue distribution as $\RO_{2n,\beta,a,b}^{[r]}$.
\end{proposition}
Note that the joint distributions of the rank-one truncation and multiplicative perturbation of the real orthogonal beta ensemble were described explicitly in Theorem 6.4 and Proposition 7.2 (b) of \cite{KK}.

Recall the Dirac operator $\ROop_{2n,\beta,a,b}$ in Definition \ref{def:ROop}. Denote by $\rev\tau_{2n,\beta,a,b}$ the reversed version of $\ROop_{2n,\beta,a,b}$, and by $\aff\tau_{2n,\beta,a,b}$ the pulled-back version of $\rev\tau_{2n,\beta,a,b}$.
Recall also the limiting $\Bessop$ operator in Definition \ref{def:Bessop} and its reversed version $\tau_{\beta,a}$ in Definition \ref{def:Bess_rev}.  By Proposition \ref{prop:Bessop_equiv} we get that $\rho^{-1}J\tau_{\beta,a}J\rho\ed \Bessop$. The next result shows that under the same transformations, the operators $\aff\tau_{2n,\beta,a,b}$ and $\ROop_{2n,\beta,a,b}$ are orthogonally equivalent.

\begin{proposition}\label{prop:ROop_equiv}
\[
\rho^{-1} J \aff \tau_{2n,\beta,a,b}J\rho \ed \ROop_{2n,\beta,a,b}. 
\]
\end{proposition}
\begin{proof}
   Let $\alpha_k,0\le k\le 2n-1$ be distributed as in Proposition \ref{prop:RO_Verblunsky}, and let $\gamma_k,0\le k\le 2n-1$ be the corresponding modified Verblunsky coefficients.
    Since $\alpha_k$'s are all real, from \eqref{eq:modifiedV} we have $\gamma_k=\alpha_k$ for all $0\le k\le 2n-1$. 
    
    Let $z_k,0\le k\le 2n$ and $\rev z_k,0\le k\le 2n$ be the path parameters of $\ROop_{2n,\beta,a,b}$ and $\rev\tau_{2n,\beta,a,b}$ in $\HH$, respectively.  Using \eqref{eq:z_rec} and \eqref{eq:RO_rev_Verblunsky} we get for 
    \[
    z_k=i\prod_{j=0}^{k-1} \frac{1+\gamma_j}{1-\gamma_j},\qquad \rev z_k=i\prod_{j=0}^{k-1}\frac{1+\gamma_{2n-2-j}}{1-\gamma_{2n-2-j}},\qquad 0\le k\le 2n.
    \]
    Recall the affine transformation $\cA_{z,\HH}$ in \eqref{e:A1def} and the corresponding transformation on the uni-disk model $\cA_{\gamma,\D}$ in \eqref{eq:Adef}. By Definition \ref{def:aff_Dirac} and \eqref{eq:A_equiv}, the path parameters of $\aff\tau_{2n,\beta,a,b}$ in $\HH$, denoted by $\aff z_k,0\le k\le 2n$, are obtained from $\rev z_k,0\le k\le 2n$ via the affine transformation $\cA_{\rev z_{2n-1},\HH}$. More precisely, we have
    \[
    \aff z_k = \cP \mat{1}{0}{0}{\Im\rev z_{2n-1}}\binom{\rev z_k}{1} = i\prod_{j=0}^{2n-2-k}\frac{1-\gamma_j}{1+\gamma_j},\qquad 0\le k\le 2n-1,
    \]
    with $\aff z_{2n}=\cA_{\rev z_{2n-1},\HH}(\rev z_{2n})=0$.
    
Note  that conjugating  with the permutation
matrix $J$ maps $z\mapsto -1/z$ for $z\in i\R$. Together with the time-reversal, we see that the path parameters of the operator $\rho J \aff \tau_{2n,\beta,a,b}J\rho$ are the same as the path parameters of $\ROop_{2n,\beta,a,b}$. By Lemmas \ref{lem:timerev} and \ref{lem:rotation}, the left and right boundary points of $\rho^{-1}J\aff \tau_{2n,\beta,a,b}J\rho$ are given by $1$ and $-1$, which corresponds to the vectors $\uu_0=\binom{1}{0}$ and $\uu_1=\binom{0}{-1}$ as desired. This completes the proof. 
\end{proof}

In the rest of the section, we aim to prove the convergence of the normalized characteristic polynomials of the truncated real orthogonal beta ensemble using Proposition \ref{prop:general2}. The main ingredient will be the operator level convergence of $\ROop_{2n,\beta,a,b}$ to its limit $\Bessop$ proved in \cite{LV}. 

\begin{proposition}[\cite{LV}]\label{prop:Bess_limit}
	Let $\uu_0=\binom{1}{0},\uu_1=\binom{0}{-1}$. There exists a coupling of the operators $\ROop_{2n,\beta,a,b}=\Dirop(iy_n(\cdot),\uu_0,\uu_1),n\ge 1$ and $\Bessop=\Dirop(iy(\cdot),\uu_0,\uu_1)$ such that as $n\to\infty$ we have almost surely $|y_n-y|\to 0$ point-wise on $[0,1)$, and 
 	\[
	\|\res \ROop_{2n,\beta,a,b}-\res \Bessop\| \to 0.
	\]
 Moreover, for $\eps>0$ small there exists a sequence of tight random variables $\kappa_n$ and an a.s.~finite random variable $\kappa>0$ such that for $0\le t<1$
	\begin{equation}\label{eq:RO_y}
	\kappa_n^{-1}(1-\lfloor nt\rfloor/n)^{2a+1+\eps}\le y_n(t)\le \kappa_n(1-\lfloor nt\rfloor/n)^{2a+1-\eps},
	\end{equation}
	and similarly, 
	\begin{equation}\label{eq:Bess_y}
	\kappa^{-1}(1-t)^{2a+1+\eps}\le y(t)\le \kappa(1-t)^{2a+1-\eps}.
	\end{equation}
\end{proposition}

Now we are ready to prove the main result of this section. 
\begin{theorem}\label{thm:RO_main}
	For fixed $\beta>0,a,b>-1$ and $n\ge 1$, let $\lambda_i,1\le i\le 2n-1$ be the size $(2n-1)$ truncated real orthogonal beta ensemble and set $p_{2n-1,\beta,a,b}(z)=\prod_{i=1}^{2n-1}\frac{z-\lambda_i}{1-\lambda_i}$ be the normalized characteristic polynomial. Let $\cE_{\beta,a}$ be the structure function of  $\tau_{\beta,a}$ defined via \eqref{def:structure}. Then there is a coupling of $p_{2n-1,\beta,a,b},n\ge 1$ and $\cE_{\beta,a}$ such that almost surely 
	\begin{equation}\label{eq:tRO_conv}
	\left|p_{2n-1,\beta,a,b}(e^{iz/(2n)})e^{-iz/2}-\cE_{\beta,a}(z)\right|\to 0 \quad \text{ uniformly on compacts as $n\to\infty$.}
	\end{equation}
	 Consequently, the truncated real orthogonal beta ensembles converge weakly to the zeros of $\cE_{\beta,a}(\cdot)$ under the edge scaling \eqref{eq:edgescaling} as $n\to\infty$. 
\end{theorem}
\begin{proof}
Let $\rev \mu_{2n}$ be the reversed version of the random probability measures $\mu_{2n}:=\mu_{2n,\beta,a,b}^{\textup{RO}}$, and let $\aff\tau_{2n,\beta,a,b}$ be the pulled-back operator. 
By Propositions \ref{prop:CMV1} and \ref{prop:matrix_tRO}, the normalized characteristic polynomial $p_{2n-1,\beta,a,b}$ has the same distribution as the monic orthogonal polynomial of degree $2n-1$ associated to $\rev \mu_{2n}$. Hence it is enough to show \eqref{eq:tRO_conv} holds with $\rev\Psi_{2n-1,\beta,a,b}$ in place of $p_{2n-1,\beta,a,b}$. 

By Proposition \ref{prop:general1}, we have 
\[
\rev\Psi_{2n-1,\beta,a,b}(e^{iz/(2n)})=e^{iz(2n-1)/(4n)}\aff H_{2n,\beta,a,b}((2n-1)/(2n),z)^{\dag}\binom{1}{-i},
\]
where $\aff H_{2n,\beta,a,b}$ solves the ODE \eqref{eq:canonical} of $\aff \tau_{2n,\beta,a,b}$. Recall that $\cE_{\beta,a}=H_{\beta,a}(1,\cdot)^{\dag}\binom{1}{-i}$. It suffices to provide a coupling of $\mu_{2n}$ and $\cE_{\beta,a}$ under which the uniform-on-compacts convergence of $\aff H_{2n,\beta,a,b}(1,z)$ to $H_{\beta,a}(1,z)$ holds. 

By Propositions \ref{prop:Bessop_equiv} and \ref{prop:ROop_equiv}, the operators $\aff \tau_{2n,\beta,a,b},n\ge 1$ and $\tau_{\beta,a}$ can be obtained from $\ROop_{2n,\beta,a,b},n\ge 1$ and $\Bessop$ under the same orthogonal transformations. 
Therefore, in the coupling of Proposition \ref{prop:Bess_limit}, we have 
\begin{equation}\label{eq:RO_diff123}
\|\res \aff\tau_{2n,\beta,a,b}-\res\tau_{\beta,a}\|\to 0,\quad \quad \ttr_{\aff\tau_{2n,\beta,a,b}}=\ttr_{\tau_{\beta,a}}=0.
\end{equation}
By the triangle inequality, the path bounds \eqref{eq:RO_y} and \eqref{eq:Bess_y}, and a standard subsequence argument, we have
$\lim_{n\to\infty}\int_0^1 |\hat y(s)^{-1/2}-\hat y_n(s)^{-1/2}|^2 ds=0$. Together with \eqref{eq:RO_diff123}, this verifies the conditions in \eqref{eq:discrete_assumption}. Using Proposition \ref{prop:general2} completes the proof of the proposition.
\end{proof}

Note that by Proposition \ref{prop:HP_equiv} the limit random analytic function $\cE_{\beta,a}$ can also be characterized as $\cE_{\beta,a}=\cH_{\beta,a}(0,z)^{\dag}\binom{1}{-i}$, where $\cH_{\beta,a}$ solves the SDE \eqref{eq:HP_cH}. 

Using Proposition \ref{prop:general3} and Theorem \ref{thm:RO_main}, one obtains the following result on the convergence of normalized characteristic polynomial of the perturbed matrix $\rRO_{2n,\beta,a,b}$.
\begin{corollary}\label{cor:RO_perturb}
For fixed $r\in[0,1]$ let $\Lambda_{2n}=\{\lambda_1,\dots,\lambda_{2n}\}$ be the size $2n$ rank-one multiplicative perturbed real orthogonal ensemble and set
$
p_{2n,\beta,a,b}^{[r]}(z) = \prod_{i=1}^{2n} \frac{z-\lambda_i}{1-\lambda_i}$ 
to be the normalized characteristic polynomial. Then under the coupling of Proposition \ref{prop:Bess_limit}, we have almost surely  as $n\to\infty$
\begin{align*}
\left|\Psi_{2n,\beta,a,b}^{[r]}(e^{iz/n})e^{-iz/(2n)}-\cE_{\beta,a}^{[r]}(z)\right|\to 0,\quad \text{ uniformly on compacts in $\C$,}
\end{align*}
where $\cE_{\beta,a}^{[r]}(z)= H_{\beta,a}(1,z)^{\dag}\binom{1}{-i\frac{1-r}{1+r}}$. Consequently, the rank-one multiplicative perturbed real orthogonal beta ensemble $\Lambda_{2n}$ converges weakly to the zero set of the random analytic function $\cE_{\beta,a}^{[r]}$ under the edge scaling \eqref{eq:edgescaling} as $n\to\infty$.
\end{corollary}
Recall that the secular function $\zeta_{\beta,a}$ of the $\Bessop$ operator is given by $\zeta_{\beta,a}= H_{\beta,a}(1,z)^{\dag}\binom{1}{0}$. The above result gives an interpolation between the scaling limits of the normalized characteristic polynomials of the unperturbed and the truncated real orthogonal beta ensemble.

\section{The truncated circular  Jacobi beta ensemble} \label{sec:trCJbE}

In Section \ref{sub:CJ_finite} we construct the 
truncated and the multiplicative perturbed circular Jacobi beta ensemble, and we compute the joint eigenvalue density of the truncated model. 
In Section \ref{sub:CJ_limit}, we derive the edge scaling limits of the truncated and perturbed models using Propositions \ref{prop:general2} and \ref{prop:general3}. 

\subsection{Matrix model and joint eigenvalue density for truncated circular Jacobi beta ensemble}\label{sub:CJ_finite}

Consider the random probability measure $\mu_{n,\beta,\delta}^{\CJb}$ in \eqref{eq:CJ_measure}, with support given by the circular Jacobi beta ensemble. Recall also the matrix model $\CJ_{n,\beta,\delta}$ defined in Proposition \ref{prop:CJ_mVerblunsky} via the sequence of regular Verblunsky coefficients $\alpha_k,0\le k\le n-1$ of $\mu_{n,\beta,\delta}^{\CJb}$. 

\begin{definition}\label{def:trunc_CJE} For  fixed $n\ge 1$, $\beta>0$, $\Re \delta>-1/2$ we define the \emph{truncated circular Jacobi beta ensemble} as the joint eigenvalue distribution of the truncated matrix $\trunc{\CJ}_{n,\beta,\delta}$. For a fixed $r\in[0,1]$ we define the perturbed circular Jacobi beta ensemble as the joint eigenvalue distribution of  the perturbed matrix  $\CJ_{n,\beta,\delta}^{[r]}:=\CJ_{n,\beta,\delta}\cdot\diag(r,1,1,\dots,1)$.    
\end{definition}

The main challenge to study these ensembles is that the Verblunsky coefficients of $\mu_{n,\beta,\delta}^{\CJb}$ are not independent, hence one cannot expect a nice description of the CMV matrices appearing in Proposition \ref{prop:CMV3}. 
However, as the next proposition shows, we can still preserve the independence by expressing the appearing CMV matrices in terms of the \emph{modified} Verblunsky coefficients.  

Recall the definition of modified Verblunsky coefficients given by the recursion \eqref{eq:modifiedV}. The recursion provides a one-to-one map between the first $k\le n-1$ Verblunsky coefficients and the first $k$ modified Verblunsky coefficients, we denoted this map by $\cT_k$.

\begin{proposition}\label{prop:tCJ_matrix}
Fix $\beta>0,\delta\in\C$ with $\Re\delta>-1/2$, and let $\gamma_k,0\le k\le n-1$ be the sequence of modified Verblunsky coefficients of $\mu_{n,\beta,\delta}^{\CJb}$.
Then the sub-unitary CMV matrix $\cC(\cT_{n-1}^{-1}(\gamma_{n-2},\gamma_{n-3}\cdots,\gamma_{0}))$ has the same joint eigenvalue distribution as the truncated model $\trunc{\CJ}_{n,\beta,\delta}$. 
For fixed $r\in[0,1]$, the matrix $\cC(\cT_{n}^{-1}(\gamma_{n-2},\gamma_{n-3}\dots,\gamma_{0},r\gamma_{n-1}))$ has the same joint eigenvalue distribution as the perturbed model $\CJ_{n,\beta,\delta}^{[r]}$. 
\end{proposition}

Before proving the proposition, we introduce a simple mapping  on $\D$. 
 For $\gamma\in\D$ recall the linear fractional transformation $\cA_{\gamma,\D}$ from \eqref{eq:Adef}. This is an isometry of the Poincar\'e disk that corresponds to an affine transformation in the Poincar\'e half-plane $\HH$. 
 The inverse of $\cA_{\gamma,\D}$ is also an isometry, and it also corresponds to an affine transformation in $\HH$, we denote the corresponding element in $\D$ by $\gamma^{\iota}$:
\begin{align}\label{eq:gamma_iota_def}
    \cA_{\gamma,\D}^{-1}=\cA_{\gamma^\iota,\D},\qquad \gamma^{\iota} = -\gamma\frac{1-\bar \gamma}{1-\gamma}.
\end{align}
Note that the $\iota: \gamma\mapsto \gamma^\iota$ mapping is an involution such that $\mathcal{A}_{\gamma,\D}(0)=\gamma^\iota$. 

We also need the following distributional identity for $\Theta(a+1,\delta)$ random variables (see Definition \ref{def:Theta_delta}).

\begin{claim}\label{claim:CJ}
	Fix $\delta\in\C$ with $\Re\delta>-1/2$. Let $\zeta_0,\zeta_1,\dots,\zeta_{n-1}$ be a sequence of independent random variables such that $\zeta_i\sim\Theta(a_i+1,\delta)$, with $a_0=0$ and $a_i\ge0$ for $i\ge 1$. Then 
	\begin{equation}\label{eq:gamma_ratio}
	\left(\frac{\zeta_1^\iota}{\zeta_0},\frac{\zeta_2^\iota}{\zeta_1},\dots\frac{\zeta_{n-1}^\iota}{\zeta_{n-2}},\frac{1}{\zeta_{n-1}}\right)\ed \left(\bar\zeta_1,\frac{\bar\zeta_{2}}{\bar\zeta_{1}^\iota},\dots,\frac{\bar\zeta_{n-1}}{\bar\zeta_{n-2}^\iota},\frac{\bar\zeta_0}{\bar\zeta_{n-1}^\iota}\right).
	\end{equation}
\end{claim}

\begin{proof}
Our statement will follow from the following (simpler) distributional identity.  Let $\zeta\sim\Theta(a+1,\delta)$ and $\eta\sim\Theta(1,\delta)$ be independent. Then 
	\begin{equation}\label{eq:CJ_swapping}
	(\eta\bar\zeta^\iota,\zeta) \ed (\zeta,\eta\bar\zeta^\iota).
	\end{equation}
 Since $|\zeta^\iota|=|\zeta|$, it is sufficient to prove that the unit length random variables $\eta\bar\zeta^{\iota}/|\zeta|$ and $\zeta/|\zeta|$ are conditionally exchangeable given 
$|\zeta|=r$. Let $z=\zeta/|\zeta|$, then under the condition $|\zeta|=r$ we have 
\[
\eta\bar\zeta^{\iota}/|\zeta|=-\eta \frac{1-rz}{z-r}.
\]
By the independence of $\eta$ and $\zeta$,  the conditional joint  density of $(\eta,z)$ (given $|\zeta|=r$) is proportional 
	\[
	 (1-\eta)^{\bar\delta}(1-\eta^{-1})^\delta(1-rz)^{\bar\delta}(1-rz^{-1})^\delta.
	\]
Since $z\mapsto \frac{1-rz}{z-r}$ is an isometry of the unit circle, the Jacobian  of the mapping from $\left(-\eta \frac{1-rz}{z-r},z\right)\mapsto (\eta,z)$ to  is equal to $1$.
	Therefore, the conditional joint density of $(\xi_1,\xi_2):=\left(-\eta \frac{1-rz}{z-r},z\right)$ (given $|\zeta|=r)$ is proportional to 
	\[
	(1-r(\xi_1+\xi_2)+\xi_1\xi_2)^{\bar\delta}(1-r(\xi_1^{-1}+\xi_2^{-1})+\xi_1^{-1}\xi_2^{-1})^\delta.
	\]
	This shows the (conditional) exchangeability $(\xi_1,\xi_2)=(\eta\bar\zeta^{\iota}/|\zeta|,\zeta/|\zeta|)$ and proves \eqref{eq:CJ_swapping}.

 Now we turn to the proof of the statement. Note that \eqref{eq:CJ_swapping} implies that \[(\zeta_1^\iota/\zeta_0,\bar \zeta_1)\ed(\bar\zeta_1,\zeta_1^\iota/\zeta_0).\]
	Starting from the random vector on the left-hand side of \eqref{eq:gamma_ratio}, we apply \eqref{eq:CJ_swapping} repeatedly and get
	\begin{align*}
	\left(\frac{\zeta_1^\iota}{\zeta_0},\frac{\zeta_2^\iota}{\zeta_1},\dots\frac{\zeta_{n-1}^\iota}{\zeta_{n-2}},\frac{1}{\zeta_{n-1}}\right)&\ed  \left(\bar\zeta_1,\frac{\bar\zeta_0\zeta_2^\iota}{\bar\zeta_1^\iota},\dots\frac{\zeta_{n-1}^\iota}{\zeta_{n-2}},\frac{1}{\zeta_{n-1}}\right)\ed\dots\\
	&\ed  \left(\bar\zeta_1,\frac{\bar\zeta_{2}}{\bar\zeta_{1}^\iota},\dots,\frac{\bar\zeta_0\zeta_{n-1}^\iota}{\bar\zeta_{n-2}^\iota},\frac{1}{\zeta_{n-1}}\right)\ed \left(\bar\zeta_1,\frac{\bar\zeta_{2}}{\bar\zeta_{1}^\iota},\dots,\frac{\bar\zeta_{n-1}}{\bar\zeta_{n-2}^\iota},\frac{\bar\zeta_0}{\bar\zeta_{n-1}^\iota}\right),
	\end{align*}
proving \eqref{eq:gamma_ratio}.
\end{proof}

Now we return to the proof of Proposition \ref{prop:tCJ_matrix}.

\begin{proof}[Proof of Proposition \ref{prop:tCJ_matrix}]
Let $\alpha_k, 0\le k\le n-1$ 
be the (regular) Verblunsky coefficients of $\mu_{n,\beta,\delta}^{\CJb}$ and set $\rev \alpha_k,0\le k \le n-1$ to be the reversed version of $\alpha_k,0\le k\le n-1$ defined via \eqref{eq:rev_Verblunsky}.
By Proposition \ref{prop:CMV3}, the truncated and perturbed models $\trunc{\CJ}_{n,\beta,\delta}$ and $\CJ_{n,\beta,\delta}^{[r]}$ have the same eigenvalues as the  CMV matrices $\cC(\rev \alpha_{0},\rev \alpha_{1},\cdots,\rev \alpha_{n-2})$ and $\cC(\rev \alpha_{0},\rev \alpha_{1},\cdots,\rev \alpha_{n-2},r\rev \alpha_{n-1})$, respectively. Hence, the statement of the proposition follows  if we can show  
\begin{equation}\label{eq:CJ_dist1}
    (\rev \alpha_{0},\rev \alpha_{1},\dots,\rev \alpha_{n-2},\rev \alpha_{n-1})
    \ed \cT_{n}^{-1}(\gamma_{n-2},\gamma_{n-3},\cdots,\gamma_0,\gamma_{n-1}).
\end{equation}
Introduce the  temporary notation 
\[(\hat \alpha_0,\hat\alpha_1,\dots,\hat\alpha_{n-1})=\cT_n^{-1}(\gamma_{n-2},\gamma_{n-3},\dots,\gamma_0, \gamma_{n-1}).\]
We need to show $(\rev \alpha_{0},\rev \alpha_{1},\dots,\rev \alpha_{n-2},\rev \alpha_{n-1})\ed (\hat \alpha_0,\hat\alpha_1,\dots,\hat\alpha_{n-1})$, which will follow from 
\begin{equation*}	
\left(\rev \alpha_{0},-\frac{\rev \alpha_{1}}{\rev \alpha_{0}}\dots,-\frac{\rev \alpha_{n-1}}{\rev \alpha_{n-2}}\right)\ed \left(\hat\alpha_0,-\frac{\hat\alpha_1}{\hat\alpha_0}\dots,-\frac{\hat\alpha_{n-1}}{\hat\alpha_{n-2}}\right).
\end{equation*}
From \ref{eq:modifiedV} and \eqref{eq:gamma_iota_def} it follows that
\begin{align}\label{eq:alpha_ratio}
  - \frac{\alpha_k}{\alpha_{k-1}}= \frac{\bar \gamma_k}{\bar \gamma_{k-1}^{\iota}}, \quad k\ge1, \qquad \alpha_0=\bar \gamma_0. 
\end{align}
Hence
\[
\left(\hat\alpha_0,-\frac{\hat\alpha_1}{\hat\alpha_0}\dots,-\frac{\hat\alpha_{n-1}}{\hat\alpha_{n-2}}\right)=\left(
\bar \gamma_{n-2}, \frac{\bar \gamma_{n-3}}{\bar \gamma_{n-2}^{\iota}},\dots,
\frac{\bar \gamma_0}{\bar \gamma_1^\iota}, \frac{\bar \gamma_{n-1}}{\bar \gamma_0^{\iota}}
\right).
\]
On the other hand, from \ref{eq:rev_Verblunsky} and \eqref{eq:alpha_ratio} we have 
\[
\left(\rev \alpha_{0},-\frac{\rev \alpha_{1}}{\rev \alpha_{0}}\dots,-\frac{\rev \alpha_{n-1}}{\rev \alpha_{n-2}}\right)=\left(
-\frac{\bar \alpha_{n-2}}{\bar \alpha_{n-1}}, -\frac{\bar \alpha_{n-3}}{\bar \alpha_{n-2}}, \dots, -\frac{\bar \alpha_0}{\bar \alpha_1}, \frac{1}{\bar \alpha_0}
\right)=\left(
\frac{\gamma_{n-2}^{\iota}}{\gamma_{n-1}},\frac{\gamma_{n-3}^{\iota}}{\gamma_{n-2}},\dots, \frac{\gamma_0^\iota}{\gamma_1},\frac{1}{\gamma_0}
\right)
\]
From Proposition \ref{prop:CJ_mVerblunsky} we know that $\gamma_k,0\le k\le n-1$ are independent with $\gamma_k\sim \Theta(\beta(n-k-1)+1),\delta)$. 
Using Claim  \ref{claim:CJ} with $\zeta_j=\gamma_{n-1-j}, 0\le j\le n-1$ we get
\[
\left(
\frac{\gamma_{n-2}^{\iota}}{\gamma_{n-1}},\frac{\gamma_{n-3}^{\iota}}{\gamma_{n-2}},\dots, \frac{\gamma_0^\iota}{\gamma_1},\frac{1}{\gamma_0}
\right)\ed \left(
\bar \gamma_{n-2}, \frac{\bar \gamma_{n-3}}{\bar \gamma_{n-2}^{\iota}},\dots,
\frac{\bar \gamma_0}{\bar \gamma_1^\iota}, \frac{\bar \gamma_{n-1}}{\bar \gamma_0^{\iota}}
\right),
\]
and the statement of the proposition follows.
\end{proof}

Proposition \ref{prop:tCJ_matrix} allows us to derive the joint eigenvalue distribution of the finite truncated circular Jacobi beta ensemble.

\begin{theorem}\label{thm:tCJ_joint}
Fix $\beta>0,\delta\in\C$ with $\Re\delta>-1/2$. Then the eigenvalues of $\trunc{\CJ}_{n+1,\beta,\delta}$ are distributed in $\D^n$ according to the density    \begin{equation}\label{eq:tCJ_joint_density}
c_{n,\beta,\delta}\prod_{j,k=1}^n(1-z_j\bar z_k)^{\frac{\beta}{2}-1}\prod_{j<k}|z_k-z_j|^2\prod_{j=1}^n\left((1-z_j)^{\bar\delta}(1-\bar z_j)^\delta\right)
\end{equation}
with respect to the Lebesgue measure on $\mathbb{D}^n$. Here $c_{n,\beta,\delta}= \frac{1}{\pi^n n!}\prod_{j=1}^n\frac{\Gamma(\frac{\beta}{2}j+1+\delta)\Gamma(\frac{\beta}{2}j+1+\bar\delta)}{\Gamma(\frac{\beta}{2}j)\Gamma(\frac{\beta}{2}j+1+\delta+\bar\delta)} $ is the normalizing constant. 
\end{theorem}
\begin{proof}
    
The proof of the statement relies on the following computations of Jacobian determinants. We refer to Section 6 and Appendix B of \cite{KK} for more details. 

Let $\gamma_k,0\le k\le n-1$ be distributed as in Proposition \ref{prop:tCJ_matrix}, and let $(\alpha_0,\alpha_1,\cdots,\alpha_{n-1})=\cT_n^{-1}(\gamma_0,\gamma_1,\cdots,\gamma_{n-1})$. From \eqref{eq:modifiedV} we have
\begin{equation}\label{eq:Jac_alpha_gamma}
\left|\frac{\partial(\alpha_0,\dots,\alpha_{n-1})}{\partial(\gamma_0,\dots,\gamma_{n-1})}\right|=1. 
\end{equation}
Denote by $z_k,1\le k\le n$ the eigenvalues of $\trunc{\CJ}_{n+1,\beta,\delta}$.
It has been shown in \cite{KK} that 
\begin{equation}\label{eq:Jac_alpha_zero}
\left|\frac{\partial(\alpha_0,\dots,\alpha_{n-1})}{\partial(z_1,\dots,z_n)}\right|=|\Delta(z_1,\dots,z_n)|^2\prod_{j=0}^{n-1}(1-|\alpha_j|^2)^{-j},
\end{equation}
where $\Delta(z_1,\dots,z_n)=\prod_{1\le j<k\le n}(z_j-z_k)$ denotes the Vandermonde determinant of $z_k,1\le k\le n$. 
Using $|\alpha_k|=|\gamma_k|$, \eqref{eq:Jac_alpha_gamma}, and \eqref{eq:Jac_alpha_zero}, we obtain 
\begin{equation}\label{eq:Jac_gamma_zero}
\left|\frac{\partial(\gamma_0,\dots,\gamma_{n-1})}{\partial(z_1,\dots,z_n)}\right|=|\Delta(z_1,\dots,z_n)|^2\prod_{j=0}^{n-1}(1-|\gamma_j|^2)^{-j}.
\end{equation}
Proposition 2.5 of \cite{BNR2009} shows that we have
	\begin{equation}\label{eq:CJ_Jacobi1}
	\prod_{j=0}^{n-1}(1-\gamma_j) = \Phi_n(1)=\prod_{j=1}^n(1-z_j),\quad \prod_{j=0}^{n-1}(1-\bar\gamma_j) = \bar\Phi_n(1)=\prod_{j=1}^n(1-\bar z_j).
	\end{equation}
Since $|\alpha_k|=|\gamma_k|$, by Lemma B.1 (vi) of \cite{KK}, we have 	\begin{equation}\label{eq:CJ_Jacobi2}
	\prod_{j=0}^{n-1}(1-|\gamma_j|^2)^{(\frac{\beta}{2}-1)(j+1)}=\prod_{j=0}^{n-1}(1-|\alpha_j|^2)^{(\frac{\beta}{2}-1)(j+1)}=\prod_{j,k=1}^n(1-z_j\bar z_k)^{\frac{\beta}{2}-1}.
	\end{equation}
By \eqref{eq:Jac_gamma_zero}-\eqref{eq:CJ_Jacobi2} and the explicit joint distribution of $\gamma_k,0\le k\le n-1$, the joint density of $z_k,1\le k\le n$ is given by
\begin{align*}
c_{n,\beta,\delta}\prod_{j=0}^{n-1}&(1-|\gamma_j|^2)^{(\frac{\beta}{2}-1)(j+1)}(1-\gamma_j)^{\bar\delta}(1-\bar\gamma_j)^{\delta}|\Delta(z_1,\dots,z_n)|^2 \\
&=
 c_{n,\beta,\delta}\prod_{j,k=1}^n(1-z_j\bar z_k)^{\frac{\beta}{2}-1}\prod_{j=1}^n\left((1-z_j)^{\bar\delta}(1-\bar z_j)^{\delta}\right) |\Delta(z_1,\dots,z_n)|^2. 
\end{align*}
    Collecting the normalizing constants of the joint distribution of $\gamma_k,0\le k\le n-1$, we get 
	\[
	c_{n,\beta,\delta} = \frac{1}{\pi^n n!}\prod_{j=0}^{n-1}\frac{\Gamma(\frac{\beta}{2}(j+1)+1+\delta)\Gamma(\frac{\beta}{2}(j+1)+1+\bar\delta)}{\Gamma(\frac{\beta}{2}(j+1))\Gamma(\frac{\beta}{2}(j+1)+1+\delta+\bar\delta)}. 
	\]
	This finishes the proof. 
\end{proof}

Note that using the explicit description of the joint distribution of the modified Verblunsky coefficients of the random probability measure $\mu_{n,\beta,\delta}^{\textup{CJ}}$, and the method developed in Section 7 of \cite{KK}, one can also obtain the joint density of the perturbed circular Jacobi beta ensemble. We omit the computation to shorten our representation.

\subsection{Edge limit of the truncated circular Jacobi  beta ensemble}\label{sub:CJ_limit}

In this section, we prove the edge limits of the rank-one truncation and multiplicative perturbation of the circular Jacobi beta ensemble. 

Recall the Dirac operator representation $\CJop_{n,\beta,\delta}$ defined in Definition \ref{def:CJop}. 
   We denote by $\rev \tau_{n,\beta,\delta}$ the reversed version of $\CJop_{n,\beta,\delta}$ via Definition \ref{def:rev_Dirac}, and by $\aff \tau_{n,\beta,\delta}$ to be the pulled-back version of $\rev \tau_{n,\beta,\delta}$ via Definition \ref{def:aff_Dirac}.

Our approach will be similar to the one used for the circular beta ensemble case (which correspond to $\delta=0$). We will show that under appropriate transformations, the operators $\aff \tau_{n,\beta,\delta}$ and $\CJop_{n,\beta,\delta}$ are orthogonally equivalent. Note however, that in the $\delta\neq 0$  the measure $\mu_{n,\beta,\delta}^{\CJb}$ is no longer invariant under rotations, which requires us to develop a new method to prove the orthogonal equivalence. 
The key ingredient is the following proposition, providing equivalent descriptions of the conditioned path parameters of $\CJop_{n,\beta,\delta}$. 

\begin{proposition}\label{prop:CJ_conditioned} 
Let $\mu=\mu_{n,\beta,\delta}^{\textup{CJ}}$ and let $\gamma_k,0\le k\le n-1$ be its modified Verblunsky coefficients. Let $b_k,0\le k\le n$ be the path parameters of $\CJop_{n,\beta,\delta}$ in $\D$. Then the following sequences have the same joint distribution.
\begin{enumerate}[(1)]
    \item The path parameters $b_k,0\le k\le n$ conditioned on $b_n=1$.
    
    \item The path parameters corresponding to the sequence of modified Verblunsky coefficients $\bar \gamma_0^\iota,\bar \gamma_1^\iota,\dots,\bar \gamma_{n-2}^\iota,1$.

    \item The `pulled back' path parameters $b_k':=\cA_{\hat b_{n-1},\D}(\hat b_{n-k-1}), 0\le  k\le n$, where $\hat b_{-1}=1$ and $\hat b_k,0\le k\le n-1$ are the first $n$ elements of the path parameters produced by the sequence of modified Verblunsky coefficients $\bar \gamma_{n-2},\bar \gamma_{n-3},\dots,\bar \gamma_0$.
\end{enumerate}
\end{proposition}

The proof relies on a special decomposition property of the $\Theta(a+1,\delta)$ distribution and an application of Doob's h-transform. Before presenting the proof, we first introduce a Pearson-type distribution, and a couple of facts about the $\Theta(a+1,\delta)$ distribution. 

\begin{definition}\label{def:Pearson}
	For $m>1/2$ and $\mu\in \R$ we denote by $P_{IV}(m,\mu)$ the distribution of the (unscaled) Pearson type IV distribution on $\R$ that has density function 
	\begin{align}
	\frac{2^{2m-2}|\Gamma(m+\frac{\mu}{2} i)|^2}{\pi \, \Gamma(2m-1)} (1+x^2)^{-m}e^{-\mu \arctan x}.
	\end{align}
\end{definition}

Note that the random variable $\Theta(1,\delta)$ can be connected to the Pearson random variable $P_{IV}(\Re\delta+1,-2\Im\delta)$ via the mapping  $e^{i\theta}\mapsto -\cot (\theta/2)$. 

\begin{fact}[\cite{LV}]\label{fact:factor}
	Suppose that $\gamma\sim \Theta(a+1,\delta)$ with $a\ge 0$ and $\Re \delta>-1/2$. Define $w, v\in \R$ with $\tfrac{2\gamma}{1-\gamma}= w-iv$. Then the joint density of $(v,1+w)\in\R\times \R_+$ is given by	\begin{equation}\label{eq:joint}
	f_{a,\delta}(x,y)=c_{a,\delta}\,y^{\frac{a}{2}-1}(x^2+(1+y)^2)^{-(\frac{a}{2}+\Re\delta+1)}e^{2\Im\delta \arctan \frac{x}{1+y}},
	\end{equation}
	with $c_{a,\delta}=2^{a+2\Re\delta}\frac{\Gamma(a/2+1+\delta)\Gamma(a/2+1+\bar\delta)}{\pi\Gamma(a/2)\Gamma(a/2+1+2\Re\delta)}$.
 
	Moreover, the random variables $w$ and $\frac{v}{2+w}$ are independent.  The distribution of the random variable $\tfrac{v}{2+w}$ is given by $P_{IV}(\tfrac{a}{2}+\Re\delta+1,-2 \Im \delta)$. The distribution of $1+w$ is the same as the distribution of $G_1/G_2$, where $G_1, G_2$ are the  independent (standard) Gamma random variables with parameters $\frac{a}{2}$ and $\frac{a}{2}+2\Re\delta+1$, respectively. 
\end{fact}
Using Fact \ref{fact:factor} and the definition of $\gamma^\iota$ in \eqref{eq:gamma_iota_def}, one obtains the following property.

\begin{fact}\label{fact:iota}
    Suppose that $\gamma\sim \Theta(a+1,\delta)$ with $a\ge 0$ and $\Re \delta>-1/2$. 
Then  $\gamma^\iota\sim \Theta(\tl a+1,\tl \delta)$ with $\tl a=a+4\Re\delta+2$, $\tl \delta=-(1+\delta)$.
\end{fact}
\begin{proof}
With a bit of abuse of notation, define $w,v,w^\iota,v^\iota \in\R$ with
\begin{equation}\label{eq:wv_iota}
w-iv=\frac{2\gamma}{1-\gamma},\quad \quad w^\iota-iv^\iota=\frac{2\gamma^\iota}{1-\gamma^\iota}.
\end{equation}
They satisfy the identities  $\cU^{-1}(\gamma)=v+i(1+w)$ and $\cU^{-1}(\gamma^\iota)=v^\iota+i(1+w^\iota)$.
By \eqref{eq:wv_iota} and \eqref{eq:gamma_iota_def} we have 
\[
1+w^{\iota}= \frac{1}{1+w},\quad \quad\frac{v^\iota}{2+w^\iota} = -\frac{v}{2+w}.
\] 
From  Fact \ref{fact:factor}, we obtain that the joint density of $(v^{\iota},1+w^{\iota})$ is proportional to 
\[
y^{\frac{a}{2}+2\Re\delta}(x^2+(1+y)^2)^{-(\frac{a}{2}+\Re\delta+1)}e^{-2\Im\delta \arctan \frac{x}{1+y}}.
\]
This shows that if $\gamma\sim\Theta(a+1,\delta)$, then $\gamma^\iota\sim\Theta(\tl a+1,\tl \delta)$, with $\tl a:=a+4\Re\delta+2$ and $\tilde \delta:=-(1+\delta)$. Note that the random variable $\gamma^\iota$ is well-defined since the conditions $\tl a>0$ and $\tl a/2+\Re \tl\delta=a/2+\Re\delta>-1/2$ are still satisfied. 
\end{proof}

The next result shows that the right boundary point of the $\CJop_{n,\beta,\delta}$ operator has the same distribution as a $\Theta(1,\delta)$ random variable. 
\begin{proposition}\label{prop:hitting}
    Let $b_n$ be the right boundary point of the Dirac operator $\CJop_{n,\beta,\delta}$ in $\D$. Then $b_n\ed \Theta(1,\delta)$. 
\end{proposition}
\begin{proof}
    Fix $n$, and  let $b_k:=b_k^{(n)},0\le k\le n-1$ be the path parameters of $\CJop_{n,\beta,\delta}$ in $\D$, these solve the recursion \eqref{eq:b_rec} with $b_0=0$. We can consider the solution $b_{k,\gamma}, 0\le k\le n$ of this recursion for any initial value $b_{0,\gamma}=\gamma\in \D$. We denote the probability density function of $b_{n,\gamma}$ by $P_n(\gamma,\eta)$.
    
Recall from Proposition \ref{prop:CJ_mVerblunsky} that the  modified Verblunsky coefficients $\gamma_k=\gamma_k^{(n)},0\le k\le n-1$ of $\CJop_{n,\beta,\delta}$ are independent, with $\gamma_k\sim \Theta(\beta(n-k-1)+1,\delta)$. 

    In the case when $n=1$, we get $b_n=b_1=\gamma_0$  from the recursion \eqref{eq:b_rec}, so the statement follows.  We also have 
    \begin{align*}
    P_1(0,\eta) = c_\delta(1-\eta)^{\bar \delta}(1-\bar \eta)^{\delta},\quad c_\delta =  \frac{\Gamma(1+\delta)\Gamma(1+\bar \delta)}{\Gamma(1+\delta+\bar \delta)}.  
    \end{align*}
    Note  that by \eqref{eq:b_rec} the distribution of $b_{k+1}$ given $b_k$ is invariant with respect to isometries of the disk that preserve  $1$. These isometries are just the Poincar\'e disk version of the affine isometries of $\HH$,  which can be parameterized as $\cA_{a,\D}(z)=\frac{z-a}{1-\bar az}\frac{1-\bar a}{1-a}, a\in \D$ according to \eqref{eq:Adef}. Therefore, we have 
    \begin{align*}
    P_1(\gamma,\eta)=c_\delta\left(\frac{(1-|\gamma|^2)(1-\eta)}{(1-\bar\gamma \eta)(1-\gamma)}\right)^{\bar \delta}\left(\frac{(1-|\gamma|^2)(1-\bar\eta)}{(1-\gamma\bar \eta)(1-\bar \gamma)}\right)^{\delta}\frac{1-|\gamma|^2}{|1-\bar\gamma \eta|^2}.
    \end{align*}

    For $n\ge 2$ we will proceed by induction. Assume that $P_{n-1}(\gamma,\eta)=P_1(\gamma,\eta)$ (or equivalently, $P_{n-1}(0,\eta)=P_1(0,\eta)$). The proof will be completed if we can show that $P_n(0,\eta)=P_1(0,\eta)$. Let 
    \begin{align*}
        f_{a,\delta}= c_{a,\delta}(1-|z|^2)^{\frac{a}{2}-1}(1-z)^{\bar\delta}(1-\bar z)^{\delta},\quad a=\beta(n-1)
    \end{align*}
    be the density function of $\gamma_0^{(n)}$ as in Definition \ref{def:Theta_delta}.     
    Note that  $\gamma_0^{(n)}$ is independent of $\gamma_k^{(n)}, 1\le k\le n-1$, which have the same joint distribution as $\gamma_k^{(n-1)}, 0\le k\le n-2$. Hence from \eqref{eq:b_rec} we get 
    \begin{align*}
        P_n(0,\eta) &= \int_{z\in\D} P_{n-1}(z,\eta) f_{a,\delta}(z)dz=\int_{z\in\D} P_{1}(z,\eta) f_{a,\delta}(z)dz\\
        &=       
        c_\delta (1-\eta)^{\bar\delta}(1-\bar\eta)^{\delta}\int_{z\in\D}c_{a,\delta}\frac{(1-|z|^2)^{\frac{a}{2}+\delta+\bar\delta}}{(1-\bar z\eta)^{\delta+1}(1-z\bar \eta)^{\bar\delta+1}} dz.
    \end{align*}
    Using the change of variable $z\mapsto z\bar \eta$ we see that the integral does not depend on $\eta$. Hence $P_n(0,\eta)$ is a constant 
    multiple of $P_1(0,\eta)$ which means that it must be equal to it. This completes the induction step and finishes the proof. 
\end{proof}

Now we turn to the proof of Proposition \ref{prop:CJ_conditioned}.

\begin{proof}[Proof of Proposition \ref{prop:CJ_conditioned}]
We first prove the equivalence between (2) and (3). By the recursion \eqref{eq:bk_cA} and the fact $\cA_{\gamma,\D}^{-1}(0)=\gamma$, we have  
 \begin{align*}
b_k' &= \cA_{\bar \gamma_{0},\D}\circ\cA_{\bar \gamma_{1},\D}\circ\cdots\circ \cA_{\bar  \gamma_{n-2},\D}\big(\cA_{\bar \gamma_{n-2},\D}^{-1}\circ\cdots\circ\cA_{\bar \gamma_{k},\D}^{-1}(0)\big)\\
&=\cA_{\bar \gamma_{0}^{\iota},\D}^{-1}\circ\cdots\circ\cA_{\bar {\gamma}_{k-1}^\iota,\D}^{-1}(0).
 \end{align*}
 This shows that the sequence $b_k',0\le k\le n-1$ are the first $n$ elements of the path produced by $\bar\gamma_0^\iota,\dots,\bar\gamma_{n-2}^{\iota}$. (This fact was also observed in Lemma 50 of \cite{BVBV_palm}.) 
By setting $\tl b_{-1}=1$, we have $b_n'=\cA_{b_{n-1},\D}(1)=1$. Note also, that $b_n'=1$ if and only if the corresponding last modified Verblunsky coefficient is equal to $1$. This proves that the path parameters described in (2) and (3) are equal in law. 

We now prove that the paths described in (1) and (2) have the same distribution. 
Let $z_k=\cU^{-1}(b_k),0\le k\le n$ be the path parameters of $\CJop_{n,\beta,\delta}$ in $\HH$. 
By \eqref{eq:z_rec} this is a (time-inhomogeneous) Markov chain, with transition densities that are invariant under affine transformations. By Proposition \ref{prop:hitting}  it follows that 
$z_n\ed \cU^{-1}(\gamma_{n-1})\sim P_{IV}(\Re\delta+1,-2\Im\delta)$ with density 
\[
    g(q)= \frac{4^{\Re\delta}|\Gamma(1+\delta)|^2}{\pi \, \Gamma(2\Re\delta+1)} (1+q^2)^{-\Re\delta-1}e^{2\Im\delta\arctan q}. 
\]
Moreover, the proof of the same proposition also implies that the conditional distribution of $z_n$ given $z_k=z=x+i y$ for a fixed $0\le k\le n-2$ is the same as that of $y \cU^{-1}(\gamma_{n-1})+x$. This random variable has density 
\[
g_{x,y}(q):=y^{-1}g(\frac{q-x}{y}).
\]
For any fixed $z=x+iy, z'=x'+i y'\in \HH$ we get
\begin{equation}\label{eq:q_limit}
    \lim_{q\to\infty} \frac{g_{x',y'}(q)}{g_{x,y}(q)}=\frac{h(z')}{h(z)}, \qquad h(z):=(\Im z)^{2\Re\delta+1}.
    \end{equation}
Let $f_{a_k,\delta}(x,y)$ be the density function of $\cU^{-1}(\gamma_k)$ defined via \eqref{eq:joint} with parameters $a_k=\beta(n-k-1)$ and $\delta$. 
Then the transition density of the Markov chain $z_k, 0\le k\le n$ is given by 
\[
P\Big((z,k),(z',k+1) \Big)= f_{a_k,\delta}\left(\frac{\Re(z'-z)}{\Im z},\frac{\Im z'}{\Im z}\right), \qquad 0\le k\le n-1. 
\]
Now consider the distribution of $z_k, 0\le k\le n$ conditioned on $z_n=q$ with $q\to \infty$.  By Doob's h-transform,  the conditioned transition density is given by 
\begin{align*}
Q\Big((z,k),(z',k+1)\Big)&=P\Big((z,k),(z',k+1) \Big) \frac{h(z')}{h(z)}\\
&=c_{a_k,\delta}\,v^{\frac{\beta}{2}(n-k-1)+2\Re\delta}(u^2+(1+v^2))^{-\frac{\beta}{2}(n-k-1)-\Re\delta-1}e^{2\Im\delta\arctan (u/(1+v))}
\end{align*}
where $u=(\Re z'-\Re z)/\Im z$ and $v=\Im z'/\Im z$. By Proposition \ref{fact:iota}, $Q((z,k),(z',k+1))$ is exactly the density function of the random variable $\cU^{-1}(\bar\gamma_k^\iota)$. This proves the equivalence between the statements (1) and (2), and hence completes the proof.
\end{proof}

By Proposition \ref{prop:CJ_conditioned} and Fact \ref{fact:iota} we see that the effect of conditioning the path parameters in $\D$ to hit $1$ is equivalent to changing the parameter $\delta\mapsto -(1+\bar \delta)$. This coincides with a similar factorization lemma for the generating path of the $\tau_{\beta,\delta}$ operator, see Theorem 43 of \cite{LV}.

\begin{corollary}\label{cor:CJ_factor}
    Consider the same setup as in Proposition \ref{prop:CJ_conditioned},
    in particular the  path $b_k', 0\le k\le n$ defined in (3).
      Let $\eta\sim\Theta(1,\delta)$ be independent of $\gamma_k,0\le k\le n-2$. Then the rotated path parameters $\check{b}_k:=\eta b_k', 0\le k\le n$ have the same joint distribution as the  path parameters of the $\CJop_{n,\beta,\delta}$ operator.
\end{corollary}

\begin{proof}
    By Proposition \ref{prop:CJ_conditioned}, the path parameters $b_k',0\le k\le n$ can be produced by the sequence of modified Verblunsky coefficients $\bar\gamma_0^\iota,\dots,\bar\gamma_{n-2}^\iota,1$.  Let $(\alpha_0',\dots,\alpha_{n-1}'):=\cT_n^{-1}(\bar\gamma_0^\iota,\dots,\bar\gamma_{n-2}^\iota,1)$ be the corresponding sequence of Verblunsky coefficients. By  \eqref{eq:modifiedV}
 and    \eqref{eq:alpha_ratio} 
    we get
\[
\alpha_k' =(-1)^k \frac{\gamma_k^{\iota}\cdots \gamma_0^{\iota}}{\gamma_{k-1}\cdots \gamma_0}, 0\le k\le n-2, \qquad \alpha_{n-1}'=(-1)^{n-1} \frac{\gamma_{n-2}^{\iota}\cdots \gamma_0^{\iota}}{\gamma_{n-2}\cdots \gamma_0}. 
\]
Recall also the (equivalent) description of $b_k',0\le k\le n$ using $\alpha_k',0\le k\le n-1$ in \eqref{eq:b_alpha}. 
    Observe that 
    for $Z=\diag(z,1)$, we have
\begin{align}\label{eq:Aleks_path}
	\mathcal P Z^{-1}\mat{1}{\bar \alpha_0}{\alpha_0}{1}\cdots \mat{1}{\bar \alpha_{k-1}}{\alpha_{k-1}}{1} Z\bin{0}{1}=\bar z b_k.
	\end{align}
This shows that the Verblunsky coefficients $\check{\alpha}_k,0\le k\le n-1$ corresponding to $\check{b}_k,0\le k\le n$ are given by $\check{\alpha}_k= \bar \eta  \alpha_k'$, in particular 
\[
\check{\alpha}_k =(-1)^k \bar \eta \frac{\gamma_k^{\iota}\cdots \gamma_0^{\iota}}{\gamma_{k-1}\cdots \gamma_0}, 0\le k\le n-2, \qquad \check{\alpha}_{n-1}=(-1)^{n-1} \bar \eta \frac{\gamma_{n-2}^{\iota}\cdots \gamma_0^{\iota}}{\gamma_{n-2}\cdots \gamma_0}. 
\]
(The relation between rotating the sequences of path parameters and the corresponding Verblunsky coefficients is related to the so-called Aleksandrov measure. We refer to \cite{OPUC1} for more background.)

Finally, using Claim \ref{claim:CJ} for $\zeta_0=\eta$, $\zeta_j=\gamma_{j-1}, 1\le j\le n$ we get the following distributional identity:
\[
\left(\gamma_0^{\iota} \bar \eta, \frac{\gamma_1^{\iota}}{\gamma_0}, \dots, \frac{\gamma_{n-2}^{\iota}}{\gamma_{n-3}},\frac{1}{\gamma_{n-2}}\right)\ed
\left(
\bar\gamma_0, \frac{\bar \gamma_1}{\bar \gamma_0^{\iota}}, \dots, \frac{\bar \gamma_{n-2}}{\bar \gamma_{n-3}^{\iota}},\frac{\bar \eta}{\bar \gamma_{n-2}^{\iota}}
\right).
\]
This means that $\check\alpha_k, 0\le k\le n-1$ have the same joint distribution as the random variables
\[
\dot{\alpha}_k=(-1)^k \frac{\bar \gamma_k\cdots \bar \gamma_0}{\bar \gamma_{k-1}^{\iota}\cdots \bar \gamma_0^{\iota}}, 0\le k\le n-2, \qquad \dot{\alpha}_{n-1}=(-1)^{n-1} \frac{\bar \eta \bar \gamma_{n-2}\cdots \bar \gamma_0}{\bar \gamma_{n-2}^{\iota}\cdots \bar \gamma_0^{\iota}}.
\]
Since the joint distribution of $\gamma_0,\dots, \gamma_{n-2},\eta$ is the same as the joint distribution of $\gamma_0, \dots, \gamma_{n-1}$,  \eqref{eq:modifiedV}
 and    \eqref{eq:alpha_ratio} now shows that 
 \[
 (\check{\alpha}_0,\check{\alpha}_1,\dots, \check{\alpha}_{n-1})\ed  (\dot{\alpha}_0, \dots, \dot{\alpha}_{n-1})\ed \cT_n^{-1}(\gamma_0,\dots, \gamma_{n-1}),
 \]
which implies that $\check{b}_k, 0\le k\le n$ have the same joint distribution as the path parameters $b_k, 0\le k\le n$ for $\CJop_{n,\beta,\delta}$.
\end{proof}

Recall the path $\aff b_k, 0\le k\le n$, the pulled back version of the reversed path $\rev b_k, 0\le k\le n$ corresponding to the random measure $\mu_{n,\beta,\delta}^{\CJb}$. By Proposition \ref{prop:tCJ_matrix} $\rev b_k, 0\le k\le n-1$ has the same distribution as the path built from the modified Verblunsky coefficients $\gamma_{n-2}, \dots, \gamma_{0}$. This  path is the complex conjugate of the path $\hat b_k, 0\le k\le n-1$ in (3) of Proposition \ref{prop:CJ_conditioned}. This also means that 
$\aff b_k, 0\le k\le n$ is just the time-reversed and complex conjugated version of the path $b_k', 0\le k\le n$ in (3) of Proposition \ref{prop:CJ_conditioned}. Together with Corollary \ref{cor:CJ_factor} this implies that applying an independent random rotation, a complex conjugation, and a time reversal to $\aff b_k, 0\le k\le n$ produces a path that has the same distribution as the driving path of the operator $\CJop_{n,\beta,\delta}$. 
This statement allows us to show that $\aff{\tau}_{n,\beta,\delta}$ and $\CJop_{n,\beta,\delta}$ are orthogonally equivalent.

\begin{proposition}\label{prop:CJ_equiv}
Let $\aff{b}_k,0\le k\le n$ be the path parameters defined in Corollary \ref{cor:CJ_factor}, 
$\aff{\tau}_{n,\beta,\delta}$ the corresponding Dirac-type operator, and let 
\[
Q=\frac{1}{\sqrt{1+q^2}}\begin{pmatrix} q & 1\\ -1 & q\end{pmatrix},\qquad q=\cU^{-1}(\aff{b}_n).
\]
Then the operator
$\rho^{-1}(SQ)\aff{\tau}_{n,\beta,\delta}(SQ)^{-1}\rho$ has the same distribution as $\CJop_{n,\beta,\delta}$. 
\end{proposition}
\begin{proof}
Recall the transformation $\mathcal{Q}$ defined in Section \ref{sub:Dirac_secular}. Observe that in the unit-disk model, the transformation $\mathcal{Q}$ behaves exactly as a rotation, i.e.~for $z\in\D$ we have $\mathcal{Q}(z)=\mathcal{U}\circ\mathcal{Q}\circ\mathcal{U}^{-1}(z)=\bar \eta z$, where we denote $\eta:=\aff b_n$.

Together with Propositions \ref{prop:tCJ_matrix} and \ref{prop:CJ_conditioned}, this observation shows that the path parameters of $\rho^{-1}(SQ)\aff \tau_{n,\beta,\delta}(SQ)^{-1}\rho$ have the same joint distribution as $\check{b}_k=\eta b_k',0\le k\le n$ where $b_k',0\le k\le n$ are distributed as in Proposition \ref{prop:CJ_conditioned}.  
Moreover, by the same argument one can check that the left boundary point of $\rho^{-1}(SQ)\aff\tau_{n,\beta,\delta}(SQ)^{-1}\rho$ is equal to $1$ as desired. The proof is now completed by Corollary \ref{cor:CJ_factor}.
\end{proof}

Recall the limiting operator $\HPop$ defined in Definition \ref{def:HPop}, and its reversed and transformed version $\tau_{\beta,\delta}$ defined in Definition \ref{def:HPop_rev}. By Propositions \ref{prop:hitting} and  \ref{prop:CJ_equiv} the 
transformation connecting  $\HPop$ and $\tau_{\beta,\delta}$ has the same distribution as the one connecting $\CJop_{n,\beta,\delta}$ and $\aff \tau_{n,\beta,\delta}$.

The operator level convergence of the $\CJop_{n,\beta,\delta}$ to the limiting $\HPop$ operator was proved in \cite{LV} and can be summarized as follows. This is also the final ingredient to prove the convergence of the normalized characteristic polynomial of the truncated circular Jacobi beta ensemble.

\begin{proposition}[\cite{LV}]\label{prop:HP_limit}
Fix $\beta>0$ and $\Re\delta>-1/2$. There exists a coupling of the operators $\CJop_{n,\beta,\delta}=\Dirop(z_n(\cdot),\uu_0,\uu_1^{(n)}),n\ge 1$ with $\uu_0=\binom{1}{0},\uu_1^{(n)}=\binom{-z_n(1)}{-1}$, and $\HPop=\Dirop(z(\cdot),\uu_0,\uu_1)$ with $\uu_1=\binom{-z(1)}{-1}$ such that as $n\to\infty$ we have almost surely $z_n\to z$ pointwise on $[0,1)$, and 
\[
\|\res \CJop_{n,\beta,\delta}-\res \HPop\|_{\textup{HS}} \to 0,\quad  \quad \ttr_{\CJop_{n,\beta,\delta}}-\ttr_{\HPop}\to 0.
\]
Set $c_\delta=\frac{4}{\beta}(\Re\delta+\frac12)>0$, and for $\eps>0$ small define $c_1=c_\delta-\eps, c_2=c_\delta+\eps$.
Then there exists a sequence of tight random variables $\kappa_n,n\ge 1$ and an a.s.~finite random variable $\kappa>0$ such that for $0\le t<1$
\begin{align}\label{eq:CJ_path}
\kappa_n^{-1}\left(1-\frac{\lfloor nt\rfloor}{n}\right)^{c_2}\le y_n(t)\le \kappa_n\left(1-\frac{\lfloor nt\rfloor}{n}\right)^{c_1},\quad  |x_n(1)-x_n(t)|\le \kappa_n\left(1-\frac{\lfloor nt\rfloor}{n}\right)^{c_1}
\end{align}
and similarly, 
\begin{equation}\label{eq:HP_path}
\kappa^{-1}(1-t)^{c_2}\le y(t)\le \kappa(1-t)^{c_1},\quad  |x(1)-x(t)|\le \kappa(1-t)^{c_1}.
\end{equation}
\end{proposition}

We now have all the ingredients to prove the edge scaling limit of the truncated circular Jacobi beta ensemble.

\begin{theorem}\label{thm:CJ_main}
    For fixed $n\ge 1$, $\beta>0$ and $\Re\delta>-1/2$, let $\lambda_i,1\le i\le n-1$ be size $n$ truncated circular Jacobi beta ensemble nd set 
$p_{n-1,\beta,\delta}(z)=\prod_{i=1}^{n-1} \frac{z-\lambda_i}{1-\lambda_i}$ be the normalized characteristic polynomial. Let $\cE_{\beta,\delta}$ be the structure function of $\tau_{\beta,\delta}$ defined via \eqref{def:structure}.
Then there is a coupling of $p_{n,\beta,\delta},n\ge 1$ and $\cE_{\beta,\delta}$ such that
	\begin{equation*}
	|p_{n-1,\beta}(e^{iz/n})e^{-iz/2}-\cE_{\beta,\delta}(z)|\to 0 \quad \text{ almost surely, uniformly on compacts as $n\to\infty$}
	\end{equation*}
	Consequently, under the edge scaling \eqref{eq:edgescaling}, the truncated circular Jacobi beta ensembles converge weakly to the zeros of the random analytic function $\cE_{\beta,\delta}(\cdot)$.
\end{theorem}

\begin{proof}
Let $\rev \mu_{n}$ be the reversed version of the measure $\mu_{n}:=\mu_{n,\beta,\delta}^{\textup{CJ}}$, and $\aff\tau_{n,\beta,\delta}$ be the pulled-back operator. By Propositions \ref{prop:CMV1} and \ref{prop:tCJ_matrix},  $p_{n-1,\beta,\delta}$ has the same distribution as the monic orthogonal polynomial of degree $n-1$ associated to $\rev \mu_{n}$, denoted as $\rev\Psi_{n-1,\beta,\delta}$. By Proposition \ref{prop:general1}, we have 
\[
\rev\Psi_{n-1,\beta,\delta}(e^{iz/n})=e^{iz(n-1)/(2n)}\aff H_{n}((n-1)/n,z)^{\dag}\binom{1}{-i},
\]
where $\aff H_{n}$ solves the ODE \eqref{eq:canonical} of $\aff \tau_{n,\beta,\delta}$. Recall $\cE_{\beta,\delta}=H_{\beta,\delta}(1,\cdot)^{\dag}\cdot\binom{1}{-i}$, where $H_{\beta,\delta}(t,z)$ solves \eqref{eq:canonical} of $\tau_{\beta,\delta}$. It is enough to show the uniform-on-compacts convergence of $\aff H_{n}(1,z)$ to $H_{\beta,\delta}(1,z)$. 

By Propositions \ref{prop:HP_equiv} and \ref{prop:CJ_equiv}, one can obtain the operators $\aff \tau_{n,\beta,\delta},n\ge 1$ and $\tau_{\beta,\delta}$  from $\CJop_{n,\beta,\delta},n\ge 1$ and $\HPop$ under the same orthogonal transformations.
Therefore, applying the coupling of the operators $\CJop_{n,\beta,\delta}$ and $\HPop$ in Proposition \ref{prop:HP_limit} gives
\[
\|\res \tau_{n,\beta,\delta}-\res\tau_{\beta,\delta}\|\to 0,\quad \quad |\ttr_{\tau_{n,\beta,\delta}}-\ttr_{\tau_{\beta,\delta}}|\to 0.
\]
In order to apply Proposition \ref{prop:general2}, we need to show that
\begin{equation}\label{eq:HP_conv1}
    \int_0^1|\mathfrak{a}_{0,n}(s)-\mathfrak{a}_0(s)|^2 ds=\int_0^1\left|\frac{1}{\sqrt{y_n(s)}}-\frac{1}{\sqrt{y(s)}}\right|^2\to 0, 
    \end{equation}
where $y_n=\Im z_n$ and $y=\Im z$ are the imaginary part of the generating paths of $\tau_{n,\beta,\delta}$ and $\tau_{\beta,\delta}$, respectively. 
By a standard subsequence argument, the path bound \eqref{eq:CJ_path} allows one to choose a subsequence so that $\kappa_n$ converges to an a.s.~finite random variable. The convergence \eqref{eq:HP_conv1} now follows from the point-wise convergence of $z_n\to z$, the path bounds \eqref{eq:CJ_path}, \eqref{eq:HP_path} (under the time reversal), and an application of the triangle inequality. This verifies \eqref{eq:discrete_assumption} and completes the proof. 
\end{proof}

Note that by Proposition \ref{prop:HP_equiv}, the structure function $\cE_{\beta,\delta}$ can also be characterized as $\cE_{\beta,\delta}=\cH_{\beta,\delta}(0,z)^{\dag}\binom{1}{-i}$, where $\cH_{\beta,\delta}(u,z)$ solves the SDE \eqref{eq:HP_cH}.

Note that in the coupling of Proposition \ref{prop:HP_limit}, the path bounds \eqref{eq:CJ_path}, \eqref{eq:HP_path} and the point-wise convergence of the generating paths $z_n\to z$ imply that $\uu_1^{(n)}\to \uu_1$ a.s.~(see also Proposition \ref{prop:hitting}). Then using Theorem \ref{thm:CJ_main} and Proposition \ref{prop:general3} one obtains the following result on the (edge) scaling limit of the multiplicative perturbed circular Jacobi beta ensemble.
\begin{corollary}\label{cor:CJ_perturb}
For fixed $r\in[0,1]$ let $\Lambda_{n}^{[r]}=\{\lambda_1,\dots,\lambda_n\}$ be the set of eigenvalues of the perturbed matrix $\rCJ_{n,\beta,\delta}$. Set
$
p_{n,\beta,\delta}^{[r]}(z) = \prod_{i=1}^{n} \frac{z-\lambda_i}{1-\lambda_i}$ 
to be the normalized characteristic polynomial, and let $H_{\beta,\delta}$ be the solution to the ODE \eqref{eq:canonical} of $\tau_{\beta,\delta}$.  Let $q\sim\Theta(1,\delta)$ be independent of $H_{\beta,\delta}$. Then under the coupling of Proposition \ref{prop:HP_limit}, we have almost surely  as $n\to\infty$
\begin{align*}
\left|p_{n,\beta,\delta}^{[r]}(e^{iz/n})e^{-iz/(2n)}-\cE_{\beta,\delta}^{[r]}(z)\right|\to 0,\quad \text{ uniformly on compacts in $\C$,}
\end{align*}
where $\cE_{\beta,\delta}^{[r]}(z)=H_{\beta,\delta}(1,z)\cdot\binom{1}{-c_r}$ with $c_r=\frac{q+i \frac{1-r}{1+r}}{1-i q \frac{1-r}{1+r}}$. In particular, this implies the weak convergence of $\Lambda_{n}^{[r]}$ under the edge scaling \eqref{eq:edgescaling} to the zero set of the random analytic function $\cE_{\beta,\delta}^{[r]}$ as $n\to\infty$. 
\end{corollary}
Similar to the statement of Theorem \ref{thm:trunc_circ}, we have $c_r=q$ when $r=1$ and $c_r=i$ when $r=0$, hence this result shows the connection between the scaling limits of the truncated and the unperturbed circular Jacobi beta ensemble.

\section{Appendix}\label{sec:app}

\subsection{Overview of some results  for $\beta=2$}\label{app:det}

Let $\Lambda$ be a locally compact Polish space (in this section it will always be $\C$). 
A simple point process on $\Lambda$ is called \emph{determinantal} with respect to a reference measure $\mu$ with kernel function $K:\Lambda\times \Lambda \to \R$ if for any $k\ge 1$ the $k$th joint intensity function of the process with respect to $\mu$ is given by 
\begin{align}\label{eq:det}
    \rho_{k}(x_1,\dots,x_k)=\det (K(x_i,x_j))_{1\le i,j\le k}.
\end{align}
(See \cite{HKPV} for more on determinantal point processes.)
Determinantal processes appear naturally from  joint probability densities containing the square of the Vandermonde. In this section we will always assume that all absolute moments of $\mu$ are finite.

\begin{proposition}[\cite{mehta,Konig}]\label{prop:det}
    Suppose that the $X_1,X_2,\dots,X_n$ are  complex valued  random variables with joint density given by
    \begin{align}\label{eq:vdm}
       \frac{1}{Z_{n,\mu}} \prod_{1\le i<j\le n} |z_i-z_j|^2
    \end{align}
with respect to a product  measure $\mu^{\otimes n}$ (on $\R^n$ of $\C^n$). (Here $Z_{n,\mu}$ is a finite constant.) Let $\varphi_k(z)$ be the degree $k$  orthonormal polynomials with respect to $\mu$. 
Then $\sum_{k=1}^n \delta_{X_k}$ is a determinantal point process with respect to $\mu$ with kernel given by
\begin{align}
    K(z,w)=\sum_{k=0}^{n-1} \varphi_k(z) \bar \varphi_k(w).
\end{align}   
\end{proposition}
Proposition \ref{prop:det} together with \eqref{eq:cue} and \eqref{eq:tcue} immediately imply that both the circular unitary ensemble and the truncated circular unitary ensemble are determinantal. 
The size $n$ circular unitary ensemble has kernel 
 function 
\begin{align}
    K_{\Circ_{2,n}}(z,w)=\sum_{k=0}^{n-1} z^k \bar w^k,
\end{align}
with respect to the uniform measure on the unit circle. The size $n$ truncated circular unitary ensemble has a similar kernel function 
\begin{align}\label{eq:det_tcue}
    K_{{\Circ}_{2,n}^{\ulcorner}}(z,w)=\sum_{k=0}^{n-1} (k+1) z^k \bar w^k,
\end{align}
with respect to the uniform measure on the unit disk. Note that \cite{ZS} also treats more general truncations of Haar unitary matrices. They showed that if we delete the first $m$ rows and columns of a size $n+m$  Haar unitary matrix then the eigenvalues of the resulting submatrix have joint eigenvalue density given by 
\begin{equation}\label{eq:tcuek}
\frac{1}{Z_{n,m}}\prod_{1\le j<k\le n}|z_j-z_k|^2 \prod_{k=1}^n (1-|z_k|^2)^{m-1} , \qquad z_j\in \D,
\end{equation}
with respect to the Lebesgue measure on the unit disk. By Proposition \ref{prop:det} this is also a determinantal point process, with respect to the measure with density $(1-|z|^2)$ on the unit disk.

Determinantal processes have a number of  nice analytic features. The following proposition shows that if we understand the scaling limit of the kernels of a sequence of determinantal processes then we can derive the scaling limit and the limit is also determinantal. 

\begin{proposition}[\cite{Soshnikov,ShiraiTakahashi}]
    Suppose that $\mathcal{X}, \mathcal{X}_1, \mathcal{X}_2, \dots$ are determinantal processes on $\C$ with respect to the common reference measure $\mu$, with kernel functions $K, K_1, K_2, \dots$. Assume that 
    \begin{align}
        K_n(z,w)\to K(z,w)
    \end{align}
uniformly on compacts in $\C^2$. Then $\mathcal{X}_n$ converges in distribution to $\mathcal{X}$.  
\end{proposition}
Using some simple transformations one can rewrite the determinantal kernel of the circular unitary ensemble as 
\[
   \tilde K_{\Circ_{2,n}}(e^{i t},e^{i s})=D_n(t-s), \qquad D_n(u)=\frac{\sin(nu/2)}{\sin(u/2)}. 
\]
By taking the limit of this kernel as $n\to \infty$ we get the celebrated result of 
Gaudin, Mehta, Dyson regarding the 
the scaling limit of the circular ensemble. 
\begin{theorem}[\cite{AGZ, mehta}]
Let $\Lambda_n$ be the angles in the size $n$ circular unitary ensemble parametrized in $(-\pi,\pi]$. Then $n\Lambda_n\Rightarrow \Lambda$ where $\Lambda$ is a determinantal point process on $\R$ with kernel 
\[
K_{\textup{Sine}_2}(s,t)=\frac{ \sin((s-t)/2)}{\pi(s-t)}
\]
with respect to the Lebesgue measure.
\end{theorem}
Taking the limit of the kernel \eqref{eq:det_tcue} (without any additional scaling) we obtain the point process studied by Peres and Vir\'ag in \cite{PV}. 
\begin{theorem}[\cite{PV}]\label{thm:PeresVirag}
Let $\Lambda_n$ be the eigenvalues of $\trunc{\Circ_{n+1}}$. Then $\Lambda_n$ converges to a determinantal point process on the unit disk with kernel
\begin{align}\label{eq:Bergman_kernel}
    K_{\textup{Bergman}}(z,w)=\frac{1}{(1-z \bar w)^2}
\end{align}
with respect to the uniform measure on the unit disk. The resulting point process has the same distribution as the zero set of the Gaussian analytic function 
\begin{align}\label{eq:GAF}
    f_{GAF}(z)=\sum_{n=0}^\infty \xi_n z^n,
\end{align}
where $\xi_n, n\ge 0$ are i.i.d.~standard complex normals. 
\end{theorem}
Note that \cite{Krishnapur} provides a generalization of this result by connecting the limit of the eigenvalues of rank $m$ truncation of Haar unitary matrices with the singular points of a matrix valued Gaussian analytic function that generalizes \eqref{eq:GAF}.

The circular Jacobi beta ensemble \eqref{eq:CJ_pdf} for $\beta=2$ is also determinantal, this is also called the Hua-Pickrell distribution. 
The kernel function can be expressed in terms of the orthogonal polynomials with respect to the probability measure $\mu_{\delta}$ that has probability density function proportional to 
\[
(1-\bar z)^{ \delta}(1-z)^{ \bar\delta} 
\]
on the unit circle. The Hua-Pickrell distribution can be realized as the joint eigenvalue distribution of the random unitary matrices that have density proportional to $|\det(1-U)^\delta|^2$ with respect to the Haar measure. 

In \cite{LinQiuWang} the authors studied the truncations of these random matrices. They showed that the eigenvalues of the  rank-$m$ truncated matrices form a determinantal point process on the unit disk, and derived their kernel function. Moreover, they showed that the scaling limit of the joint eigenvalue distribution (without any additional scaling) leads to the same limits as in the Haar unitary case (as described in \cite{PV} and \cite{Krishnapur}). 

If we are interested in the  (hard) edge scaling limit of the eigenvalues of the truncated matrix $\trunc{\Circ_{2,n+1}}$ then one needs to transform the kernel \eqref{eq:det_tcue} according to \eqref{eq:edgescaling}, and take the limit. This was considered in \cite{ABKN} where the following result was shown. 
\begin{theorem}[\cite{ABKN}]
   Let $\Lambda_n$ be the eigenvalues of  $\trunc{\Circ_{2,n+1}}$. Then the sequence $-n i \log \Lambda_n$ converges to a determinantal point process $\mathcal{X}_2$ supported on the upper half plane $\HH$ with kernel function
   \begin{align}
      \label{eq:edge_det_kernel}
K_{\textup{edge}}(z,w)=f(z-\bar w),\quad f(u)=\frac{1}{\pi}\int_0^1te^{itu}dt,\quad z,w\in\HH,
\end{align}
with respect to the Lebesgue measure on $\HH$.
\end{theorem}
In fact, this is a special case of a more general problem that \cite{ABKN} considers: the product of independent copies of rank-$m$ truncations of Haar unitary matrices. The point process $\mathcal{X}_2$ also appears in \cite{FyodorovKhoruzhenko}  as the scaling limit of the rank-one additive anti-Hermitian perturbation for the Gaussian unitary ensemble. 

Note that our Theorem \ref{thm:main} for $\beta=2$ provides a new characterization for the determinantal point process $\mathcal{X}_2$. It would be interesting to see whether the determinantal structure could be proved directly from that result.

\subsection{Open problems}
\label{app:open}

We end  with a couple of open problems. \medskip

\noindent \textbf{Problem 1.} Our results for general $\beta>0$ consider the models with the edge scaling \eqref{eq:edgescaling}. It would be interesting to explore the limiting behavior under the bulk scaling (i.e.~when we do not scale at all). In other words, it would be interesting to see whether one could extend the results of \cite{PV} (see Theorem \ref{thm:PeresVirag}) and  \cite{LinQiuWang} for general $\beta$. \medskip

 \noindent \textbf{Problem 2.} In the $\beta=2$ case Theorem  \ref{thm:PeresVirag} provides a description of the bulk scaling limit of the truncated circular unitary ensemble as the zero set of a Gaussian analytic function. It would be interesting to see if one could connect this result to  some sort of a scaling limit of the characteristic polynomial of the truncated circular ensemble under the bulk scaling. \smallskip

\noindent \textbf{Problem 3.} 
A simple calculation shows that under the scaling $z\mapsto c^{-1} z, c\to \infty$ the kernel $\eqref{eq:edge_det_kernel}$ converges to a transformed version of the kernel  \eqref{eq:Bergman_kernel}, where the transformation is the Cayley transform mapping the unit disk $\D$ to the upper half plane $\HH$.  This implies that as $c\to \infty$ the scaled edge limit process  $c^{-1} \mathcal{X}_2$ converges to the image of the bulk limit process of Peres-Vir\'ag under the Cayley transform. Note that similar `edge-to-bulk' limits are known for other random matrix ensembles, see e.g.~\cite{BVBV} for a similar transition involving the point process limits of the Gaussian beta ensemble. It would be interesting to see if a similar limit exists for $\mathcal X_\beta$ for general $\beta>0$. Presumably, this could also provide a way to prove the `edge-to-bulk' transition in the $\beta=2$ case without using the determinantal process framework. \bigskip


  {{  
 		\bigskip
 		\footnotesize

 		\noindent Yun Li, \textsc{Yau Mathematical Sciences Center, Tsinghua University, Beijing 100084, China.}  \texttt{liyun1723@mail.tsinghua.edu.cn}
 		
 		\medskip
 		
 		\noindent Benedek Valk\'o, \textsc{Department of Mathematics, University of Wisconsin -- Madison,
 			Madison, WI 53706, USA.} \texttt{valko@math.wisc.edu}
 		
 }}
\end{document}